\documentclass[12pt,etds]{amsart}
\usepackage{amsmath}
\usepackage{amssymb}
\usepackage{amsfonts}
\usepackage{amsthm}
\usepackage{enumerate}
\usepackage{tikz}
\usepackage{url}
\usetikzlibrary{matrix}

\usepackage{pb-diagram}

\newtheorem{theorem}{Theorem}[section]
\newtheorem{lemma}[theorem]{Lemma}
\newtheorem{prop}[theorem]{Proposition}
\newtheorem{cor}[theorem]{Corollary}
\theoremstyle{remark}
\newtheorem{definition}[theorem]{Definition}
\newtheorem{example}[theorem]{Example}

\newtheorem{remark}[theorem]{Remark}

\def\N{{\mathbb N}}

\def\Q{{\mathbb Q}}
\def\C{{\mathbb C}}
\def\F{{\mathbb F}}

\def\Zp{{{\mathbb Z}_p}}
\def\Zq{{{\mathbb Z}_q}}

\def\Z{{\mathbb Z}}

\def\A{{\mathcal{A}}}

\def\U{{\mathcal{U}}}
\def\K{{\mathcal{K}}}
\def\H{{\mathcal{H}}}
\def\T{{\mathcal{T}}}
\def\I{{\mathcal{I}}}

\def\L{{\mathcal{L}}}

\def\M{{\mathcal{M}}}
\def\c{{\bf c}}
\newcommand{\clsp}{\overline{\operatorname{span}}}
\newcommand{\lsp}{\operatorname{span}}
\newcommand{\rt}{\operatorname{rt}}
\newcommand{\lt}{\operatorname{lt}}
\newcommand{\id}{\operatorname{id}}
\newcommand{\bicov}{\operatorname{bicov}}

\newcommand{\piso}{\operatorname{piso}}
\newcommand{\iso}{\operatorname{iso}}

\newcommand{\End}{\operatorname{End}}

\newcommand{\order}{{o}}
\newcommand{\whitesquare}{\hfill $\whitesquare$\newline\vspace{0.4cm}}

\def\newspan{\operatorname{span}}

\numberwithin{equation}{section}

\begin{document}

\title[Partial-isometric crossed products]
{Partial-isometric crossed products of dynamical systems by left LCM semigroups}

\author[Saeid Zahmatkesh]{Saeid Zahmatkesh}
\address{Department of Mathematics, Faculty of Science, King Mongkut's University of Technology Thonburi, Bangkok 10140, THAILAND}
\email{saeid.zk09@gmail.com, saeid.kom@kmutt.ac.th}



\maketitle
\begin{center}
\emph{Dedicated to Sriwulan Adji}
\end{center}

\begin{abstract}
Let P be a left LCM semigroup, and $\alpha$ an action of $P$ by endomorphisms of a $C^*$-algebra $A$. We study a semigroup
crossed product $C^*$-algebra in which the action $\alpha$ is implemented by partial isometries. This crossed product gives a model
for the Nica-Teoplitz algebras of product systems of Hilbert bimodules (associated with semigroup dynamical systems) studied
first by Fowler, for which we provide a structure theorem as it behaves well under short exact sequences and tensor products.
\end{abstract}

\section{Introduction}
\label{intro}
Let $P$ be a unital semigroup whose unit element is denoted by $e$. Suppose that $(A,P,\alpha)$ is a dynamical system consisting
of a $C^*$-algebra $A$, and an action $\alpha:P\rightarrow \End (A)$ of $P$ by endomorphisms of $A$ such that
$\alpha_{e}=\id_{A}$. Note that, since the $C^*$-algebra $A$ is not necessarily unital, we need to assume that each endomorphism $\alpha_{x}$ is extendible, which means that it extends to a strictly continuous endomorphism $\overline{\alpha}_{x}$ of the multiplier algebra $\M(A)$.
Recall that an endomorphism $\alpha$ of $A$ is extendible if and only if there exists an approximate identity $\{a_{\lambda}\}$ in $A$ and a projection $p\in \M(A)$ such that $\alpha(a_{\lambda})$ converges strictly to $p$ in $\M(A)$. However, the extendibility of $\alpha$ does not necessarily imply $\overline{\alpha}(1_{\M(A)})=1_{\M(A)}$.

There have been huge efforts on the study of $C^*$-algebras associated with semigroups and semigroup dynamical systems.
In the line of those efforts in this regard, Fowler
in \cite{Fowler}, for the dynamical system $(A,P,\alpha)$, where $P$ is the positive cone of a group $G$ such that $(G,P)$
is quasi-lattice ordered in the sense of Nica \cite{Nica}, defined a covariant representation called the
\emph{Nica-Toeplitz covariant representation} of the system, such that the endomorphisms $\alpha_{x}$ are implemented by partial isometries.
He then showed that there exists a universal $C^*$-algebra $\T_{\textrm{cov}}(X)$ associated with the system
generated by a universal Nica-Toeplitz covariant representation of the system such that there is a bijection between the Nica-Toeplitz
covariant representations of the system and the nondegenerate representations of $\T_{\textrm{cov}}(X)$.
To be more precise, $X$ is actually the product system of Hilbert bimodules associated with the system $(A,P,\alpha)$ introduced by him, and the algebra $\T_{\textrm{cov}}(X)$ is universal for Toeplitz representations of $X$ satisfying a covariance condition called
\emph{Nica covariance}. He called this universal algebra the \emph{Nica-Toeplitz crossed product} (or \emph{Nica-Toeplitz algebra})
of the system $(A,P,\alpha)$ and denoted it by $\T_{\textrm{cov}}(A\times_{\alpha} P)$.
When the group $G$ is totally ordered and abelian (with the positive cone $G^{+}=P$), the Nica covariance condition
holds automatically, and the Toeplitz algebra $\T(X)$ is the \emph{partial-isometric crossed product}
$A\times_{\alpha}^{\piso} P$ of the system $(A,P,\alpha)$ introduced and studied by the authors of \cite{LR}.
In other word, the semigroup crossed product $A\times_{\alpha}^{\piso} P$ actually gives a model for the Teoplitz algebras $\T(X)$ of
product systems $X$ of Hilbert bimodules associated with the systems $(A,P,\alpha)$, where $P$ is the positive cone of a totally ordered
abelian group $G$. Further studies on the structure of the crossed product $A\times_{\alpha}^{\piso} P$
have been done progressively in \cite{Adji-Abbas}, \cite{AZ}, \cite{AZ2}, \cite{LZ}, and \cite{SZ} since then.

In the very recent years, mathematicians in \cite{NB, Bar1, Bar2}, following the idea of Fowler,
have extended and studied the notion of the Nica-Toeplitz algebra of a product system $X$ over more general semigroups $P$, namely,
right LCM semigroups (see also \cite{Flet}). These are the semigroups that appear as a natural generalization of the well-known
notion of quasi-latticed ordered groups introduced first by Nica in \cite{Nica}.
Recall that the notation $\mathcal{N}\T(X)$ is used for the Nica-Toeplitz algebra of $X$ in \cite{NB, Bar1, Bar2}, which are
the works that brought this question to our attention that whether we could define a partial-isometric crossed product
corresponding to the system $(A,P,\alpha)$, where the semigroup $P$ goes beyond the positive cones of totally ordered abelian
groups. Although, based on the work of Fowler in \cite{Fowler} (see also the effort in \cite{Adji-ICM} in this direction),
we were already aware that the answer to this question
must be ``yes" for the positive cones $P$ of quasi-latticed ordered groups $(G,P)$, \cite{NB, Bar1, Bar2} made us very
enthusiastic to seek even more than that. Hence, the initial investigations in the present work indicated that the semigroup
$P$ must be left LCM (see \S\ref{sec:pre}). More precisely, in the dynamical system $(A,P,\alpha)$, we considered the semigroup $P$
to be left LCM (so, the opposite semigroup $P^{\textrm{o}}$ becomes right LCM). Then, following \cite{Fowler}, we defined a covariant
representation of the system satisfying a covariance condition called the \emph{(right) Nica covariance}, in which the endomorphisms
$\alpha_{x}$ are implemented by partial isometries. We called this representation the \emph{covariant partial-isometric
representation} of the system. More importantly, we showed that every system $(A,P,\alpha)$ admits a
nontrivial covariant partial-isometric representation. Next, we proved that the Nica-Toeplitz algebra $\mathcal{N}\T(X)$ of
the product system $X$ associated with the dynamical system $(A,P,\alpha)$ is generated by a covariant partial-isometric
representation of the system which is universal for covariant partial-isometric representations of the system. We called
this universal algebra the \emph{partial-isometric crossed product} of the system $(A,P,\alpha)$ and denoted it by
$A\times_{\alpha}^{\piso} P$ (following \cite{LR}), which is unique up to isomorphism. We then studied the behavior of
crossed product $A\times_{\alpha}^{\piso} P$ under short exact sequences and tensor products, from which, a structure theorem followed.
In addition, as an example, when $P$ and $P^{\textrm{o}}$ are both left LCM semigroups we studied the distinguished system
$(B_{P},P,\tau)$, where $B_{P}$ is the $C^*$-subalgebra of $\ell^{\infty}(P)$ generated by the characteristic functions
$\{1_{y}: y\in P\}$, and the action $\tau$ on $B_{P}$ is induced by the shift on $\ell^{\infty}(P)$. It was shown that
the algebra $B_{P}\times_{\tau}^{\piso} P$ is universal for \emph{bicovariant partial-isometric representations} of $P$, which
are the partial-isometric representations of $P$ satisfying both right and left Nica covariance conditions.

Here, prior to talking about the organization of the present work, we would like to mention that, by \cite{SZ2},
if $P$ is the positive cone of an abelian lattice-ordered group $G$, then the Nica-Toeplitz algebra $\T_{\textrm{cov}}(A\times_{\alpha} P)$
of the system $(A,P,\alpha)$ is a full corner in a classical crossed product by the group $G$. Thus, by the present work, since
$A\times_{\alpha}^{\piso} P\simeq \T_{\textrm{cov}}(A\times_{\alpha} P)$, the same corner realization holds for the partial-isometric
crossed products of the systems $(A,P,\alpha)$ consisting of the positive cones $P$ of abelian lattice-ordered
groups (see also \cite{SZ}).

Now, the present work as an extension of the idea in \cite{LR} follows the framework of \cite{Larsen} for partial-isometric
crossed products. We begin with a preliminary section containing a summary on LCM semigroups and discrete product systems of Hilbert bimodules.
In section \ref{sec:Nica-piso} and \ref{sec:Piso-CP}, for the system $(A,P,\alpha)$ with a left LCM semigroup $P$, a covariant
representation of the system is defined which satisfies a covariance condition called the \emph{(right) Nica covariance}, where the endomorphisms
$\alpha_{x}$ are implemented by partial isometries. This representation is called the \emph{covariant partial-isometric representation}
of the system. We also provide an example which shows that every system admits a nontrivial covariant partial-isometric representation.
Then, we show that there is a $C^*$-algebra $B$ associated with the system generated by a covariant partial-isometric
representation of the system which is universal for covariant partial-isometric representations of the system, in the sense that
there is a bijection between the covariant partial-isometric representations of the system and the nondegenerate
representations of the $C^*$-algebra $B$. This universal algebra $B$ is called the \emph{partial-isometric crossed product} of
the system $(A,P,\alpha)$ and denoted by $A\times_{\alpha}^{\piso} P$, which is unique up to isomorphism. We also show that
this crossed product behaves well under short exact sequences. In section \ref{sec:tensor}, we show that under some certain conditions
the crossed product $(A\otimes_{\textrm{max}} B)\times^{\piso}(P\times S)$ can be decomposed as the maximal tensor product of the crossed
products $A\times^{\piso} P$ and $B\times^{\piso} S$. Also, when $P$ and the opposite semigroup $P^{\textrm{o}}$ are both left LCM
we consider the distinguished system $(B_{P},P,\tau)$, where $B_{P}$ is the $C^*$-subalgebra of $\ell^{\infty}(P)$ generated
by the characteristic functions $\{1_{y}: y\in P\}$, and the action $\tau$ on $B_{P}$ is induced by the shift on $\ell^{\infty}(P)$.
Note that each $1_{y}$ is actually the characteristic function of the right ideal $yP=\{yx: x\in P\}$ in $P$.
We then show that the crossed prodcut $B_{P}\times_{\tau}^{\piso} P$ is universal for \emph{bicovariant partial-isometric representations} of $P$,
which are the partial-isometric representations of $P$ satisfying both right and left Nica covariance conditions.
In section \ref{sec:ideals-tensor}, for the crossed product $(A\otimes_{\max} B)\times^{\piso} P$ a composition series
$$0\leq \I_{1}\leq \I_{2} \leq (A\otimes_{\max} B)\times^{\piso} P$$
of ideals is obtained, for which we identify the subquotients
$$\I_{1},\ \ \ \I_{2}/\I_{1},\ \ \textrm{and}\ \ ((A\otimes_{\max} B)\times^{\piso} P)/\I_{2}$$
with familiar terms. Finally in section \ref{sec:apply}, as an application, we study the partial-isometric crossed product of the
dynamical system considered in \cite{LPR}.

\section{Preliminaries}
\label{sec:pre}


\subsection{LCM semigroups}
\label{subsec:lcm}

Let $P$ be a discrete semigroup. We assume that $P$ is unital, which means that there is an element $e\in P$ such that $xe=ex=x$ for all $x\in P$. Recall that $P$ is called \emph{right cancellative} if $xz=yz$, then $x=y$ for every $x,y,z\in P$.

\begin{definition}
\label{lcm-def}
A unital semigroup $P$ is called left LCM (least common multiple) if it is right cancellative and for every
$x,y\in P$, we have either $Px\cap Py=\emptyset$ or $Px\cap Py=Pz$ for some $z\in P$.
\end{definition}

Let $P^{*}$ denote the set of all invertible elements of $P$, which is obviously not empty as $e\in P^{*}$. In fact, $P^{*}$
is a group with the action inherited from $P$. Now, if $Px\cap Py=Pz$, since $z=ez\in Pz$, we have $sx=z=ty$ for some
$s,t\in P$. So, $z$ can be viewed as a least common left multiple of $x,y$. However, such a least common left multiple may not be
unique. Actually one can see that if $z$ and $\tilde{z}$ are both least common left multiples of $x,y$, then there is an
invertible element $u$ of $P$ ($u\in P^{*}$) such that $\tilde{z}=uz$.
Note that right LCM semigroups are defined similarly.
A unital semigroup $P$ is called \emph{right LCM} if it is left cancellative and for every
$x,y\in P$, we have either $xP\cap yP=\emptyset$ or $xP\cap yP=zP$ for some $z\in P$.
Let $P^{\textrm{o}}$ denote the opposite semigroup endowed with the action $\bullet$ such that $x\bullet y=yx$ for all $x,y\in P^{\textrm{o}}$.
Clearly, $P$ is a left LCM semigroup if and only if $P^{\textrm{o}}$ is a right LCM semigroup.

LCM semigroups actually appear as a natural generalization of the well-known notion of quasi-latticed ordered groups introduced
first by Nica in \cite{Nica}. Let $G$ be a group and $P$ a unital subsemigroup such that $P\cap P^{-1}=\{e\}$. There is a partial order
on $G$ defined by $P$ such that
$$x\leq_{\rt} y \Longleftrightarrow yx^{-1}\in P \Longleftrightarrow y\in Px \Longleftrightarrow Py\subseteq Px$$
for all $x,y\in G$. Moreover, we have
$$x\leq_{\rt} y \Longleftrightarrow xz\leq_{\rt} yz$$
for all $x,y,z\in G$, which means that this partial order is right-invariant. The partial order $\leq_{\rt}$ on $G$ is said to be a
\emph{right quasi-lattice order} if every finite subset of $G$ which has an upper bound in $P$ has a least upper bound in $P$.
In this case, the pair $(G,P)$ is called a \emph{right quasi-lattice ordered group}. Note that a left-invariant partial order
on $G$ is also defined by $P$ such that $x\leq_{\lt} y$ iff $x^{-1}y\in P$. So, a \emph{left quasi-lattice ordered group}
$(G,P)$ can be defined similarly. Now, it is not difficult to see that
if $(G,P)$ is a right quasi-lattice ordered group, then the semigroup $P$ is a left LCM semigroup.
Similarly, if $(G,P)$ is a left quasi-lattice ordered group, then $P$ is a right LCM semigroup.
Note that, however, a LCM semigroup $P$ is not necessarily embedded in a group $G$
such that $(G,P)$ is a quasi-lattice ordered group (see more in \cite{B-Li}).

\begin{example}
\label{lcm1}
For any positive integer $k$, the positive cone $\N^{k}$ of the abelian lattice-ordered group $\Z^{k}$ is a left LCM semigroup, such that
$$\N^{k}x \cap \N^{k}y=\N^{k}(x\vee y)$$
for all $x,y\in \N^{k}$, where $x\vee y$ denotes the supremum of $x$ and $y$ given by $(x\vee y)_{i}=\max\{x_{i},y_{i}\}$ for $1\leq i\leq k$.
\end{example}

\begin{example}
\label{lcm2}
Let $\N^{\times}$ denote the set of positive integers, which is a unital (abelian) semigroup with the usual multiplication. It is indeed
a left LCM semigroup such that
$$\N^{\times}r \cap \N^{\times}s=\N^{\times} (r\vee s)$$
for all $r,s\in \N^{\times}$, where
$$r\vee s=\textrm{the least common multiple of numbers}\ r\ \textrm{and}\ s.$$
In fact, if $\Q_{+}^{\times}$ denotes the abelian multiplicative group of positive rationals, then it is a lattice-ordered group with
$$r\leq s\Longleftrightarrow (\frac{s}{r})\in \N^{\times} \Longleftrightarrow sr^{-1}\in \N^{\times}$$
for all $r,s\in \Q_{+}^{\times}$. So, $\N^{\times}$ is actually the positive cone of $\Q_{+}^{\times}$, namely,
$$\N^{\times}=\{r\in \Q_{+}^{\times}: r\geq 1\}.$$
\end{example}

\begin{example}
\label{free-G}
Assume that $n$ is an integer such that $n\geq 2$. Let $\F_{n}$ be the free group on $n$ generators $\{a_{1}, a_{2},..., a_{n}\}$, and
$\F_{n}^{+}$ the unital subsemigroup of $\F_{n}$ generated by the nonnegative powers of $a_i$'s. Then, for the right-invariant
partial order $\leq_{\rt}$ on $\F_{n}$ defined by $\F_{n}^{+}$, we have $x\leq_{\rt}y$ if and only if $x$ is a final string on the right
of $y$, whit which, $(\F_{n},\F_{n}^{+})$ is a right quasi-lattice ordered group. Thus, $\F_{n}^{+}$ is a left LCM semigroup.
The left-invariant partial order $\leq_{\lt}$ on $\F_{n}$ is given such that $x\leq_{\lt}y$ if and only if $x$ is an initial string on the left
of $y$, whit which, $(\F_{n},\F_{n}^{+})$ is a left quasi-lattice ordered group. So, in this case, $\F_{n}^{+}$ is a right LCM semigroup.
\end{example}

\subsection{Discrete product systems of Hilbert bimodules}
\label{subsec:prod-syst}

A \emph{Hilbert bimodule} over a $C^*$-algebra $A$ is a right Hilbert $A$-module $X$ together with a
homomorphism $\phi:A\rightarrow \L(X)$ which defines a left action of $A$ on $X$ by $a\cdot x=\phi(a)x$ for all $a\in A$ and
$x\in X$. A \emph{Toeplitz representation} of $X$ in a $C^*$-algebra $B$ is a pair $(\psi,\pi)$ consisting of a linear
map $\psi:X\rightarrow B$ and a homomorphism $\pi:A\rightarrow B$ such that
$$\psi(x\cdot a)=\psi(x)\pi(a),\ \ \psi(x)^{*}\psi(y)=\pi(\langle x,y\rangle_{A}),\ \ \textrm{and}\ \ \psi(a\cdot x)=\pi(a)\psi(x)$$
for all $a\in A$ and $x,y\in X$. Then, there is a homomorphism (Pimsner homomorphism) $\rho:\K(X)\rightarrow B$ such that
\begin{align}
\label{pimsner-hom}
\rho(\Theta_{x,y})=\psi(x)\psi(y)^{*}\ \ \ \textrm{for all}\ x,y\in X.
\end{align}

The \emph{Toeplitz algebra} of $X$ is the $C^*$-algebra $\T(X)$ which is universal for Toeplitz representations
of $X$ (see \cite{pimsner, Fow-Rae}).

Recall that every right Hilbert $A$-module $X$ is essential, which means that we have
$$X=\clsp\{x\cdot a: x\in X, a\in A\}.$$
Moreover, a Hilbert bimodule $X$ over a $C^*$-algebra $A$ is called \emph{essential} if
$$X=\clsp\{a\cdot x: a\in A, x\in X\}=\clsp\{\phi(a)x: a\in A, x\in X\},$$
which means that $X$ is also essential as a left $A$-module.

Now, let $A$ be a $C^*$-algebra and $S$ a unital (countable) discrete semigroup.
We recall from \cite{Fowler} that the disjoint union $X=\bigsqcup_{s\in S} X_{s}$ of Hilbert bimodules $X_{s}$ over $A$ is called a
\emph{discrete product system} over $S$ if there is a multiplication
\begin{align}
\label{eq31}
(x,y)\in X_{s}\times X_{t}\mapsto xy\in X_{st}
\end{align}
on $X$, with which, $X$ is a semigroup, and the map (\ref{eq31}) extends to an isomorphism of the Hilbert bimodules
$X_{s}\otimes_{A} X_{t}$ and $X_{xt}$ for all $s,t\in S$ with $s,t\neq e$. The bimodule $X_{e}$ is $_{A}A_{A}$, and the
multiplications $X_{e}\times X_{s}\mapsto X_{s}$ and $X_{s}\times X_{e}\mapsto X_{s}$ are just given by the module actions of
$A$ on $X_{s}$. Note that also we write $\phi_{s}:A\rightarrow \L(X_{s})$ for the homomorphism which defines
the left action of $A$ on $X_{s}$.

Note that, for every $s,t\in S$ with $s\neq e$, there is a homomorphism $\iota_{s}^{st}:\L(X_{s})\rightarrow \L(X_{st})$
characterized by
$$\iota_{s}^{st}(T)(xy)=(Tx)y$$
for all $x\in X_{s}$, $y\in X_{t}$ and $T\in \L(X_{s})$. In fact, $\iota_{s}^{st}(T)=T\otimes \id_{X_{t}}$.

A \emph{Toeplitz representation} of the product system $X$ in a $C^*$-algebra $B$ is a map $\psi:X\rightarrow B$ such that
\begin{itemize}
\item[(1)] $\psi_{s}(x)\psi_{t}(y)=\psi_{st}(xy)$ for all $s,t\in S$, $x\in X_{s}$, and $y\in X_{t}$; and
\item[(2)] the pair $(\psi_{s},\psi_{e})$ is a Toeplitz representation of $X_{s}$ in $B$ for all $s\in S$,
\end{itemize}
where $\psi_{s}$ denotes the restriction of $\psi$ to $X_{s}$. For every $s\in S$, let $\psi^{(s)}:\K(X_{s})\rightarrow B$ be the
Pimsner homomorphism corresponding to the pair $(\psi_{s},\psi_{e})$ defined by
$$\psi^{(s)}(\Theta_{x,y})=\psi_{s}(x)\psi_{s}(y)^{*}$$
for all $x,y\in X_{s}$ (see (\ref{pimsner-hom})).

By \cite[Proposition 2.8]{Fowler}, for every product system $X$ over $S$, there is a $C^*$-algebra $\T(X)$, called the
\emph{Toeplitz algebra} of $X$, which is generated by a universal Toeplitz representation $i_{X}:X\rightarrow \T(X)$ of $X$.
The pair $(\T(X),i_{X})$ is unique up to isomorphism, and $i_{X}$ is isometric.

Next, we recall that for any quasi-lattice ordered group $(G,S)$, the notions of compactly aligned product system over $S$
and Nica covariant Toeplitz representation of it were introduced first by Fowler in \cite{Fowler}. Then, authors in \cite{NB} extended
these notions to product systems over right LCM semigroups. Suppose that $S$ is a unital right LCM semigroup. A product
system $X$ over $S$ of Hilbert bimodules is called \emph{compactly aligned} if for all $r,t\in S$ such that
$rS\cap tS=sS$ for some $s\in S$ we have
$$\iota_{r}^{s}(R)\iota_{t}^{s}(T)\in \K(X_{s})$$
for all $R\in \K(X_{r})$ and $T\in \K(X_{t})$. Let $X$ be a compactly aligned product system over a right LCM semigroup $S$,
and $\psi:X\rightarrow B$ a Toeplitz representation of $X$ in a $C^*$-algebra $B$. Then,
$\psi$ is called \emph{Nica covariant} if
\begin{align}
\psi^{(r)}(R)\psi^{(t)}(T)=
   \begin{cases}
      \psi^{(s)}\big( \iota_{r}^{s}(R)\iota_{t}^{s}(T) \big) &\textrm{if}\empty\ \text{$rS\cap tS=sS$,}\\
      0 &\textrm{if}\empty\ \text{$rS\cap tS=\emptyset$}\\
   \end{cases}
\end{align}
for all $r,t\in S$, $R\in \K(X_{r})$ and $T\in \K(X_{t})$.

For a compactly aligned product system $X$ over a right LCM semigroup $S$, the \emph{Nica-Toeplitz algebra} $\mathcal{N}\T(X)$
is the $C^*$-algebra generated by a Nica covariant Toeplitz representation $i_{X}: X\rightarrow \mathcal{N}\T(X)$ which is
universal for Nica covariant Toeplitz representations of $X$, which means that,
for every Nica covariant Toeplitz representation of $\psi$ of $X$, there is a representation $\psi_{*}$ of $\mathcal{N}\T(X)$ such that
$\psi_{*}\circ i_{X}=\psi$ (see \cite{Fowler,NB}).

\section{Nica partial-isometric representations}
\label{sec:Nica-piso}

Let $P$ be a left LCM semigroup. A \emph{partial-isometric representation} of $P$ on a Hilbert space $H$ is a
map $V:P\rightarrow B(H)$ such that each $V_{x}:=V(x)$ is a partial isometry, and the map $V$ is a unital semigroup homomorphism
of $P$ into the multiplicative semigroup $B(H)$. Moreover, if the representation $V$ satisfies the equation
\begin{align}\label{Nica-1}
V_{x}^{*}V_{x}V_{y}^{*}V_{y}=
   \begin{cases}
      V_{z}^{*}V_{z} &\textrm{if}\empty\ \text{$Px\cap Py=Pz$,}\\
      0 &\textrm{if}\empty\ \text{$Px\cap Py=\emptyset$,}\\
   \end{cases}
\end{align}
then it is called a \emph{Nica partial-isometric representation} of $P$ on $H$.
The equation (\ref{Nica-1}) is called the \emph{Nica covariance condition}. Of course, since the least common left multiple $z$
may not be unique, we must check that whether the Nica covariance condition is well-defined. So, assume that
$Pz=Px\cap Py=P\tilde{z}$. If follows that
$$V_{x}^{*}V_{x}V_{y}^{*}V_{y}=V_{z}^{*}V_{z}$$
and
$$V_{x}^{*}V_{x}V_{y}^{*}V_{y}=V_{\tilde{z}}^{*}V_{\tilde{z}}.$$
Now, since $\tilde{z}=uz$ for some invertible element $u$ of $P$, we have
$$V_{\tilde{z}}^{*}V_{\tilde{z}}=V_{uz}^{*}V_{uz}=(V_{u}V_{z})^{*}V_{u}V_{z}=V_{z}^{*}V_{u}^{*}V_{u}V_{z}.$$
But it is not difficult to see that $V_{u}$ is actually a unitary, and therefore,
$$V_{\tilde{z}}^{*}V_{\tilde{z}}=V_{z}^{*}V_{z}.$$ This implies that the equation (\ref{Nica-1}) is indeed well-defined.

The following example shows that given a left LCM semigroup $P$, a Nica partial-isometric representation of $P$ exists.

\begin{example}
\label{S-W-exmp}
Suppose that $P$ is a left LCM semigroup and $H$ a Hilbert space.
Define a map $S:P\rightarrow B(\ell^{2}(P)\otimes H)$ by
\[
(S_{y}f)(x)=
   \begin{cases}
      f(r) &\textrm{if}\empty\ \text{$x=ry$ for some $r\in P$,}\\
      0 &\textrm{otherwise.}\\
   \end{cases}
\]
for every $f\in \ell^{2}(P)\otimes H$. Note that $x=ry$ for some $r\in P$ is equivalent to saying that $x\in Py$. Moreover,
if $sy=x=ry$ for some $r,s\in P$, then $s=r$ by the right cancellativity of $P$, and hence $f(r)=f(s)$. This implies that
each $S_{y}$ is well-defined. One can see that each $S_{y}$ is a linear operator. We claim that each $S_{y}$ is actually
an isometry, and in particular, $S_{e}=1$. We have
$$\|S_{y}f\|^{2}=\displaystyle\sum_{x\in P} \|(S_{y}f)(x)\|^{2}=\displaystyle\sum_{r\in P} \|(S_{y}f)(ry)\|^{2}=
\displaystyle\sum_{r\in P} \|f(r)\|^{2}=\|f\|^{2},$$
which implies that each $S_{y}$ is an isometry. In particular,
$$(S_{e}f)(x)=(S_{e}f)(xe)=f(x),$$ which shows that $S_{e}=1$. In addition, a simple calculation shows that
\begin{align}
\label{SxSy}
S_{x}S_{y}=S_{yx}=S_{x\bullet y}\ \ \textrm{for all}\ x,y\in P.
\end{align}
Next, we want to show that the adjoint of each $S_{y}$ is given by
$$(W_{y}f)(x)=f(xy)$$ for all $f\in \ell^{2}(P)\otimes H$. For every $f,g\in \ell^{2}(P)\otimes H$, we have
\begin{eqnarray*}
\begin{array}{rcl}
\langle S_{y}f|g \rangle&=&\displaystyle\sum_{x\in P}\langle (S_{y}f)(x)|g(x) \rangle\\
&=&\displaystyle\sum_{r\in P}\langle (S_{y}f)(ry)|g(ry) \rangle\\
&=&\displaystyle\sum_{r\in P}\langle f(r)|g(ry) \rangle=\displaystyle\sum_{r\in P}\langle f(r)|(W_{y}g)(r) \rangle
=\langle f|W_{y}g \rangle.
\end{array}
\end{eqnarray*}
So, $S_{y}^{*}=W_{y}$ for every $y\in P$. Also, for every $x,y\in P$, by applying (\ref{SxSy}), we get
$$W_{x}W_{y}=S_{x}^{*}S_{y}^{*}=[S_{y}S_{x}]^{*}=S_{xy}^{*}=W_{xy}.$$
Therefore, since each $W_{x}$ is obviously a partial-isometry, it follows that the map
$W:P\rightarrow B(\ell^{2}(P)\otimes H)$ defined by $(W_{y}f)(x)=f(xy)$ is a partial-isometric representation of
$P$ on $\ell^{2}(P)\otimes H$. We claim that the representation $W$ satisfies the Nica covariance condition (\ref{Nica-1}).
Firstly,
\begin{align}
\label{eq2}
W_{x}^{*}W_{x}W_{y}^{*}W_{y}=S_{x}S_{x}^{*}S_{y}S_{y}^{*}.
\end{align}
Then, for every $f\in \ell^{2}(P)\otimes H$, we have
\begin{align}
\label{eq3}
(S_{x}^{*}S_{y}f)(r)=(S_{x}^{*}(S_{y}f))(r)=(S_{y}f)(rx)\ \ \textrm{for all}\ r\in P.
\end{align}
Now, if $Px\cap Py=\emptyset$, since $rx\in Px$, it follows that $rx\not\in Py$, and therefore,
$$(S_{y}f)(rx)=0.$$
Thus, $(S_{x}^{*}S_{y}f)(r)=0$, which implies that the equation (\ref{eq2}) must be equal to zero when $Px\cap Py=\emptyset$.
Suppose the otherwise, namely, $Px\cap Py=Pz$. Note that first,
if $\{\varepsilon_{s}: s\in P\}$ is the usual orthonormal basis of $\ell^{2}(P)$, then each $S_{y}S_{y}^{*}$ is a projection
onto the closed subspace $\ell^{2}(Py)\otimes H$ of $\ell^{2}(P)\otimes H$ spanned by the elements
$$\{\varepsilon_{sy}\otimes h: s\in P, h\in H\},$$ which is indeed equal to the $\ker (1-S_{y}S_{y}^{*})$. So, for every
$f\in \ell^{2}(P)\otimes H$,
\[
\big(S_{x}S_{x}^{*}(S_{y}S_{y}^{*}f)\big)(r)=
   \begin{cases}
      (S_{y}S_{y}^{*}f)(r) &\textrm{if}\empty\ \text{$r\in Px$,}\\
      0 &\textrm{otherwise.}\\
   \end{cases}
\]
Moreover, for $(S_{y}S_{y}^{*}f)(r)$, where $r\in Px$, we have
\[
(S_{y}S_{y}^{*}f)(r)=
   \begin{cases}
      f(r) &\textrm{if}\empty\ \text{$r\in Py$,}\\
      0 &\textrm{otherwise.}\\
   \end{cases}
\]
It thus follows that
\[
\big(S_{x}S_{x}^{*}(S_{y}S_{y}^{*}f)\big)(r)=
   \begin{cases}
      f(r) &\textrm{if}\empty\ \text{$r\in (Px\cap Py)=Pz$,}\\
      0 &\textrm{otherwise,}\\
   \end{cases}
\]
which equals $(S_{z}S_{z}^{*}f)(r)$. Therefore, we have
$$S_{x}S_{x}^{*}S_{y}S_{y}^{*}=S_{z}S_{z}^{*},$$ from which, for the equation (\ref{eq2}), we get
$$W_{x}^{*}W_{x}W_{y}^{*}W_{y}=S_{z}S_{z}^{*}=W_{z}^{*}W_{z}.$$
Consequently, $W$ is indeed a Nica partial-isometric representation.

\end{example}

\begin{remark}
\label{Nica-piso-TOG}
If $P$ is the positive cone of a totally ordered group $G$, then every partial-isometric representation $V$ of $P$ automatically satisfies the Nica covariance condition (\ref{Nica-1}), such that
$$V_{x}^{*}V_{x}V_{y}^{*}V_{y}=V_{\max\{x,y\}}^{*}V_{\max\{x,y\}}\ \ \ \textrm{for all}\ x,y\in P.$$
\end{remark}

\begin{lemma}
\label{Nica-piso-FG}
Consider the right quasi-lattice ordered group $(\F_{n},\F_{n}^{+})$ (see Example \ref{free-G}). A partial-isometric representation $V$ of $\F_{n}^{+}$ satisfies the Nica covariance condition (\ref{Nica-1}) if and only if the initial projections $V_{a_{i}}^{*}V_{a_{i}}$ and $V_{a_{j}}^{*}V_{a_{j}}$ have orthogonal ranges, where
$1\leq i,j\leq n$ such that $i\neq j$.
\end{lemma}

\begin{proof}
Suppose that $V$ is a partial-isometric representation of $\F_{n}^{+}$ on a Hilbert space $H$. If it satisfies the Nica covariance condition (\ref{Nica-1}), then one can see that, for every $1\leq i,j\leq n$ with $i\neq j$, we have
$$V_{a_{i}}^{*}V_{a_{i}}V_{a_{j}}^{*}V_{a_{j}}=0$$
as $\F_{n}^{+}a_{i} \cap \F_{n}^{+}a_{j}=\emptyset$. So, it follows that the initial projections
$V_{a_{i}}^{*}V_{a_{i}}$ and $V_{a_{j}}^{*}V_{a_{j}}$ have orthogonal ranges for every $1\leq i,j\leq n$ with $i\neq j$.

Conversely, suppose that for every $1\leq i,j\leq n$ with $i\neq j$, the initial projections
$V_{a_{i}}^{*}V_{a_{i}}$ and $V_{a_{j}}^{*}V_{a_{j}}$ have orthogonal ranges. Therefore, we have
\begin{align}
\label{eq33}
V_{a_{i}}^{*}V_{a_{i}}V_{a_{j}}^{*}V_{a_{j}}=0
\end{align}
for every $1\leq i,j\leq n$ with $i\neq j$. Now, for every $x,y\in \F_{n}^{+}$, if $\F_{n}^{+}x \cap \F_{n}^{+}y\neq\emptyset$,
then $x$ is a final string on the right of $y$ or $y$ is a final string on the right of $x$. Suppose that $x$ is a
final string on the right of $y$, from which, it follows that $\F_{n}^{+}x \cap \F_{n}^{+}y=\F_{n}^{+}y$, and $y=(yx^{-1})x$,
where $yx^{-1}\in \F_{n}^{+}$. Therefore, we have
\begin{eqnarray*}
\begin{array}{rcl}
V_{x}^{*}V_{x}V_{y}^{*}V_{y}&=&V_{x}^{*}V_{x}(V_{yx^{-1}x})^{*}V_{y}\\
&=&V_{x}^{*}V_{x}(V_{yx^{-1}}V_{x})^{*}V_{y}\\
&=&V_{x}^{*}V_{x}V_{x}^{*}V_{yx^{-1}}^{*}V_{y}\\
&=&V_{x}^{*}V_{yx^{-1}}^{*}V_{y}\\
&=&(V_{yx^{-1}}V_{x})^{*}V_{y}=(V_{yx^{-1}x})^{*}V_{y}=V_{y}^{*}V_{y}.
\end{array}
\end{eqnarray*}
If $y$ is the final string on the right of $x$, a similar computation shows that
$V_{x}^{*}V_{x}V_{y}^{*}V_{y}=V_{x}^{*}V_{x}$ as $\F_{n}^{+}x \cap \F_{n}^{+}y=\F_{n}^{+}x$.
If $\F_{n}^{+}x \cap \F_{n}^{+}y=\emptyset$, then $x\neq y$. Note that we can consider $x$ and $y$ as two strings of letters $a_i$'s
with equal lengths by adding, for example, a finite number of $(a_{1})^{0}$ to the left of the shorter string. Therefore, since $x\neq y$,
we can write
$$x=s a_{i} z\ \ \textrm{and}\ \ y=t a_{j} z,$$
where $s,t,z\in \F_{n}^{+}$, and $i\neq j$. It follows that
\begin{eqnarray*}
\begin{array}{rcl}
V_{x}^{*}V_{x}V_{y}^{*}V_{y}&=&V_{x}^{*}V_{s a_{i} z}(V_{t a_{j} z})^{*}V_{y}\\
&=&V_{x}^{*} V_{s} V_{a_{i}} V_{z} (V_{t} V_{a_{j}} V_{z})^{*} V_{y}\\
&=&V_{x}^{*} V_{s} (V_{a_{i}}) V_{z}V_{z}^{*} (V_{a_{j}}^{*}) V_{t}^{*}  V_{y}\\
&=&V_{x}^{*} V_{s} (V_{a_{i}}V_{a_{i}}^{*}V_{a_{i}}) V_{z}V_{z}^{*} (V_{a_{j}}^{*}V_{a_{j}}V_{a_{j}}^{*}) V_{t}^{*}  V_{y}\\
&=&V_{x}^{*} V_{s} V_{a_{i}} (V_{a_{i}}^{*}V_{a_{i}}) (V_{z}V_{z}^{*}) (V_{a_{j}}^{*}V_{a_{j}}) V_{a_{j}}^{*} V_{t}^{*}  V_{y}.
\end{array}
\end{eqnarray*}
Now, in the bottom line, since the product $V_{a_{i}} V_{z}$ of the partial-isometries $V_{a_{i}}$ and $V_{z}$ is a partial-isometry, namely,
$V_{a_{i}z}$, by \cite[Lemma 2]{Halmos}, $V_{a_{i}}^{*}V_{a_{i}}$ commutes with $V_{z}V_{z}^{*}$. So, we get
\begin{eqnarray*}
\begin{array}{rcl}
V_{x}^{*}V_{x}V_{y}^{*}V_{y}&=&
V_{x}^{*} V_{s} V_{a_{i}} (V_{z}V_{z}^{*}) (V_{a_{i}}^{*}V_{a_{i}} V_{a_{j}}^{*}V_{a_{j}}) V_{a_{j}}^{*} V_{t}^{*}  V_{y}\\
&=&V_{x}^{*} V_{s} V_{a_{i}} (V_{z}V_{z}^{*}) (0) V_{a_{j}}^{*} V_{t}^{*}  V_{y}=0\ \ \ \ \ (by\ (\ref{eq33})).
\end{array}
\end{eqnarray*}
Thus, the representation $V$ satisfies the Nica covariance condition (\ref{Nica-1}).
\end{proof}

\section{Partial-isometric crossed products}
\label{sec:Piso-CP}

\subsection{Covariant partial-isometric representations}
\label{subsec:cov-piso-rep}

Let $P$ be a left LCM semigroup, and $(A,P,\alpha)$ a dynamical system consisting
of a $C^*$-algebra $A$, and an action $\alpha:P\rightarrow \End (A)$ of $P$ by extendible endomorphisms of $A$ such that
$\alpha_{e}=\id_{A}$.

\begin{definition}
\label{cov.pair}
A \emph{covariant partial-isometric representation} of $(A,P,\alpha)$ on a Hilbert space $H$ is a pair $(\pi,V)$ consisting of
a nondegenerate representation $\pi:A\rightarrow B(H)$ and a Nica partial-isometric representation $V:P\rightarrow B(H)$ of $P$ such that
\begin{align}
\label{cov1}
\pi(\alpha_{x}(a))=V_{x}\pi(a) V_{x}^{*}\ \ \textrm{and}\ \ V_{x}^{*}V_{x} \pi(a)=\pi(a) V_{x}^{*}V_{x}
\end{align}
for all $a\in A$ and $x\in P$.
\end{definition}

\begin{lemma}
\label{cov2-lemma}
Every covariant partial-isometric pair $(\pi,V)$ extends to a covariant partial-isometric representation $(\overline{\pi},V)$ of
the system $(M(A),P,\overline{\alpha})$, and (\ref{cov1}) is equivalent to
\begin{align}
\label{cov2}
\pi(\alpha_{x}(a))V_{x}=V_{x}\pi(a)\ \ \textrm{and}\ \ V_{x}V_{x}^{*}=\overline{\pi}(\overline{\alpha}_{x}(1))
\end{align}
for all $a\in A$ and $x\in P$.
\end{lemma}

\begin{proof}
We skip the proof as it follows by similar discussions to the first part of \cite[\S4]{LR}.
\end{proof}

The following example shows that every dynamical system $(A,P,\alpha)$ admits a nontrivial (nonzero)
covariant partial-isometric representation.

\begin{example}
\label{exmp-piso-rep}
Suppose that $(A,P,\alpha)$ is a dynamical system, and $\pi_{0}:A\rightarrow B(H)$ a nondegenerate representation of $A$ on
a Hilbert space $H$. Define a map $\pi:A\rightarrow B(\ell^{2}(P)\otimes H)$ by
$$(\pi(a)f)(x)=\pi_{0}(\alpha_{x}(a))f(x)$$
for all $a\in A$ and $f\in \ell^{2}(P)\otimes H\simeq \ell^{2}(P, H)$. One can see that $\pi$ is a representation of $A$ on the Hilbert space
$\ell^{2}(P)\otimes H$. Let $q:\ell^{2}(P)\otimes H\rightarrow \ell^{2}(P)\otimes H$ be a map defined by
$$(qf)(x)=\overline{\pi_{0}}(\overline{\alpha}_{x}(1))f(x)$$
for all $f\in \ell^{2}(P)\otimes H$. It is not difficult to see that $q\in B(\ell^{2}(P)\otimes H)$, which is actually a projection
onto a closed subspace $\H$ of $\ell^{2}(P)\otimes H$. We claim that if $\{a_{i}\}$ is any approximate unit in $A$, then
$\pi(a_{i})$ converges strictly to $q$ in $\M\big(\K(\ell^{2}(P)\otimes H)\big)=B(\ell^{2}(P)\otimes H)$. To prove our claim, since
the net $\{\pi(a_{i})\}$ is a norm bounded subset of $B(\ell^{2}(P)\otimes H)$, and $\pi(a_{i})^{*}=\pi(a_{i})$ for each $i$
as well as $q^{*}=q$, by \cite[Proposition C.7]{RW}, we only need to show that $\pi(a_{i})\rightarrow q$ strongly
in $B(\ell^{2}(P)\otimes H)$. If $\{\varepsilon_{x}: x\in P\}$ is the usual orthonormal basis of $\ell^{2}(P)$, then it is enough
to see that
$$\pi(a_{i})(\varepsilon_{x}\otimes \pi_{0}(a)h)\rightarrow q(\varepsilon_{x}\otimes \pi_{0}(a)h)$$
for each spanning element $(\varepsilon_{x}\otimes \pi_{0}(a)h)$ of $\ell^{2}(P)\otimes H$ (recall that $\pi_{0}$ is nondegenerate).
We have
$$\pi(a_{i})(\varepsilon_{x}\otimes \pi_{0}(a)h)=\varepsilon_{x}\otimes \pi_{0}(\alpha_{x}(a_{i}))\pi_{0}(a)h
=\varepsilon_{x}\otimes \pi_{0}(\alpha_{x}(a_{i})a)h,$$
which is convergent to
$$\varepsilon_{x}\otimes \pi_{0}(\overline{\alpha}_{x}(1)a)h=\varepsilon_{x}\otimes \overline{\pi_{0}}(\overline{\alpha}_{x}(1))\pi_{0}(a)h
=q(\varepsilon_{x}\otimes \pi_{0}(a)h)$$
in $\ell^{2}(P)\otimes H$. This is due to the extendibility of each $\alpha_{x}$. Therefore, $\pi(a_{i})\rightarrow q$ strictly in
$B(\ell^{2}(P)\otimes H)$.

Next, let $W:P\rightarrow B(\ell^{2}(P)\otimes H)$ be the Nica partial-isometric representation introduced in Example \ref{S-W-exmp}.
We aim at constructing a covariant partial-isometric representation $(\rho,V)$ of $(A,P,\alpha)$ on the Hilbert space (closed subspace) $\H$
by using the pair $(\pi,W)$. Note that, in general, $\pi$ is not nondegenerate on $\ell^{2}(P)\otimes H$, unless
$\overline{\alpha}_{x}(1)=1$ for every $x\in P$. So, for our purpose, we first show that
\begin{align}
\label{eq4}
W_{x}\pi(a)=\pi(\alpha_{x}(a))W_{x}\ \ \textrm{and}\ \ W_{x}^{*}W_{x}\pi(a)=\pi(a)W_{x}^{*}W_{x}
\end{align}
for all $a\in A$ and $x\in P$. For every $f\in \ell^{2}(P)\otimes H$, we have
\begin{eqnarray*}
\begin{array}{rcl}
(W_{x}\pi(a)f)(r)&=&(W_{x}(\pi(a)f))(r)\\
&=&(\pi(a)f)(rx)\\
&=&\pi_{0}(\alpha_{rx}(a))f(rx)\\
&=&\pi_{0}(\alpha_{r}(\alpha_{x}(a)))(W_{x}f)(r)\\
&=&(\pi(\alpha_{x}(a))W_{x}f)(r)\\
\end{array}
\end{eqnarray*}
for all $r\in P$. So, $W_{x}\pi(a)=\pi(\alpha_{x}(a))W_{x}$ is valid, from which, we get
$\pi(a)W_{x}^{*}=W_{x}^{*}\pi(\alpha_{x}(a))$. One can apply these two equations to see that
$W_{x}^{*}W_{x}\pi(a)=\pi(a)W_{x}^{*}W_{x}$ is also valid. Also, since $W_{x}W_{x}^{*}=S_{x}^{*}S_{x}=1$ (see Example \ref{S-W-exmp}),
each $W_{x}$ is a coisometry, and hence, by applying the equation $W_{x}\pi(a)=\pi(\alpha_{x}(a))W_{x}$, we have
\begin{align}
\label{eq5}
W_{x}\pi(a)W_{x}^{*}=\pi(\alpha_{x}(a))W_{x}W_{x}^{*}=\pi(\alpha_{x}(a)).
\end{align}
Now we claim that the pair $(\rho,V)=(q\pi q,qWq)$ is a covariant partial-isometric representation of
of $(A,P,\alpha)$ on $\H$. More precisely, consider the maps
$$\rho:A\rightarrow qB(\ell^{2}(P)\otimes H)q\simeq B(\H)$$
and
$$V:P\rightarrow qB(\ell^{2}(P)\otimes H)q\simeq B(\H)$$
defined by
$$\rho(a)=q\pi(a)q=\pi(a)\ \ \ \textrm{and}\ \ \ V_{x}=qW_{x}q$$
for all $a\in A$ and $x\in P$, respectively. Since for any approximate unit $\{a_{i}\}$ in $A$,
$\rho(a_{i})=\pi(a_{i})\rightarrow q$ strongly in $B(\H)$, where $q=1_{B(\H)}$, it follows that the
representation $\rho$ is nondegenerate. Moreover, by applying (\ref{eq5}), we have
$$V_{x}\rho(a)V_{x}^{*}=qW_{x}q\pi(a)qW_{x}^{*}q=qW_{x}\pi(a)W_{x}^{*}q=q\pi(\alpha_{x}(a))q=\rho(\alpha_{x}(a)).$$
Also, by applying the first equation of (\ref{eq4}) and $\pi(a)W_{x}^{*}=W_{x}^{*}\pi(\alpha_{x}(a))$ along with
the fact that $\rho(a)=q\pi(a)q=q\pi(a)=\pi(a)q=\pi(a)$, we get
\begin{eqnarray*}
\begin{array}{rcl}
V_{x}^{*}V_{x}\rho(a)&=&qW_{x}^{*}qW_{x}q\pi(a)q\\
&=&qW_{x}^{*}qW_{x}\pi(a)q\\
&=&qW_{x}^{*}q\pi(\alpha_{x}(a))W_{x}q\\
&=&qW_{x}^{*}\pi(\alpha_{x}(a))qW_{x}q\\
&=&q\pi(a)W_{x}^{*}qW_{x}q\\
&=&q\pi(a)qW_{x}^{*}qW_{x}q=\rho(a)V_{x}^{*}V_{x}.
\end{array}
\end{eqnarray*}
Thus, it is only left to show that the map $V$ is a Nica partial-isometric representation. To see that each $V_{x}$ is a
partial-isometry, note that, for any approximate unit $\{a_{i}\}$ in $A$,
$$qW_{x}\pi(a_{i})W_{x}^{*}qW_{x}q$$
converges strongly to
$$qW_{x}qW_{x}^{*}qW_{x}q=V_{x}V_{x}^{*}V_{x}$$ in $B(\ell^{2}(P)\otimes H)$. On the other hand,
by applying the covariance equations of the pair $(\pi,W)$, we have
\begin{eqnarray*}
\begin{array}{rcl}
q[W_{x}\pi(a_{i})W_{x}^{*}]qW_{x}q&=&q\pi(\alpha_{x}(a_{i}))qW_{x}q\\
&=&q\pi(\alpha_{x}(a_{i}))W_{x}q=qW_{x}\pi(a_{i})q,
\end{array}
\end{eqnarray*}
which converges strongly to $qW_{x}q=V_{x}$. So, we must have $V_{x}V_{x}^{*}V_{x}=V_{x}$, which means that each
$V_{x}$ is a partial-isometry. To see $V_{x}V_{y}=V_{xy}$ for every $x,y\in P$, we first need to compute $V_{x}f$ for
any $f\in \H$. So, knowing that $qf=f$, we have
\begin{eqnarray*}
\begin{array}{rcl}
[V_{x}f](r)=[qW_{x}f](r)&=&[q(W_{x}f)](r)\\
&=&\overline{\pi_{0}}(\overline{\alpha}_{r}(1))(W_{x}f)(r)\\
&=&\overline{\pi_{0}}(\overline{\alpha}_{r}(1))f(rx)\\
&=&\overline{\pi_{0}}(\overline{\alpha}_{r}(1))(qf)(rx)\\
&=&\overline{\pi_{0}}(\overline{\alpha}_{r}(1))\overline{\pi_{0}}(\overline{\alpha}_{rx}(1))f(rx)\\
&=&\overline{\pi_{0}}(\overline{\alpha}_{r}(1)\overline{\alpha}_{rx}(1))f(rx)\\
&=&\overline{\pi_{0}}(\overline{\alpha}_{rx}(1))f(rx)\\
&=&(qf)(rx)=f(rx)
\end{array}
\end{eqnarray*}
for all $r\in P$.
Thus, by applying the above computation, we get
\begin{eqnarray*}
\begin{array}{rcl}
[V_{x}V_{y}f](r)&=&[V_{x}(V_{y}f)](r)\\
&=&(V_{y}f)(rx)\\
&=&f((rx)y)\\
&=&f(r(xy))=[V_{xy}f](r).
\end{array}
\end{eqnarray*}
So, it follows that $V_{x}V_{y}=V_{xy}$ for all $x,y\in P$. Finally, we show that the partial-isometric representation $V$
satisfies the Nica covariance condition (\ref{Nica-1}). Let us first mention that the Hilbert space $\H$ is spanned by the elements
$$\{\varepsilon_{r}\otimes \overline{\pi_{0}}(\overline{\alpha}_{r}(1))h: r\in P, h\in H\}$$
as a closed subspace of $\ell^{2}(P)\otimes H$. Then, for every $y\in P$ and $f\in \H$, we have
$$(V_{y}^{*}f)(r)=(qW_{y}^{*}f)(r)=(q(S_{y}f))(r)=\overline{\pi_{0}}(\overline{\alpha}_{r}(1))(S_{y}f)(r).$$
Now, if $r=sy$ for some $s\in P$, which means that $r\in Py$, we get
$$(V_{y}^{*}f)(r)=\overline{\pi_{0}}(\overline{\alpha}_{sy}(1))(S_{y}f)(sy)=\overline{\pi_{0}}(\overline{\alpha}_{sy}(1))f(s).$$
Otherwise, $(V_{y}^{*}f)(r)=0$. It therefore follows that, if $r=sy$ for some $s\in P$, then
\begin{eqnarray*}
\begin{array}{rcl}
[V_{y}^{*}V_{y}f](r)&=&[V_{y}^{*}(V_{y}f)](r)\\
&=&\overline{\pi_{0}}(\overline{\alpha}_{sy}(1))(V_{y}f)(s)\\
&=&\overline{\pi_{0}}(\overline{\alpha}_{sy}(1))f(sy)\\
&=&(qf)(sy)=f(sy)=f(r).
\end{array}
\end{eqnarray*}
Otherwise, $[V_{y}^{*}V_{y}f](r)=0$. This implies that each $V_{y}^{*}V_{y}$ is the projection of $\H$ onto the closed subspace
$$\H_{y}:=\{f\in\H: f(r)=0\ \textrm{if}\ r\not\in Py\}=\ker (1-V_{y}^{*}V_{y})$$
of $\H$, which is actually spanned by the elements
$$\{\varepsilon_{sy}\otimes \overline{\pi_{0}}(\overline{\alpha}_{sy}(1))h: s\in P, h\in H\}.$$
Now, if $Px\cap Py=\emptyset$, then for every $f\in \H$,
\begin{eqnarray*}
\begin{array}{rcl}
[V_{x}V_{y}^{*}f](r)&=&[V_{x}(V_{y}^{*}f)](r)\\
&=&(V_{y}^{*}f)(rx)=0.
\end{array}
\end{eqnarray*}
This is due to the fact that, since $rx\in Px$, $rx\not\in Py$. So, it follows that $V_{x}V_{y}^{*}=0$, and hence,
$$V_{x}^{*}V_{x}V_{y}^{*}V_{y}=0.$$
If $Px\cap Py=Pz$, for every $f\in \H$,
\[
\big(V_{x}^{*}V_{x}(V_{y}^{*}V_{y}f)\big)(r)=
   \begin{cases}
      (V_{y}^{*}V_{y}f)(r) &\textrm{if}\empty\ \text{$r\in Px$,}\\
      0 &\textrm{otherwise.}\\
   \end{cases}
\]
Moreover, for $(V_{y}^{*}V_{y}f)(r)$, where $r\in Px$, we have
\[
(V_{y}^{*}V_{y}f)(r)=
   \begin{cases}
      f(r) &\textrm{if}\empty\ \text{$r\in Py$,}\\
      0 &\textrm{otherwise.}\\
   \end{cases}
\]
So, it follows that
\[
\big(V_{x}^{*}V_{x}(V_{y}^{*}V_{y}f)\big)(r)=
   \begin{cases}
      f(r) &\textrm{if}\empty\ \text{$r\in (Px\cap Py)=Pz$,}\\
      0 &\textrm{otherwise,}\\
   \end{cases}
\]
which is equal to $(V_{z}^{*}V_{z}f)(r)$. Therefore,
$$V_{x}^{*}V_{x}V_{y}^{*}V_{y}=V_{z}^{*}V_{z}.$$
Consequently, the pair $(\rho,V)$ is a (nontrivial) covariant partial-isometric representation of $(A,P,\alpha)$ on $\H$.

Note that, if $\pi_{0}$ is faithful, then it is not difficult to see that $\rho$ becomes faithful. Hence, every system
$(A,P,\alpha)$ has a (nontrivial) covariant pair $(\rho,V)$ with $\rho$ faithful.
\end{example}

\subsection{Crossed products and Nica-Teoplitz algebras of Hilbert bimodules}
\label{subsec:Piso-CP-NT-alg}
Let $P$ be a left LCM semigroup, and $(A,P,\alpha)$ a dynamical system consisting
of a $C^*$-algebra $A$, and an action $\alpha:P\rightarrow \End (A)$ of $P$ by extendible endomorphisms of $A$ such that
$\alpha_{e}=\id_{A}$.

\begin{definition}
\label{NT-CP-df}
A \emph{partial-isometric crossed product} of $(A,P,\alpha)$ is a triple $(B,i_{A},i_{P})$ consisting of a $C^*$-algebra $B$,
a nondegenerate injective homomorphism $i_{A}:A\rightarrow B$, and a Nica partial-isometric representation $i_{P}:P\rightarrow \M(B)$ such that:
\begin{itemize}
\item[(i)] the pair $(i_{A}, i_{P})$ is a covariant partial-isometric representation of $(A,P,\alpha)$ in $B$;
\item[(ii)] for every covariant partial-isometric representation $(\pi,V)$ of $(A,P,\alpha)$ on a Hilbert space $H$,
there exists a nondegenerate representation $\pi\times V: B\rightarrow B(H)$ such that $(\pi\times V) \circ i_{A}=\pi$
 and $(\overline{\pi\times V}) \circ i_{P}=V$; and
\item[(iii)] the $C^*$-algebra $B$ is generated by $\{i_{A}(a)i_{P}(x): a\in A, x\in P\}$.
\end{itemize}

We call the algebra $B$ the \emph{partial-isometric crossed product} of the system $(A,P,\alpha)$ and denote
it by $A\times_{\alpha}^{\piso} P$.
\end{definition}

\begin{remark}
\label{rmk-1}
Note that in the definition above, for part (iii), we actually have
\begin{align}
\label{span-B}
B=\overline{\newspan}\{i_{P}(x)^{*} i_{A}(a) i_{P}(y) : x,y \in P, a\in A\}.
\end{align}
To see this, we only need to show that the right hand side of $(\ref{span-B})$ is closed under multiplication. To do so, we apply
the Nica covariance condition to calculate each product
\begin{align}
\label{prod-1}
[i_{P}(x)^{*} i_{A}(a) i_{P}(y)][i_{P}(s)^{*} i_{A}(b) i_{P}(t)].
\end{align}
We have
\begin{eqnarray*}
\begin{array}{l}
[i_{P}(x)^{*} i_{A}(a) i_{P}(y)][i_{P}(s)^{*} i_{A}(b) i_{P}(t)]\\
=i_{P}(x)^{*} i_{A}(a) i_{P}(y)[i_{P}(y)^{*}i_{P}(y)i_{P}(s)^{*}i_{P}(s)]i_{P}(s)^{*} i_{A}(b) i_{P}(t),
\end{array}
\end{eqnarray*}
which is zero if $Py\cap Ps=\emptyset$. But if $Py\cap Ps=Pz$ for some $z\in P$, then $ry=z=qs$ for some $r,q\in P$,
and therefore by the covariance of the pair $(i_{A}, i_{P})$, we get
\begin{eqnarray*}
\begin{array}{l}
[i_{P}(x)^{*} i_{A}(a) i_{P}(y)][i_{P}(s)^{*} i_{A}(b) i_{P}(t)]\\
=i_{P}(x)^{*} i_{A}(a) i_{P}(y)i_{P}(z)^{*}i_{P}(z)i_{P}(s)^{*} i_{A}(b) i_{P}(t)\\
=i_{P}(x)^{*} i_{A}(a) i_{P}(y)i_{P}(ry)^{*}i_{P}(qs)i_{P}(s)^{*} i_{A}(b) i_{P}(t)\\
=i_{P}(x)^{*} i_{A}(a) [i_{P}(y)i_{P}(y)^{*}]i_{P}(r)^{*}i_{P}(q)[i_{P}(s)i_{P}(s)^{*}] i_{A}(b) i_{P}(t)\\
=i_{P}(x)^{*} i_{A}(a) \overline{i_{A}}(\overline{\alpha}_{y}(1)) i_{P}(r)^{*}i_{P}(q) \overline{i_{A}}(\overline{\alpha}_{s}(1))
i_{A}(b) i_{P}(t)\\
=i_{P}(x)^{*} i_{A}(a\overline{\alpha}_{y}(1)) i_{P}(r)^{*}i_{P}(q) i_{A}(\overline{\alpha}_{s}(1)b) i_{P}(t)\\
=i_{P}(x)^{*}i_{P}(r)^{*} i_{A}(\alpha_{r}(c))  i_{A}(\alpha_{q}(d)) i_{P}(q) i_{P}(t)\\
=i_{P}(rx)^{*} i_{A}(\alpha_{r}(c)\alpha_{q}(d)) i_{P}(qt),\\
\end{array}
\end{eqnarray*}
which is in the right hand side of $(\ref{span-B})$, where $c=a\overline{\alpha}_{y}(1)$ and $d=\overline{\alpha}_{s}(1)b$.
Thus, $(\ref{span-B})$ is indeed true.
\end{remark}

Next, we want to show that the partial-isometric crossed product of the system $(A,P,\alpha)$ always exists, and it is unique
up to isomorphism. Firstly, since $P$ is a left LCM semigroup, the opposite semigroup $P^{\textrm{o}}$ is a right LCM semigroup.
Therefore, one can easily see that $(A,P^{\textrm{o}},\alpha)$ is a dynamical system in the sense of
\cite[Definition 3.1]{Bar1}. Then, following \cite[\S3]{Fowler} (see also \cite[\S3]{Bar1}), for every $s\in P$, let
$$X_{s}:=\{s\}\times \overline{\alpha}_{s}(1)A,$$
where
$\overline{\alpha}_{s}(1)A=\alpha_{s}(A)A=\clsp\{\alpha_{s}(a)b: a,b\in A\}$
as each endomorphism $\alpha_{s}$ is extendible. Now, each $X_{s}$ is given
the structure of a Hilbert bimodule over $A$ via
$$(s,x)\cdot a:=(s,xa),\ \ \langle (s,x), (s,y)\rangle_{A}:=x^{*}y,$$ and $$a\cdot (s,x):=(s,\alpha_{s}(a)x).$$
Let $X=\bigsqcup_{s\in P} X_{s}$, which is equipped with a multiplication
$$X_{s}\times X_{t}\rightarrow X_{s\bullet t};\ \ ((s,x),(t,y))\mapsto (s,x)(t,y)$$
defined by
$$(s,x)(t,y):=(ts,\alpha_{t}(x)y)=(s\bullet t,\alpha_{t}(x)y)$$
for every $x\in \overline{\alpha}_{s}(1)A$ and $y\in \overline{\alpha}_{t}(1)A$.
By \cite[Lemma 3.2]{Fowler}, $X$ is a product system over the opposite semigroup $P^{\textrm{o}}$ of essential Hilbert bimodules,
and the left action of $A$ on each fiber $X_{s}$ is by compact operators. So, $X$ is compactly aligned
by \cite[Proposition 5.8]{Fowler}. Let $(\mathcal{N}\T(X), i_{X})$ be the Nica-Toeplitz algebra corresponding to $X$ (see \cite{Fowler},
\cite[\S 6]{NB}, and \cite{Bar1}), which is generated by the universal Nica covariant Toeplitz representation
$i_{X}: X\rightarrow \mathcal{N}\T(X)$. We show that this algebra is the partial-isometric crossed product
of the system $(A,P,\alpha)$. But we first need to recall that, for any approximate unit $\{a_{i}\}$ in $A$, by a similar discussion to
\cite[Lemma 3.3]{Fowler}, one can see that $i_{X}(s,\alpha_{s}(a_{i}))$ converges strictly in
the multiplier algebra $\M(\mathcal{N}\T(X))$ for every $s\in P$. Now, we have:

\begin{prop}
\label{CP-exists}
Suppose that $P$ is a left LCM semigroup, and $(A,P,\alpha)$ a dynamical system. Let $\{a_{i}\}$ be any approximate unit in $A$.
Define the maps
$$i_{A}:A\rightarrow \mathcal{N}\T(X)\ \ \textrm{and}\ \ i_{P}:P\rightarrow \M(\mathcal{N}\T(X))$$
by
$$i_{A}(a):=i_{X}(e,a)\ \ \textrm{and}\ \ i_{P}(s):=\lim_{i} i_{X}(s,\alpha_{s}(a_{i}))^{*}\ (\textrm{strictly convergence})$$
for all $a\in A$ and $s\in P$. Then the triple $(\mathcal{N}\T(X),i_{A},i_{P})$ is a partial-isometric crossed product
for $(A,P,\alpha)$, which is unique up to isomorphism.
\end{prop}

\begin{proof}
For any approximate unit $\{a_{i}\}$ in $A$,
$$i_{A}(a_{i})=i_{X}(e,a_{i})=i_{X}(e,\alpha_{e}(a_{i}))$$
converges strictly to $1$ in the multiplier algebra $\M(\mathcal{N}\T(X))$. One can see this again by \cite[Lemma 3.3]{Fowler}
similarly when $s=e$. It follows that $i_{A}$ is a nondegenerate homomorphism. By a similar discussion to the first part
of the proof of \cite[Proposition 3.4]{Fowler}, we can see that the map $i_{P}$ is a partial-isometric representation such
that together with the (nondegenerate) homomorphism $i_{A}$ satisfy the covariance equations
$$i_{A}(\alpha_{s}(a))=i_{P}(s)i_{A}(a)i_{P}(s)^{*}\ \ \textrm{and}\ \ i_{A}(a)i_{P}(s)^{*}i_{P}(s)=i_{P}(s)^{*}i_{P}(s)i_{A}(a)$$
for all $a\in A$ and $s\in P$. So, we only need to show that the representation $i_{P}$ satisfies the Nica covariance condition.
By the same calculation as (3.7) in the proof of \cite[Proposition 3.4]{Fowler}, we have
\begin{align}
\label{eq6}
i_{A}(ab^{*})i_{P}(s)^{*}i_{P}(s)=i_{X}(s,\alpha_{s}(a))i_{X}(s,\alpha_{s}(b))^{*}
\end{align}
for all $a,b\in A$ and $s\in P$, and since
$$i_{X}(s,\alpha_{s}(a))i_{X}(s,\alpha_{s}(b))^{*}=i_{X}^{(s)}(\Theta_{(s,\alpha_{s}(a)),(s,\alpha_{s}(b))}),$$
it follows that
\begin{align}
\label{eq7}
i_{A}(ab^{*})i_{P}(s)^{*}i_{P}(s)=i_{X}^{(s)}(\Theta_{(s,\alpha_{s}(a)),(s,\alpha_{s}(b))}).
\end{align}
Therefore,
$$i_{A}(ab^{*})i_{P}(s)^{*}i_{P}(s)i_{A}(cd^{*})i_{P}(t)^{*}i_{P}(t)=
i_{X}^{(s)}(\Theta_{(s,\alpha_{s}(a)),(s,\alpha_{s}(b))}) i_{X}^{(t)}(\Theta_{(t,\alpha_{t}(c)),(t,\alpha_{t}(d))}),$$
and since
\begin{eqnarray*}
\begin{array}{rcl}
i_{A}(ab^{*})i_{P}(s)^{*}i_{P}(s)i_{A}(cd^{*})i_{P}(t)^{*}i_{P}(t)&=&
i_{A}(ab^{*})i_{A}(cd^{*})i_{P}(s)^{*}i_{P}(s)i_{P}(t)^{*}i_{P}(t)\\
&=&i_{A}(ab^{*}(cd^{*}))i_{P}(s)^{*}i_{P}(s)i_{P}(t)^{*}i_{P}(t),
\end{array}
\end{eqnarray*}
it follows that
\begin{align}
\label{eq8}
i_{A}(ab^{*}(cd^{*}))i_{P}(s)^{*}i_{P}(s)i_{P}(t)^{*}i_{P}(t)=
i_{X}^{(s)}(\Theta_{(s,\alpha_{s}(a)),(s,\alpha_{s}(b))}) i_{X}^{(t)}(\Theta_{(t,\alpha_{t}(c)),(t,\alpha_{t}(d))}).
\end{align}
Now, if $Ps\cap Pt=Pr$ for some $r\in P$, which is equivalent to saying that
$$s\bullet P^{o}\cap t\bullet P^{\textrm{o}}=r\bullet P^{\textrm{o}},$$
since $i_{X}$ is Nica covariant, we have
\begin{align}
\label{eq9}
i_{A}(ab^{*}(cd^{*}))i_{P}(s)^{*}i_{P}(s)i_{P}(t)^{*}i_{P}(t)=
i_{X}^{(r)}\bigg(\iota_{s}^{r}\big(\Theta_{(s,\alpha_{s}(a)),(s,\alpha_{s}(b))}\big) \iota_{t}^{r}\big(\Theta_{(t,\alpha_{t}(c)),(t,\alpha_{t}(d))}\big)\bigg).
\end{align}
Next, we want to calculate the product
$$\iota_{s}^{r}\big(\Theta_{(s,\alpha_{s}(a)),(s,\alpha_{s}(b))}\big) \iota_{t}^{r}\big(\Theta_{(t,\alpha_{t}(c)),(t,\alpha_{t}(d))}\big)$$
of compact operators in $\K(X_{r})$ to show that it is equal to
$$\Theta_{(r,\alpha_{r}(ab^{*})),(r,\alpha_{r}(dc^{*}))}.$$
Since $s\bullet p=r=t\bullet q$ for some $p,q\in P$, and $X_{t}\otimes_{A}X_{q}\simeq X_{t\bullet q}$, it is enough to see this on the
spanning elements $(t,\overline{\alpha}_{t}(1)f)(q,\overline{\alpha}_{q}(1)g)$ of $X_{t\bullet q}=X_{r}$,
where $f,g\in A$. First, $(t,\overline{\alpha}_{t}(1)f)(q,\overline{\alpha}_{q}(1)g)$ by
$\iota_{t}^{r}(\Theta_{(t,\alpha_{t}(c)),(t,\alpha_{t}(d))})$ is mapped to
\begin{eqnarray*}
\begin{array}{l}
\iota_{t}^{t\bullet q}\big(\Theta_{(t,\alpha_{t}(c)),(t,\alpha_{t}(d))}\big)
\big((t,\overline{\alpha}_{t}(1)f)(q,\overline{\alpha}_{q}(1)g)\big)\\
=\big(\Theta_{(t,\alpha_{t}(c)),(t,\alpha_{t}(d))}(t,\overline{\alpha}_{t}(1)f)\big)(q,\overline{\alpha}_{q}(1)g)\\
=\big((t,\alpha_{t}(c))\cdot\big\langle (t,\alpha_{t}(d)),(t,\overline{\alpha}_{t}(1)f) \big\rangle_{A} \big)
(q,\overline{\alpha}_{q}(1)g)\\
=\big((t,\alpha_{t}(c))\cdot(\alpha_{t}(d^{*})\overline{\alpha}_{t}(1)f)\big) (q,\overline{\alpha}_{q}(1)g)\\
=\big((t,\alpha_{t}(c))\cdot(\alpha_{t}(d^{*})f)\big) (q,\overline{\alpha}_{q}(1)g)\\
=\big((t,\alpha_{t}(c)\alpha_{t}(d^{*})f)\big) (q,\overline{\alpha}_{q}(1)g)\\
=(t,\alpha_{t}(cd^{*})f) (q,\overline{\alpha}_{q}(1)g)\\
=(qt,\alpha_{q}(\alpha_{t}(cd^{*})f)\overline{\alpha}_{q}(1)g)\\
=(qt,\alpha_{q}(\alpha_{t}(cd^{*})f)g)\\
=(qt,\alpha_{qt}(cd^{*})\alpha_{q}(f)g)\\
=(t\bullet q,\alpha_{t\bullet q}(cd^{*})\alpha_{q}(f)g)\\
=(r,\alpha_{r}(cd^{*})\alpha_{q}(f)g)\\
=(s\bullet p,\alpha_{s\bullet p}(cd^{*})\alpha_{q}(f)g)\\
=(ps,\alpha_{ps}(cd^{*})\alpha_{q}(f)g)=(s,\alpha_{s}(cd^{*}))(p,\overline{\alpha}_{p}(1)\alpha_{q}(f)g).\\
\end{array}
\end{eqnarray*}
We then let $\iota_{s}^{r}(\Theta_{(s,\alpha_{s}(a)),(s,\alpha_{s}(b))})$ act on
$(s,\alpha_{s}(cd^{*}))(p,\overline{\alpha}_{p}(1)\alpha_{q}(f)g)$, and hence,
\begin{eqnarray*}
\begin{array}{l}
\iota_{s}^{s\bullet p}\big(\Theta_{(s,\alpha_{s}(a)),(s,\alpha_{s}(b))}\big)
\big((s,\alpha_{s}(cd^{*}))(p,\overline{\alpha}_{p}(1)\alpha_{q}(f)g)\big)\\
=\big(\Theta_{(s,\alpha_{s}(a)),(s,\alpha_{s}(b))}(s,\alpha_{s}(cd^{*}))\big)(p,\overline{\alpha}_{p}(1)\alpha_{q}(f)g)\\
=\big((s,\alpha_{s}(a))\cdot\big\langle (s,\alpha_{s}(b)),(s,\alpha_{s}(cd^{*})) \big\rangle_{A} \big)
(p,\overline{\alpha}_{p}(1)\alpha_{q}(f)g)\\
=\big((s,\alpha_{s}(a))\cdot[\alpha_{s}(b^{*})\alpha_{s}(cd^{*})]\big) (p,\overline{\alpha}_{p}(1)\alpha_{q}(f)g)\\
=\big((s,\alpha_{s}(a))\cdot[\alpha_{s}(b^{*}cd^{*})]\big) (p,\overline{\alpha}_{p}(1)\alpha_{q}(f)g)\\
=(s,\alpha_{s}(a)\alpha_{s}(b^{*}cd^{*})) (p,\overline{\alpha}_{p}(1)\alpha_{q}(f)g)\\
=(s,\alpha_{s}(ab^{*}cd^{*})) (p,\overline{\alpha}_{p}(1)\alpha_{q}(f)g)\\
=(ps,\alpha_{p}(\alpha_{s}(ab^{*}cd^{*}))\overline{\alpha}_{p}(1)\alpha_{q}(f)g)\\
=(ps,\alpha_{p}(\alpha_{s}(ab^{*}cd^{*}))\alpha_{q}(f)g)\\
=(ps,\alpha_{ps}(ab^{*}cd^{*})\alpha_{q}(f)g)\\
=(s\bullet p,\alpha_{s\bullet p}(ab^{*}cd^{*})\alpha_{q}(f)g)=(r,\alpha_{r}(ab^{*}cd^{*})\alpha_{q}(f)g).\\
\end{array}
\end{eqnarray*}
Thus, it follows that
\begin{align}
\label{eq10}
\iota_{s}^{r}\big(\Theta_{(s,\alpha_{s}(a)),(s,\alpha_{s}(b))}\big) \iota_{t}^{r}\big(\Theta_{(t,\alpha_{t}(c)),(t,\alpha_{t}(d))}\big)
\big((t,\overline{\alpha}_{t}(1)f)(q,\overline{\alpha}_{q}(1)g)\big)=(r,\alpha_{r}(ab^{*}cd^{*})\alpha_{q}(f)g).
\end{align}
On the other hand, since
\begin{eqnarray*}
\begin{array}{rcl}
(t,\overline{\alpha}_{t}(1)f)(q,\overline{\alpha}_{q}(1)g)&=&(qt,\alpha_{q}(\overline{\alpha}_{t}(1)f)\overline{\alpha}_{q}(1)g)\\
&=&(qt,\alpha_{q}(\overline{\alpha}_{t}(1)f)g)\\
&=&(qt,\overline{\alpha}_{q}(\overline{\alpha}_{t}(1))\alpha_{q}(f)g)\\
&=&(qt,\overline{\alpha}_{qt}(1)\alpha_{q}(f)g)\\
&=&(r,\overline{\alpha}_{r}(1)\alpha_{q}(f)g),
\end{array}
\end{eqnarray*}
we have
\begin{eqnarray}
\label{eq11}
\begin{array}{l}
\Theta_{(r,\alpha_{r}(ab^{*})),(r,\alpha_{r}(dc^{*}))}\big((t,\overline{\alpha}_{t}(1)f)(q,\overline{\alpha}_{q}(1)g)\big)\\
=\Theta_{(r,\alpha_{r}(ab^{*})),(r,\alpha_{r}(dc^{*}))}(r,\overline{\alpha}_{r}(1)\alpha_{q}(f)g)\\
=(r,\alpha_{r}(ab^{*}))\cdot\big\langle (r,\alpha_{r}(dc^{*})),(r,\overline{\alpha}_{r}(1)\alpha_{q}(f)g) \big\rangle_{A}\\
=(r,\alpha_{r}(ab^{*}))\cdot[\alpha_{r}(dc^{*})^{*}\overline{\alpha}_{r}(1)\alpha_{q}(f)g]\\
=(r,\alpha_{r}(ab^{*}))\cdot[\alpha_{r}(cd^{*})\overline{\alpha}_{r}(1)\alpha_{q}(f)g]\\
=(r,\alpha_{r}(ab^{*}))\cdot[\alpha_{r}(cd^{*})\alpha_{q}(f)g]\\
=(r,\alpha_{r}(ab^{*})\alpha_{r}(cd^{*})\alpha_{q}(f)g)\\
=(r,\alpha_{r}(ab^{*}cd^{*})\alpha_{q}(f)g).
\end{array}
\end{eqnarray}
So, we conclude by (\ref{eq10}) and (\ref{eq11}) that
\begin{align}
\label{eq12}
\iota_{s}^{r}\big(\Theta_{(s,\alpha_{s}(a)),(s,\alpha_{s}(b))}\big) \iota_{t}^{r}\big(\Theta_{(t,\alpha_{t}(c)),(t,\alpha_{t}(d))}\big)
=\Theta_{(r,\alpha_{r}(ab^{*})),(r,\alpha_{r}(dc^{*}))}.
\end{align}
Consequently, if $Ps\cap Pt=Pr$, then by applying
(\ref{eq12}), (\ref{eq9}), and (\ref{eq7}), we get
\begin{eqnarray*}
\begin{array}{rcl}
i_{A}(ab^{*}(cd^{*}))i_{P}(s)^{*}i_{P}(s)i_{P}(t)^{*}i_{P}(t)&=&
i_{X}^{(r)}\bigg(\iota_{s}^{r}\big(\Theta_{(s,\alpha_{s}(a)),(s,\alpha_{s}(b))}\big) \iota_{t}^{r}\big(\Theta_{(t,\alpha_{t}(c)),(t,\alpha_{t}(d))}\big)\bigg)\\
&=&i_{X}^{(r)}\big(\Theta_{(r,\alpha_{r}(ab^{*})),(r,\alpha_{r}(dc^{*}))}\big)\\
&=&i_{A}(ab^{*}(dc^{*})^{*})i_{P}(r)^{*}i_{P}(r)\\
&=&i_{A}(ab^{*}(cd^{*}))i_{P}(r)^{*}i_{P}(r).
\end{array}
\end{eqnarray*}
We therefore have
\begin{align}
\label{eq13}
i_{A}(ab^{*}cd^{*})i_{P}(s)^{*}i_{P}(s)i_{P}(t)^{*}i_{P}(t)=i_{A}(ab^{*}cd^{*})i_{P}(r)^{*}i_{P}(r)
\end{align}
for all $a,b,c,d\in A$. Since $A$ contains an approximate unit, it follows by (\ref{eq13}) that we must have
$$i_{P}(s)^{*}i_{P}(s)i_{P}(t)^{*}i_{P}(t)=i_{P}(r)^{*}i_{P}(r)$$
when $Ps\cap Pt=Pr$. If $Ps\cap Pt=\emptyset$, then again, since $i_{X}$ is Nica covariant, the right hand side of (\ref{eq8})
is zero, and therefore,
\begin{align}
\label{eq16}
i_{A}(ab^{*}(cd^{*}))i_{P}(s)^{*}i_{P}(s)i_{P}(t)^{*}i_{P}(t)=0
\end{align}
for all $a,b,c,d\in A$. Thus, similar to the above, as $A$ contains an approximate unit, we conclude that
$$i_{P}(s)^{*}i_{P}(s)i_{P}(t)^{*}i_{P}(t)=0.$$
So, the pair $(i_{A}, i_{P})$ is a covariant partial-isometric representation of $(A,P,\alpha)$ in
the algebra $\mathcal{N}\T(X)$, and therefore, condition (i) in Definition \ref{NT-CP-df} is satisfied.

Next, suppose that $(\pi,V)$ is a covariant partial-isometric representation of $(A,P,\alpha)$ on a Hilbert space $H$. Then, the pair
$(\pi,V^{*})$ is a representation of the system $(A,P^{\textrm{o}},\alpha)$ in the sense of \cite[Definition 3.2]{Bar1},
which is Nica covariant. Note that the semigroup homomorphism $V^{*}: P^{\textrm{o}}\rightarrow B(H)$ is defined by $s\mapsto V_{s}^{*}$.
Therefore, by \cite[Proposition 3.11]{Bar1}, the map $\psi:X\rightarrow B(H)$ defined by
$$\psi(s,x):=V_{s}^{*}\pi(x)$$
is a nondegenerate Nica covariant Toeplitz representation of $X$ on $H$ (see also \cite[Proposition 9.2]{Fowler}).
So, there is a homomorphism $\psi_{*}:\mathcal{N}\T(X)\rightarrow B(H)$ such that $\psi_{*}\circ i_{X}=\psi$ (see \cite{Fowler,NB}),
which is nondegenerate. Let $\pi\times V=\psi_{*}$. Then
$$(\pi\times V)(i_{A}(a))=\psi_{*}(i_{X}(e,a))=\psi(e,a)=V_{e}^{*}\pi(a)=\pi(a)$$
for all $a\in A$. Also, since $\pi\times V$ is nondegenerate, we have
\begin{eqnarray*}
\begin{array}{rcl}
\overline{(\pi\times V)}(i_{P}(s))&=&\overline{(\pi\times V)}\big(\lim_{i} i_{X}(s,\alpha_{s}(a_{i}))^{*}\big)\\
&=&\lim_{i} (\pi\times V)\big(i_{X}(s,\alpha_{s}(a_{i}))^{*}\big)\\
&=&\lim_{i} \psi_{*}\big(i_{X}(s,\alpha_{s}(a_{i}))\big)^{*}\\
&=&\lim_{i} \psi(s,\alpha_{s}(a_{i}))^{*}\\
&=&\lim_{i} [V_{s}^{*}\pi(\alpha_{s}(a_{i}))]^{*}\\
&=&\lim_{i} [\pi(a_{i})V_{s}^{*}]^{*}\ \ \ (\textrm{by the covariance of}\ (\pi,V))\\
&=&\lim_{i} V_{s}\pi(a_{i})=V_{s}
\end{array}
\end{eqnarray*}
for all $s\in P$. Thus, condition (ii) in Definition \ref{NT-CP-df} is satisfied, too.

Finally, condition (iii) also holds as the elements of the form
$i_{X}(s,\overline{\alpha}_{s}(1)a^{*})^{*}$ generate $\mathcal{N}\T(X)$, and
$$i_{X}(s,\overline{\alpha}_{s}(1)a^{*})^{*}=i_{A}(a)i_{P}(s),$$
which follows by a simple computation.

To see that the homomorphism $i_{A}$ is injective, we recall from Example \ref{exmp-piso-rep} that the
system $(A,P,\alpha)$ admits a (nontrivial) covariant partial-isometric representation $(\pi,V)$ with $\pi$ faithful. Therefore,
it follows from the equation $(\pi\times V)\circ i_{A}=\pi$ that $i_{A}$ must be injective.

For uniqueness, suppose that $(C,j_{A},j_{P})$ is another triple which satisfies conditions (i)-(iii) in
Definition \ref{NT-CP-df}. Then, by applying the universal properties (condition (ii)) of the
algebras $C$ and $\mathcal{N}\T(X)$, once can see that there is an isomorphism of
$C$ onto $\mathcal{N}\T(X)$ which maps the pair $(j_{A},j_{P})$ to the pair $(i_{A},i_{P})$.
\end{proof}

\begin{remark}
\label{faithful-rep}
Recall that when $P$ is the positive cone of an abelian lattice-ordered group $G$, by \cite[Theorem 9.3]{Fowler},
a covariant partial-isometric representation $(\pi, V)$ of $(A,P,\alpha)$ on a Hilbert space $H$ induces a
faithful representation $\pi\times V$ of $A\times_{\alpha}^{\piso} P$ if and only if, for every
finite subset $F=\{x_{1}, x_{2}, ..., x_{n}\}$ of $P\backslash \{e\}$, $\pi$ is faithful on the range of
$$\prod_{i=1}^{n}(1-V_{x_{i}}^{*}V_{x_{i}}).$$
Also, note that, by \cite[Theorem 3.13]{Bar1}, a similar necessary and sufficient condition for the faithfulness of the
representation $\pi\times V$ of $A\times_{\alpha}^{\piso} P$ can be obtained for more general semigroups $P$, namely, LCM semigroups.
(see also \cite[Theorem 3.2]{Flet}).
\end{remark}

Suppose that $(A,P,\alpha)$ is a dynamical system, and $I$ is an ideal of $A$ such that
$\alpha_{s}(I)\subset I$ for all $s\in P$. To define a crossed product $I\times_{\alpha}^{\piso} P$ which we want it to sit naturally in
$A\times_{\alpha}^{\piso} P$ as an ideal, we need some extra condition. So, we need to recall
a definition from \cite{Ad-TH}. Let $\alpha$ be an extendible endomorphism of a $C^{*}$-algebra $A$, and $I$ an ideal of $A$.
Suppose that $\psi:A\rightarrow \M(I)$ is the canonical nondegenerate homomorphism defined
by $\psi(a)i=ai$ for all $a\in A$ and $i\in I$. Then, we say $I$ is
\emph{extendible $\alpha$-invariant} if it is $\alpha$-invariant, which means that $\alpha(I)\subset I$, and
the endomorphism $\alpha|_{I}$ is extendible, such that
$$\alpha(u_{\lambda})\rightarrow \overline{\psi}(\overline{\alpha}(1_{\M(A)}))$$
strictly in $\M(I)$, where $\{u_{\lambda}\}$ is an approximate unit in $I$.

In addition, if $(A,P,\alpha)$ is a dynamical system and $I$ is an ideal of $A$, then there is a dynamical system
$(A/I,P,\tilde{\alpha})$ with extendible endomorphisms given by $\tilde{\alpha}_{s}(a+I)=\alpha_{s}(a)+I$ for
every $a\in A$ and $s\in P$ (see again \cite{Ad-TH}).

The following theorem is actually a generalization of \cite[Theorem 3.1]{AZ2}:

\begin{theorem}
\label{piso-ext-seq}
Let $(A\times_{\alpha}^{\piso} P,i_{A},V)$ be the partial-isometric crossed product of a dynamical system $(A,P,\alpha)$,
and $I$ an extendible $\alpha_{x}$-invariant ideal of $A$ for every $x\in P$. Then, there is a short exact sequence
\begin{align}
\label{exseq2}
0 \longrightarrow I\times_{\alpha}^{\piso} P \stackrel{\mu}{\longrightarrow} A\times_{\alpha}^{\piso} P
\stackrel{\varphi}{\longrightarrow} A/I\times_{\tilde{\alpha}}^{\piso} P \longrightarrow 0
\end{align}
of $C^*$-algebras, where $\mu$ is an isomorphism of $I\times_{\alpha}^{\piso} P$ onto the ideal
$$\mathcal{E}:=\clsp\{V_{s}^{*} i_{A}(i) V_{t} : i\in I, s,t \in P\}$$
\end{theorem}
of $A\times_{\alpha}^{\piso} P$. If $q:A\rightarrow A/I$ is the quotient map, and the triples $(I\times_{\alpha}^{\piso} P,i_{I},W)$
and $(A/I\times_{\tilde{\alpha}}^{\piso} P,i_{A/I},U)$ are the crossed products of the systems $(I,P,\alpha)$ and
$(A/I,P,\tilde{\alpha})$, respectively, then
$$\mu\circ i_{I}=i_{A}|_{I},\ \overline{\mu}\circ W=V\ \ \textrm{and}\ \ \varphi\circ i_{A}=i_{A/I}\circ q,\ \overline{\varphi}\circ V=U.$$

\begin{proof}
We first show that $\mathcal{E}$ is an ideal of $A\times_{\alpha}^{\piso} P$. To do so, it suffices to see on
the spanning elements of $\mathcal{E}$ that $V_{r}^{*}\mathcal{E}$, $i_{A}(a)\mathcal{E}$, and $V_{r}\mathcal{E}$
are all contained in $\mathcal{E}$ for every $a\in A$ and $r\in P$. This first one is obvious, and the second
one follows easily by applying the covariance equation
$i_{A}(a)V_{s}^{*}=V_{s}^{*}i_{A}(\alpha_{s}(a))$. For the third one, we have
$$V_{r}V_{s}^{*} i_{A}(i) V_{t}=V_{r}[V_{r}^{*}V_{r}V_{s}^{*}V_{s}]V_{s}^{*} i_{A}(i) V_{t},$$
which is zero if $Pr\cap Ps=\emptyset$. But if $Pr\cap Ps=Pz$ for some $z\in P$, then there are $x,y\in P$ such that
$xr=z=ys$, and therefore it follows that
\begin{eqnarray*}
\begin{array}{rcl}
V_{r}V_{s}^{*} i_{A}(i) V_{t}&=&V_{r}[V_{z}^{*}V_{z}]V_{s}^{*} i_{A}(i) V_{t}\\
&=&V_{r}V_{xr}^{*}V_{ys}V_{s}^{*} i_{A}(i) V_{t}\\
&=&V_{r}[V_{x}V_{r}]^{*}V_{y}V_{s}V_{s}^{*} i_{A}(i) V_{t}\\
&=&[V_{r}V_{r}^{*}]V_{x}^{*}V_{y}[V_{s}V_{s}^{*}] i_{A}(i) V_{t}\\
&=&\overline{i_{A}}(\overline{\alpha}_{r}(1)) V_{x}^{*}V_{y} \overline{i_{A}}(\overline{\alpha}_{s}(1)) i_{A}(i) V_{t}
\ \ \ (\textrm{by Lemma (\ref{cov2-lemma})})\\
&=&V_{x}^{*} \overline{i_{A}}(\overline{\alpha}_{x}(\overline{\alpha}_{r}(1))) V_{y} i_{A}(\overline{\alpha}_{s}(1)i) V_{t}
\ \ \ (\textrm{by Lemma (\ref{cov2-lemma})})\\
&=&V_{x}^{*} \overline{i_{A}}(\overline{\alpha}_{xr}(1)) i_{A}(\alpha_{y}(\overline{\alpha}_{s}(1)i)) V_{y}V_{t}\\
&=&V_{x}^{*} \overline{i_{A}}(\overline{\alpha}_{z}(1)) i_{A}(\overline{\alpha}_{y}(\overline{\alpha}_{s}(1))\alpha_{y}(i))V_{yt}\\
&=&V_{x}^{*} \overline{i_{A}}(\overline{\alpha}_{z}(1)) i_{A}(\overline{\alpha}_{ys}(1)\alpha_{y}(i)) V_{yt}\\
&=&V_{x}^{*} \overline{i_{A}}(\overline{\alpha}_{z}(1)) i_{A}(\overline{\alpha}_{z}(1)\alpha_{y}(i)) V_{yt}\\
&=&V_{x}^{*} i_{A}(\overline{\alpha}_{z}(1)\overline{\alpha}_{z}(1)\alpha_{y}(i)) V_{yt}\\
&=&V_{x}^{*} i_{A}(\overline{\alpha}_{z}(1)\alpha_{y}(i)) V_{yt},
\end{array}
\end{eqnarray*}
which belongs to $\mathcal{E}$. Thus, $\mathcal{E}$ is an ideal of $A\times_{\alpha}^{\piso} P$. Let
$\phi:A\times_{\alpha}^{\piso} P\rightarrow \M(\mathcal{E})$ be the canonical nondegenerate homomorphism defined by
$\phi(\xi)\eta=\xi\eta$ for all $\xi\in A\times_{\alpha}^{\piso} P$ and $\eta\in \mathcal{E}$. Suppose that now
the maps
$$k_{I}:I\rightarrow \M(\mathcal{E})\ \ \textrm{and}\ \ S:P\rightarrow \M(\mathcal{E})$$
are defined by the compositions
$$I\stackrel{i_{A}|_{I}}{\longrightarrow}A\times_{\alpha}^{\piso} P\stackrel{\phi}{\longrightarrow}\M(\mathcal{E})\ \ \textrm{and}
\ \ P\stackrel{V}{\longrightarrow}\M(A\times_{\alpha}^{\piso} P)\stackrel{\overline{\phi}}{\longrightarrow}\M(\mathcal{E}),$$
respectively. We claim that the triple $(\mathcal{E},k_{I},S)$ is a partial-isometric crossed product of the system
$(I,P,\alpha)$. First, exactly by the same discussion as in the
proof of \cite[Theorem 3.1]{AZ2} using the extendibility of the ideal $I$, it follows that the homomorphism $k_{I}$ is nondegenerate.
Also, it follows easily by the definition of the map $S$ that it is indeed a Nica partial-isometric
representation. Then, by some routine calculations, one can see that the pair $(k_{I},S)$ satisfies
the covariance equations
$$k_{I}(\alpha_{t}(i))=S_{t}k_{I}(i)S_{t}^{*}\ \ \textrm{and}\ \ S_{t}^{*}S_{t}k_{I}(i)=k_{I}(i)S_{t}^{*}S_{t}$$
for all $i\in I$ and $t\in P$.

Next, suppose that the pair $(\pi,T)$ is a covariant partial-isometric representation of $(I,P,\alpha)$ on a Hilbert space $H$.
Let $\psi:A\rightarrow \M(I)$ be the canonical nondegenerate homomorphism which was mentioned about earlier. Let
the map $\rho:A\rightarrow B(H)$ be defined by the composition
$$A\stackrel{\psi}{\longrightarrow} \M(I)\stackrel{\overline{\pi}}{\longrightarrow} B(H),$$
which is a nondegenerate representation of $A$ on $H$. We claim that the pair $(\rho,T)$ is a covariant partial-isometric
representation of $(A,P,\alpha)$ on $H$. To prove our claim, we only need to show that the pair $(\rho,T)$ satisfies
the covariance equations (\ref{cov1}). Since the ideal $I$ is extendible, we have
$\overline{\alpha_{s}|_{I}}\circ \psi=\psi\circ \alpha_{s}$ for all $s\in P$. It therefore follows that
\begin{eqnarray*}
\begin{array}{rcl}
\rho(\alpha_{s}(a))&=&(\overline{\pi}\circ\psi)(\alpha_{s}(a))\\
&=&\overline{\pi}(\psi\circ\alpha_{s}(a))\\
&=&\overline{\pi}(\overline{\alpha_{s}|_{I}}\circ \psi(a))\\
&=&(\overline{\pi}\circ\overline{\alpha_{s}|_{I}})(\psi(a))\\
&=&T_{s}\overline{\pi}(\psi(a))T_{s}^{*}=T_{s}\rho(a)T_{s}^{*}.\\
\end{array}
\end{eqnarray*}
Also, one can easily see that we have $T_{s}^{*}T_{s}\rho(a)=\rho(a)T_{s}^{*}T_{s}$. Thus, there is a nondegenerate
representation $\rho\times T$ of $A\times_{\alpha}^{\piso} P$ on $H$, whose restriction $(\rho\times T)|_{\mathcal{E}}$ is a
nondegenerate representation of $\mathcal{E}$ on $H$ satisfying
$$(\rho\times T)|_{\mathcal{E}}\circ k_{I}=\pi\ \ \textrm{and}\ \ \overline{(\rho\times T)|_{\mathcal{E}}}\circ S=T.$$

Finally, the elements of the form
\begin{eqnarray*}
\begin{array}{rcl}
S_{s}^{*}k_{I}(i)S_{t}&=&\overline{\phi}(V_{s})^{*}\overline{\phi}(i_{A}(i))\overline{\phi}(V_{t})\\
&=&\phi(V_{s}^{*}i_{A}(i)V_{t})=V_{s}^{*}i_{A}(i)V_{t}
\end{array}
\end{eqnarray*}
obviously span the algebra $\mathcal{E}$. Thus, $(\mathcal{E},k_{I},S)$ is a partial-isometric crossed product of
$(I,P,\alpha)$. So, by Proposition \ref{CP-exists}, there is an isomorphism
$\mu:I\times_{\alpha}^{\piso} P\rightarrow \mathcal{E}$ such that
$$\mu(i_{I}(i)W_{t})=k_{I}(i)S_{t}=\phi(i_{A}|_{I}(i))\overline{\phi}(V_{t})=\phi(i_{A}|_{I}(i)V_{t})=i_{A}|_{I}(i)V_{t},$$
from which, it follows that
$$\mu\circ i_{I}=i_{A}|_{I}\ \ \textrm{and}\ \overline{\mu}\circ W=V.$$

To get the desired homomorphism $\varphi$, let the homomorphism $j_{A}:A\rightarrow A/I\times_{\tilde{\alpha}}^{\piso} P$ be given
by the composition
$$A\stackrel{q}{\longrightarrow} A/I \stackrel{i_{A/I}}{\longrightarrow} A/I\times_{\tilde{\alpha}}^{\piso} P,$$
which is nondegenerate. Then, it is not difficult to see that the pair $(j_{A},U)$ is a covariant partial-isometric
representation of $(A,P,\alpha)$ in the algebra $A/I\times_{\tilde{\alpha}}^{\piso} P$. Thus, there is a nondegenerate
homomorphism $\varphi:=j_{A}\times U:A\times_{\alpha}^{\piso} P\rightarrow A/I\times_{\tilde{\alpha}}^{\piso} P$ such that
$$\varphi\circ i_{A}=j_{A}=i_{A/I}\circ q\ \ \textrm{and}\ \overline{\varphi}\circ V=U,$$
which implies that $\varphi$ is onto.

Finally, we show that $\mu(I\times_{\alpha}^{\piso} P)=\mathcal{E}$ is equal to $\ker \varphi$ which means that
(\ref{exseq2}) is exact. The inclusion $\mathcal{E}\subset \ker \varphi$ is immediate. To see the other inclusion, take
a nondegenerate representation $\Pi$ of $A\times_{\alpha}^{\piso} P$ on a Hilbert space $H$ with $\ker \Pi=\mathcal{E}$.
Since $I\subset \ker (\Pi\circ i_{A})$, the composition $\Pi\circ i_{A}$ gives a (well-defined) nodegenerate representation
$\widetilde{\Pi}$ of $A/I$ on $H$. Also, the composition $\overline{\Pi}\circ V$ defines a Nica partial-isometric representation
$P$ on $H$, such that together with $\widetilde{\Pi}$ forms a covariant partial-isometric representation of $(A/I,P,\tilde{\alpha})$
on $H$. Then the corresponding (nondegenerate) representation $\widetilde{\Pi}\times (\overline{\Pi}\circ V)$
lifts to $\Pi$, which means that $[\widetilde{\Pi}\times (\overline{\Pi}\circ V)]\circ\varphi=\Pi$, from which the inclusion
$\ker \varphi\subset \mathcal{E}$ follows. This completes the proof.

\end{proof}

\begin{example}
\label{C*(P)}
Suppose that $S$ is a unital right LCM semigroup. See in \cite{BNL,Norling} that associated to $S$ there is a universal
$C^{*}$-algebra
$$C^{*}(S)=\clsp\{W_{s}W_{t}^{*}: s,t\in S\}$$
generated by a universal isometric representation $W:S\rightarrow C^{*}(S)$, which is \emph{Nica-covariant}, which means that it
satisfies
\begin{align}\label{Nica-cov-iso}
W_{r}W_{r}^{*}W_{s}W_{s}^{*}=
   \begin{cases}
      W_{t}W_{t}^{*} &\textrm{if}\empty\ \text{$rS\cap sS=tS$,}\\
      0 &\textrm{if}\empty\ \text{$rS\cap sS=\emptyset$.}\\
   \end{cases}
\end{align}
In addition, by \cite[Corollary 7.11]{NB}, $C^{*}(S)$ is isomorphic to the Nica-Toeplitz algebra $\mathcal{N}\T(X)$ of
the compactly aligned product system $X$ over $S$ with fibers $X_{s}=\C$ for all $s\in S$.
Now, consider the trivial dynamical system $(\C,P,\id)$, where $P$ is a left LCM semigroup. So, the opposite semigroup
$P^{\textrm{o}}$ is right LCM. Then, it follows by Proposition \ref{CP-exists} that there is an isomorphism
$$i_{p}(x)\in (\C\times_{\id}^{\piso} P)\mapsto W_{x}^{*}\in C^{*}(P^{\textrm{o}})$$
for all $x\in P$, where $W$ is the universal Nica-covariant isometric representation of
$P^{\textrm{o}}$ which generates $C^{*}(P^{\textrm{o}})$.

For the $C^{*}$-algebra $C^{*}(S)$ associated to any arbitrary left
cancellative semigroup $S$, readers may refer to \cite{X-Li}.

\end{example}

\begin{lemma}
\label{id-action}
For the dynamical system $(A,P,\id)$ in which $P$ is a left LCM semigroup and $\id$ denotes the trivial action, we have
$$(A\times_{\id}^{\piso} P,i) \simeq A\otimes_{\max} C^{*}(P^{\textrm{o}}).$$
The isomorphism maps each (spanning) element $i_{p}(x)^{*}i_{A}(a)i_{p}(y)$ of $A\times_{\id}^{\piso} P$ to $a\otimes W_{x}W_{y}^{*}$, where
$W$ is the universal Nica-covariant isometric representation of
$P^{\textrm{o}}$ which generates $C^{*}(P^{\textrm{o}})$.
\end{lemma}

\begin{proof}
We skip the proof as it follows easily by some routine calculations.
\end{proof}

\section{Tensor products of crossed products}
\label{sec:tensor}
Let $(A,P,\alpha)$ and $(B,S,\beta)$ be dynamical systems in which $P$ and $S$ are left LCM semigroups. Then, $P\times S$ is a unital
semigroup with the unit element $(e_{P},e_{S})$, where $e_{P}$ and $e_{S}$ are the unit elements of $P$ and $S$, respectively.
In addition, since
\begin{eqnarray}
\label{eq17}
\begin{array}{rcl}
(P\times S) (x,r)\cap (P\times S) (y,s)&=&(Px\times Sr) \cap (Py\times Ss)\\
&=&(Px\cap Py) \times (Sr\cap Ss),\\
\end{array}
\end{eqnarray}
it follows that $P\times S$ is a left LCM semigroup. More precisely, if $Px\cap Py=Pz$ and $Sr\cap Ss=St$ for some $z\in P$ and
$t\in S$, then it follows by (\ref{eq17}) that
$$(P\times S) (x,r)\cap (P\times S) (y,s)=Pz \times St=(P\times S)(z,t),$$
which means that $(z,t)$ is a least common left multiple of $(x,r)$ and $(y,s)$ in $P\times S$. Otherwise,
$(P\times S) (x,r)\cap (P\times S) (y,s)=\emptyset$. Thus, $P\times S$ is actually a left LCM semigroup
(note that, the similar fact holds if $P$ and $S$ are right LCM semigroups).

Next, for every $x\in P$ and $r\in S$, as $\alpha_{x}$ and $\beta_{r}$ are endomorphisms of the algebras $A$ and $B$, respectively,
it follows by \cite[Lemma B. 31]{RW} that there is an endomorphism $\alpha_{x}\otimes\beta_{r}$ of the maximal tensor product
$A\otimes_{\max} B$ such that $(\alpha_{x}\otimes\beta_{r})(a\otimes b)=\alpha_{x}(a)\otimes\beta_{r}(b)$ for all $a\in A$ and
$b\in B$. We therefore have an action $$\alpha\otimes \beta:P\times S\rightarrow \End(A\otimes_{\max} B)$$
of $P\times S$ on $A\otimes_{\max} B$ by endomorphisms such that
$$(\alpha\otimes \beta)_{(x,r)}=\alpha_{x}\otimes\beta_{r}\ \ \textrm{for all}\ \ (x,r)\in P\times S.$$
Moreover, it follows by the extendibility of the actions $\alpha$ and $\beta$ that the action $\alpha\otimes \beta$
on $A\otimes_{\max} B$ is actually given by extendible endomorphisms (see \cite[Lemma 2.3]{Larsen}). Thus, we have a
dynamical system $(A\otimes_{\max} B,P\times S,\alpha\otimes \beta)$, for which, we want to study the corresponding
partial-isometric crossed product. We actually aim to show that under some certain conditions we have the following isomorphism:
$$(A\otimes_{\textrm{max}} B)\times_{\alpha\otimes \beta}^{\piso}(P\times S)
\simeq (A\times_{\alpha}^{\piso} P) \otimes_{\textrm{max}} (B\times_{\beta}^{\piso} S).$$
In fact, those conditions are to ensure that the Nica partial-isometric representations of $P$ and $S$
are $*$-commuting. Hence, we first need to assume that the unital semigroups $P$, $P^{\textrm{o}}$, $S$, and
$S^{\textrm{o}}$ are all left LCM. It thus turns out that all of them must actually be both left and right LCM semigroups. The other
condition comes from the following definition:
\begin{definition}
\label{bicov-def}
Suppose that $P$ and $P^{\textrm{o}}$ are both left LCM semigroups. A \emph{bicovariant partial-isometric representation} of
$P$ on a Hilbert space $H$ is a Nica partial-isometric representation $V:P\rightarrow B(H)$ which satisfies
\begin{align}\label{Nica-2}
V_{r}V_{r}^{*}V_{s}V_{s}^{*}=
   \begin{cases}
      V_{t}V_{t}^{*} &\textrm{if}\empty\ \text{$rP\cap sP=tP$,}\\
      0 &\textrm{if}\empty\ \text{$rP\cap sP=\emptyset$.}\\
   \end{cases}
\end{align}
Note the equation (\ref{Nica-2}) is a kind of Nica covariance condition, too. So, to distinguish it from the covariance equation (\ref{Nica-1}),
we view (\ref{Nica-1}) as the \emph{right Nica covariance condition} and (\ref{Nica-2}) as the \emph{left Nica covariance condition}.
\end{definition}
Note that similar to (\ref{Nica-1}), we can see that the equation (\ref{Nica-2}) is also well-defined.

\begin{lemma}
\label{bicov-lem}
Suppose that the unital semigroups $P$, $P^{\textrm{o}}$, $S$, and $S^{\textrm{o}}$ are all left LCM. Let $V$ and $W$ be
bicovariant partial-isometric representations of $P$ and $S$ on a Hilbert space $H$, respectively, such that each $V_{p}$ $*$-commutes with
each $W_{s}$ for all $p\in P$ and $s\in S$. Then, there exits a bicovariant partial-isometric representation $U$ of
$P\times S$ on $H$ such that $U_{(p,s)}=V_{p}W_{s}$. Moreover, every bicovariant partial-isometric representation of $P\times S$
arises this way.
\end{lemma}

\begin{proof}
Define a map $U:P\times S\rightarrow B(H)$ by $$U_{(p,s)}=V_{p}W_{s}$$ for all $(p,s)\in P\times S$.
Since each $V_{p}$ $*$-commutes with each $W_{s}$, it follows that each $U_{(p,s)}$ is a partial isometry, as
\begin{eqnarray*}
\begin{array}{rcl}
U_{(p,s)}U_{(p,s)}^{*}U_{(p,s)}&=&V_{p}W_{s}[V_{p}W_{s}]^{*}V_{p}W_{s}\\
&=&V_{p}W_{s}[W_{s}V_{p}]^{*}V_{p}W_{s}\\
&=&V_{p}W_{s}V_{p}^{*}W_{s}^{*}V_{p}W_{s}\\
&=&V_{p}V_{p}^{*}W_{s}V_{p}W_{s}^{*}W_{s}\\
&=&V_{p}V_{p}^{*}V_{p}W_{s}W_{s}^{*}W_{s}=V_{p}W_{s}=U_{(p,s)}.
\end{array}
\end{eqnarray*}
Also, a simple computation shows that $$U_{(p,s)}U_{(q,t)}=U_{(p,s)(q,t)}$$ for every $(p,s)$ and $(q,t)$ in $P\times S$.
Thus, the map $U$ is a (unital) semigroup homomorphism (with partial-isometric values). Next, we want to show that
it satisfies the Nica covariance conditions (\ref{Nica-1}) and (\ref{Nica-2}), and hence, it is bicovariant. To see (\ref{Nica-1}),
we first have
\begin{eqnarray}
\label{eq18}
\begin{array}{rcl}
U_{(p,s)}^{*}U_{(p,s)}U_{(q,t)}^{*}U_{(q,t)}&=&[V_{p}W_{s}]^{*}V_{p}W_{s}[V_{q}W_{t}]^{*}V_{q}W_{t}\\
&=&[W_{s}V_{p}]^{*}V_{p}W_{s}[W_{t}V_{q}]^{*}V_{q}W_{t}\\
&=&V_{p}^{*}W_{s}^{*}V_{p}W_{s}V_{q}^{*}W_{t}^{*}V_{q}W_{t}\\
&=&V_{p}^{*}V_{p}W_{s}^{*}V_{q}^{*}W_{s}V_{q}W_{t}^{*}W_{t}\\
&=&V_{p}^{*}V_{p}V_{q}^{*}W_{s}^{*}V_{q}W_{s}W_{t}^{*}W_{t}\\
&=&[V_{p}^{*}V_{p}V_{q}^{*}V_{q}] [W_{s}^{*}W_{s}W_{t}^{*}W_{t}].
\end{array}
\end{eqnarray}
If $$(P\times S) (p,s)\cap (P\times S) (q,t)=(P\times S)(z,r)=Pz\times Sr$$ for some $(z,r)\in P\times S$, then it follows
by (\ref{eq17}) (for the left hand side in above) that
$$(Pp\cap Pq) \times (Ss\cap St)=Pz\times Sr.$$
Thus, we must have $$Pp\cap Pq=Pz\ \ \textrm{and}\ \ Ss\cap St=Sr,$$ and hence, for (\ref{eq18}), we get
\begin{eqnarray*}
\begin{array}{rcl}
U_{(p,s)}^{*}U_{(p,s)}U_{(q,t)}^{*}U_{(q,t)}&=&[V_{p}^{*}V_{p}V_{q}^{*}V_{q}] [W_{s}^{*}W_{s}W_{t}^{*}W_{t}]\\
&=&V_{z}^{*}V_{z} W_{r}^{*}W_{r}\\
&=&V_{z}^{*}W_{r}^{*} V_{z} W_{r}\\
&=&[W_{r}V_{z}]^{*} U_{(z,r)}\\
&=&[V_{z}W_{r}]^{*} U_{(z,r)}=U_{(z,r)}^{*} U_{(z,r)}.
\end{array}
\end{eqnarray*}
If $(P\times S) (p,s)\cap (P\times S) (q,t)=\emptyset$, then again, by (\ref{eq17}), we get
$$(Pp\cap Pq) \times (Ss\cap St)=\emptyset.$$ It follows that
$$Pp\cap Pq=\emptyset\ \ \vee\ \ Ss\cap St=\emptyset,$$ which implies that,
$$V_{p}^{*}V_{p}V_{q}^{*}V_{q}=0\ \ \vee\ \ W_{s}^{*}W_{s}W_{t}^{*}W_{t}=0.$$
Thus, for (\ref{eq18}), we have
$$U_{(p,s)}^{*}U_{(p,s)}U_{(q,t)}^{*}U_{(q,t)}=[V_{p}^{*}V_{p}V_{q}^{*}V_{q}] [W_{s}^{*}W_{s}W_{t}^{*}W_{t}]=0.$$

A similar discussion shows that the representation $U$ satisfies the Nica covariance condition (\ref{Nica-2}), too, namely, we have
\[
U_{(p,r)}U_{(p,r)}^{*}U_{(q,s)}U_{(q,s)}^{*}=
   \begin{cases}
      U_{(x,t)}U_{(x,t)}^{*} &\textrm{if}\empty\ \text{$(p,r)(P\times S) \cap (q,s)(P\times S) =(x,t)(P\times S)$,}\\
      0 &\textrm{if}\empty\ \text{$(p,r)(P\times S) \cap (q,s)(P\times S)=\emptyset$.}\\
   \end{cases}
\]
Therefore, $U$ is a bicovariant partial-isometric representation of $P\times S$ on $H$ satisfying $U_{(p,s)}=V_{p}W_{s}$.

Conversely, suppose that $U$ is any bicovariant partial-isometric representation of $P\times S$ on a Hilbert space $H$.
Define the maps
$$V:P\rightarrow B(H)\ \ \textrm{and}\ \ W:S\rightarrow B(H)$$
by
$$V_{p}:=U_{(p,e_{S})}\ \ \textrm{and}\ \ W_s:=U_{(e_{P},s)}$$
for all $p\in P$ and $s\in S$, respectively. It is easy to see that each $V_{p}$ is a partial isometry as well as each $W_s$, and
the maps $V$ and $W$ are (unital) semigroup homomorphisms. Next, we show that the presentation $V$ is bicovariant, and we skip the
proof for the presentation $W$ as it follows similarly. To see that the presentation $V$ satisfies the Nica covariance
condition (\ref{Nica-1}), firstly,
\begin{align}
\label{eq19}
V_{p}^{*}V_{p}V_{q}^{*}V_{q}=U_{(p,e_{S})}^{*}U_{(p,e_{S})}U_{(q,e_{S})}^{*}U_{(q,e_{S})}.
\end{align}
Now, if $Pp\cap Pq=Pz$ for some $z\in P$, then it follows by (\ref{eq17}) that
\begin{eqnarray*}
\begin{array}{rcl}
(P\times S) (p,e_{S})\cap (P\times S) (q,e_{S})&=&(Pp\cap Pq) \times (S\cap S)\\
&=&Pz\times S\\
&=&Pz\times Se_{S}=(P\times S)(z,e_{S}).
\end{array}
\end{eqnarray*}
Therefore, since $U$ is bicovariant, for (\ref{eq19}), we have
\begin{eqnarray*}
\begin{array}{rcl}
V_{p}^{*}V_{p}V_{q}^{*}V_{q}&=&U_{(p,e_{S})}^{*}U_{(p,e_{S})}U_{(q,e_{S})}^{*}U_{(q,e_{S})}\\
&=&U_{(z,e_{S})}^{*}U_{(z,e_{S})}=V_{z}^{*}V_{z}.
\end{array}
\end{eqnarray*}
If $Pp\cap Pq=\emptyset$, then it follows again by (\ref{eq17}) that
$$(P\times S) (p,e_{S})\cap (P\times S) (q,e_{S})=(Pp\cap Pq) \times (S\cap S)=\emptyset \times S=\emptyset.$$
Therefore, for (\ref{eq19}), we get
$$V_{p}^{*}V_{p}V_{q}^{*}V_{q}=U_{(p,e_{S})}^{*}U_{(p,e_{S})}U_{(q,e_{S})}^{*}U_{(q,e_{S})}=0,$$
as $U$ is bicovariant. A similar discussion shows that the representation $V$ satisfies
the Nica covariance condition (\ref{Nica-2}), too. We skip it here. Finally, as we obviously have
\begin{align}
\label{eq20}
V_{p}W_{s}=W_{s}V_{p}=U_{(p,s)},
\end{align}
it is only left to show that $V_{p}^{*}W_{s}=W_{s}V_{p}^{*}$ for all $p\in P$ and $s\in S$. To do so, we first need to recall
that the product $vw$ of two partial isometries $v$ and $w$ is a partial isometry if and only if $v^{*}v$ commutes with
$ww^{*}$ (see \cite[Lemma 2]{Halmos}). This fact can be applied to the partial isometries $V_{p}$ and $W_{s}$ due to (\ref{eq20}).
Now, we have
\begin{eqnarray}
\label{eq21}
\begin{array}{rcl}
V_{p}^{*}W_{s}&=&V_{p}^{*}[V_{p}V_{p}^{*}W_{s}W_{s}^{*}]W_{s}\\
&=&U_{(p,e_{S})}^{*}[U_{(p,e_{S})}U_{(p,e_{S})}^{*}U_{(e_{P},s)}U_{(e_{P},s)}^{*}]U_{(e_{P},s)}.
\end{array}
\end{eqnarray}
Since,
\begin{eqnarray*}
\begin{array}{rcl}
(p,e_{S})(P\times S)\cap (e_{P},s)(P\times S)&=&(pP\times S)\cap (P\times sS)\\
&=&(pP\cap P)\times (S\cap sS)\\
&=&pP\times sS=(p,s)(P\times S),
\end{array}
\end{eqnarray*}
it follows that
\begin{eqnarray}
\label{eq22}
\begin{array}{rcl}
V_{p}V_{p}^{*}W_{s}W_{s}^{*}&=&U_{(p,e_{S})}U_{(p,e_{S})}^{*}U_{(e_{P},s)}U_{(e_{P},s)}^{*}\\
&=&U_{(p,s)}U_{(p,s)}^{*}\\
&=&U_{(p,e_{S})}U_{(e_{P},s)}[U_{(e_{P},s)}U_{(p,e_{S})}]^{*}\\
&=&U_{(p,e_{S})} U_{(e_{P},s)} U_{(p,e_{S})}^{*} U_{(e_{P},s)}^{*}=V_{p} W_{s} V_{p}^{*} W_{s}^{*},
\end{array}
\end{eqnarray}
as $U$ is bicovariant. Therefore, by (\ref{eq21}) and (\ref{eq22}), we get
\begin{eqnarray*}
\begin{array}{rcl}
V_{p}^{*}W_{s}&=&V_{p}^{*}[V_{p}V_{p}^{*}W_{s}W_{s}^{*}]W_{s}\\
&=&V_{p}^{*} [V_{p} W_{s} V_{p}^{*} W_{s}^{*}] W_{s}\\
&=&[V_{p}^{*} V_{p} W_{s}W_{s}^{*}] W_{s} V_{p}^{*} [V_{p}V_{p}^{*} W_{s}^{*} W_{s}]\\
&=&[W_{s}W_{s}^{*} V_{p}^{*}V_{p}] W_{s} V_{p}^{*} [W_{s}^{*} W_{s} V_{p}V_{p}^{*}]\\
&=&W_{s}W_{s}^{*} V_{p}^{*} [V_{p} W_{s} V_{p}^{*} W_{s}^{*}] W_{s} V_{p}V_{p}^{*}]\\
&=&W_{s}W_{s}^{*} V_{p}^{*} [V_{p}V_{p}^{*}W_{s}W_{s}^{*}] W_{s} V_{p}V_{p}^{*}]\ \ \ (\textrm{by (\ref{eq22})})\\
&=&W_{s}W_{s}^{*} V_{p}^{*} W_{s} V_{p}V_{p}^{*}\\
&=&W_{s} [(V_{p} W_{s})^{*} (V_{p} W_{s})] V_{p}^{*}\\
&=&W_{s} [U_{(p,s)}^{*} U_{(p,s)}] V_{p}^{*},
\end{array}
\end{eqnarray*}
and hence,
\begin{align}
\label{eq23}
V_{p}^{*}W_{s}=W_{s} [U_{(p,s)}^{*} U_{(p,s)}] V_{p}^{*}.
\end{align}
Moreover, since similarly
$$(P\times S)(e_{P},s)\cap (P\times S)(p,e_{S})=(P\times S)(p,s),$$
we have
$$U_{(e_{P},s)}^{*} U_{(e_{P},s)}U_{(p,e_{S})}^{*} U_{(p,e_{S})}=U_{(p,s)}^{*} U_{(p,s)},$$
as $U$ is bicovariant. By applying this to (\ref{eq23}), we finally get
\begin{eqnarray*}
\begin{array}{rcl}
V_{p}^{*}W_{s}&=&W_{s} [U_{(e_{P},s)}^{*} U_{(e_{P},s)}U_{(p,e_{S})}^{*} U_{(p,e_{S})}] V_{p}^{*}\\
&=&[W_{s} W_{s}^{*} W_{s}] [V_{p}^{*} V_{p} V_{p}^{*}]\\
&=&W_{s}V_{p}^{*}.
\end{array}
\end{eqnarray*}
This completes the proof.
\end{proof}

\begin{remark}
\label{bicov-piso-TOG}
If $P$ is the positive cone of a totally ordered group $G$, then every partial-isometric representation $V$ of $P$ automatically
satisfies the left Nica covariance condition (\ref{Nica-2}), such that
$$V_{x}V_{x}^{*}V_{y}V_{y}^{*}=V_{\max\{x,y\}}V_{\max\{x,y\}}^{*}\ \ \ \textrm{for all}\ x,y\in P.$$
Therefore, every partial-isometric representation of $P$ is automatically bicovariant (see Remark \ref{Nica-piso-TOG}).
\end{remark}

\begin{remark}
\label{bicov-N2}
Recall that a partial isometry $V$ is called a \emph{power partial isometry} if $V^{n}$ is a partial isometry for every $n\in \N$.
One can see that every power partial isometry $V$ generates a partial-isometric representation of $\N$ such that $V_{n}:=V^{n}$
for every $n\in \N$, and every partial-isometric representation $V$ of $\N$ arises this way, which means that it is actually
generated by the power partial isometry $V_{1}$. Now, it follows by Lemma \ref{bicov-lem} that $U$ is a bicovariant
partial-isometric representation of $\N^{2}$ if and only if there are $*$-commuting power partial isometries $V$ and $W$ such that
$U_{(m,n)}=V^{m}W^{n}$ for every $(m,n)\in \N^{2}$.
\end{remark}

\begin{lemma}
\label{L-Nica-FG}
Consider the left quasi-lattice ordered group $(\F_{n},\F_{n}^{+})$ (see Example \ref{free-G}).
A partial-isometric representation $V$ of $\F_{n}^{+}$ satisfies the left Nica covariance condition (\ref{Nica-2}) if and only if the initial projections $V_{a_{i}}V_{a_{i}}^{*}$ and $V_{a_{j}}V_{a_{j}}^{*}$ have orthogonal ranges, where
$1\leq i,j\leq n$ such that $i\neq j$.
\end{lemma}

\begin{proof}
We skip the proof as it follows by a similar discussion to the proof of Lemma \ref{Nica-piso-FG}.
\end{proof}

\begin{lemma}
\label{bicov-N*}
Consider the abelian lattice-ordered group $(\Q_{+}^{\times},\N^{\times})$ (see Example \ref{lcm2}).
A partial-isometric representation $V$ of $\N^{\times}$ is bicovariant if and only if $V_{m}^{*}V_{n}=V_{n}V_{m}^{*}$
for every relatively prime pair $(m,n)$ of elements in $\N^{\times}$.
\end{lemma}

\begin{proof}
Suppose that $V$ is a bicovariant partial-isometric representation of $\N^{\times}$. So, we have
\begin{align}
\label{eq34}
V_{x}^{*}V_{x}V_{y}^{*}V_{y}=V_{x\vee y}^{*}V_{x\vee y},
\end{align}
and
\begin{align}
\label{eq35}
V_{x}V_{x}^{*}V_{y}V_{y}^{*}=V_{x\vee y}V_{x\vee y}^{*}
\end{align}
for all $x,y\in \N^{\times}$. If $(m,n)$ is a relatively prime pair of elements in $\N^{\times}$, then $n\vee m=nm$,
and hence, we have
\begin{eqnarray*}
\begin{array}{rcl}
V_{n}V_{m}^{*}&=&V_{n}(V_{n}^{*}V_{n}V_{m}^{*}V_{m})V_{m}^{*}\\
&=&V_{n}V_{nm}^{*}V_{nm}V_{m}^{*}\ \ \ \ \ (by\ \ (\ref{eq34}))\\
&=&V_{n}(V_{n} V_{m})^{*} V_{n}V_{m} V_{m}^{*}\\
&=&V_{n} V_{m}^{*} (V_{n}^{*}V_{n} V_{m} V_{m}^{*}),
\end{array}
\end{eqnarray*}
where in the bottom line, by \cite[Lemma 2]{Halmos}, $V_{n}^{*}V_{n}$ commutes with $V_{m} V_{m}^{*}$ as $V_{n}V_{m}=V_{nm}$.
We therefore get
\begin{eqnarray*}
\begin{array}{rcl}
V_{n}V_{m}^{*}&=&V_{n} (V_{m}^{*}V_{m} V_{m}^{*}) V_{n}^{*}V_{n}\\
&=&V_{n} (V_{m}^{*} V_{n}^{*}) V_{n}\\
&=&V_{n} (V_{n} V_{m})^{*} V_{n}\\
&=&V_{n} (V_{nm})^{*} V_{n}\\
&=&V_{n} (V_{mn})^{*} V_{n}\\
&=&V_{n} (V_{m} V_{n})^{*} V_{n}\\
&=&V_{n} (V_{n}^{*} V_{m}^{*}) V_{n}\\
&=&(V_{n} V_{n}^{*} V_{m}^{*}V_{m}) V_{m}^{*} V_{n}\\
&=&(V_{m}^{*}V_{m} V_{n} V_{n}^{*}) V_{m}^{*} V_{n}\\
&=&V_{m}^{*} V_{mn} (V_{m} V_{n})^{*} V_{n}=V_{m}^{*} V_{mn} V_{mn}^{*} V_{n}.
\end{array}
\end{eqnarray*}
Finally, in the bottom line, since $V_{mn} V_{mn}^{*}=V_{m} V_{m}^{*}V_{n} V_{n}^{*}$ by (\ref{eq35}), it follows that
$$V_{n}V_{m}^{*}=V_{m}^{*}V_{m}V_{m}^{*} V_{n}V_{n}^{*}V_{n}=V_{m}^{*} V_{n}.$$

Conversely, assume that $V$ is a partial-isometric representation of $\N^{\times}$ such that $V_{m}^{*}V_{n}=V_{n}V_{m}^{*}$
for every relatively prime pair $(m,n)$ of elements in $\N^{\times}$. We want to prove that $V$ is bicovariant. To do so, we only show that
$V$ satisfies the equation (\ref{eq34}) as the other equation, namely (\ref{eq35}), follows by a similar discussion.
Let $x,y\in \N^{\times}$. If $x$ and $y$ are relatively prime, then one can easily see that $V$ satisfies the equation (\ref{eq34}) as
$V_{x}$ commutes with $V_{y}^{*}$. Now, assume that $x$ and $y$ are not relatively prime, and therefore, we must have $x,y>1$.
By the prime factorization theorem, $x$ and $y$ can be uniquely written as
$$x=(p_{1}^{n_{1}}p_{2}^{n_{2}}\cdot\cdot\cdot p_{k}^{n_{k}}) r\ \ \ \textrm{and}\ \ \
y=(p_{1}^{m_{1}}p_{2}^{m_{2}}\cdot\cdot\cdot p_{k}^{m_{k}}) s,$$
where $p_{1}<p_{2}<...< p_{k}$ are primes, each $n_{i}$ and $m_{i}$ is a positive integer, and $r,s\in \N^{\times}$ such that
all pairs $(r,\prod_{i=1}^{k} p_{i}^{n_{i}})$, $(r,s)$, $(r,\prod_{i=1}^{k} p_{i}^{m_{i}})$,
$(s,\prod_{i=1}^{k} p_{i}^{m_{i}})$, and $(s,\prod_{i=1}^{k} p_{i}^{n_{i}})$ are relatively prime. We let
$a=\prod_{i=1}^{k} p_{i}^{n_{i}}$ and $b=\prod_{i=1}^{k} p_{i}^{m_{i}}$ for convenience. Now, we have
\begin{eqnarray}
\label{eq37}
\begin{array}{rcl}
V_{x}^{*}V_{x}V_{y}^{*}V_{y}&=&V_{ar}^{*}V_{ar}V_{bs}^{*}V_{bs}\\
&=&(V_{a}V_{r})^{*}V_{a}V_{r}(V_{b}V_{s})^{*}V_{b}V_{s}\\
&=&V_{r}^{*}V_{a}^{*} V_{a}(V_{r} V_{s}^{*})V_{b}^{*} V_{b}V_{s}\\
&=&V_{r}^{*}V_{a}^{*} (V_{a} V_{s}^{*}) (V_{r} V_{b}^{*}) V_{b}V_{s}\ \ \ (\textrm{since g.c.d}(r,s)=1)\\
&=&V_{r}^{*} (V_{a}^{*}V_{s}^{*}) V_{a} V_{b}^{*} (V_{r} V_{b}) V_{s}\ \ \ (\textrm{since g.c.d}(a,s)=\textrm{g.c.d}(r,b)=1)\\
&=&V_{r}^{*} (V_{s}V_{a})^{*} V_{a} V_{b}^{*} V_{rb} V_{s}\\
&=&V_{r}^{*} (V_{sa})^{*} V_{a} V_{b}^{*} V_{br} V_{s}\\
&=&V_{r}^{*} (V_{as})^{*} V_{a} V_{b}^{*} V_{b}V_{r} V_{s}\\
&=&V_{r}^{*} (V_{a}V_{s})^{*} V_{a} V_{b}^{*} V_{b} V_{rs}\\
&=&(V_{r}^{*} V_{s}^{*}) (V_{a}^{*} V_{a} V_{b}^{*} V_{b}) V_{rs}\\
&=&(V_{s} V_{r})^{*} (V_{a}^{*} V_{a} V_{b}^{*} V_{b}) V_{rs}\\
&=&V_{sr}^{*} (V_{a}^{*} V_{a} V_{b}^{*} V_{b}) V_{rs}=V_{rs}^{*} (V_{a}^{*} V_{a} V_{b}^{*} V_{b}) V_{rs}.
\end{array}
\end{eqnarray}
Then, in the bottom line, for $V_{a}^{*} V_{a} V_{b}^{*} V_{b}$, since the pairs $(V_{t},V_{u})$ and $(V_{t}^{*},V_{u}^{*})$ obviously commute
for all $t,u\in\N^{\times}$, and g.c.d$(p^{m},q^{n})=1$ for distinct primes $p$ and $q$ and positive integers $m$ and $n$, we have
\begin{eqnarray*}
\begin{array}{l}
V_{a}^{*} V_{a} V_{b}^{*} V_{b}\\
=(V_{p_{k}^{n_{k}}}^{*} \cdot\cdot\cdot V_{p_{2}^{n_{2}}}^{*} V_{p_{1}^{n_{1}}}^{*})
(V_{p_{1}^{n_{1}}} V_{p_{2}^{n_{2}}} \cdot\cdot\cdot V_{p_{k}^{n_{k}}})
(V_{p_{k}^{m_{k}}}^{*} \cdot\cdot\cdot V_{p_{2}^{m_{2}}}^{*} V_{p_{1}^{m_{1}}}^{*})
(V_{p_{1}^{m_{1}}} V_{p_{2}^{m_{2}}} \cdot\cdot\cdot V_{p_{k}^{m_{k}}})\\
=(V_{p_{1}^{n_{1}}}^{*} V_{p_{2}^{n_{2}}}^{*} \cdot\cdot\cdot V_{p_{k}^{n_{k}}}^{*})
(V_{p_{1}^{n_{1}}} V_{p_{2}^{n_{2}}} \cdot\cdot\cdot V_{p_{k}^{n_{k}}})
(V_{p_{1}^{m_{1}}}^{*} V_{p_{2}^{m_{2}}}^{*} \cdot\cdot\cdot V_{p_{k}^{m_{k}}}^{*})
(V_{p_{1}^{m_{1}}} V_{p_{2}^{m_{2}}} \cdot\cdot\cdot V_{p_{k}^{m_{k}}})\\
=(V_{p_{1}^{n_{1}}}^{*}V_{p_{1}^{n_{1}}} V_{p_{1}^{m_{1}}}^{*}V_{p_{1}^{m_{1}}})
(V_{p_{2}^{n_{2}}}^{*} \cdot\cdot\cdot V_{p_{k}^{n_{k}}}^{*})
(V_{p_{2}^{n_{2}}} \cdot\cdot\cdot V_{p_{k}^{n_{k}}})
(V_{p_{2}^{m_{2}}}^{*} \cdot\cdot\cdot V_{p_{k}^{m_{k}}}^{*})
(V_{p_{2}^{m_{2}}} \cdot\cdot\cdot V_{p_{k}^{m_{k}}})\\
\cdot\\
\cdot\\
\cdot\\
=(V_{p_{1}^{n_{1}}}^{*}V_{p_{1}^{n_{1}}} V_{p_{1}^{m_{1}}}^{*}V_{p_{1}^{m_{1}}})
(V_{p_{2}^{n_{2}}}^{*}V_{p_{2}^{n_{2}}} V_{p_{2}^{m_{2}}}^{*}V_{p_{2}^{m_{2}}})
(V_{p_{3}^{n_{3}}}^{*}V_{p_{3}^{n_{3}}} V_{p_{3}^{m_{3}}}^{*}V_{p_{3}^{m_{3}}})
\cdot\cdot\cdot
(V_{p_{k}^{n_{k}}}^{*}V_{p_{k}^{n_{k}}} V_{p_{k}^{m_{k}}}^{*}V_{p_{k}^{m_{k}}}).
\end{array}
\end{eqnarray*}
Let $t_{i}=\max\{n_{i}, m_{i}\}$ for every $1\leq i\leq k$. If $\max\{n_{i}, m_{i}\}=n_{i}$, then
\begin{eqnarray*}
\begin{array}{rcl}
V_{p_{i}^{n_{i}}}^{*}V_{p_{i}^{n_{i}}} V_{p_{i}^{m_{i}}}^{*}V_{p_{i}^{m_{i}}}&=&
V_{p_{i}^{n_{i}}}^{*}(V_{p_{i}^{(n_{i}-m_{i})}p_{i}^{m_{i}}}) V_{p_{i}^{m_{i}}}^{*}V_{p_{i}^{m_{i}}}\\
&=&V_{p_{i}^{n_{i}}}^{*} V_{p_{i}^{(n_{i}-m_{i})}} (V_{p_{i}^{m_{i}}} V_{p_{i}^{m_{i}}}^{*}V_{p_{i}^{m_{i}}})\\
&=&V_{p_{i}^{n_{i}}}^{*} V_{p_{i}^{(n_{i}-m_{i})}} V_{p_{i}^{m_{i}}}\\
&=&V_{p_{i}^{n_{i}}}^{*} V_{p_{i}^{(n_{i}-m_{i})} p_{i}^{m_{i}}}\\
&=&V_{p_{i}^{n_{i}}}^{*} V_{p_{i}^{n_{i}}}=V_{p_{i}^{t_{i}}}^{*} V_{p_{i}^{t_{i}}}.
\end{array}
\end{eqnarray*}
If $\max\{n_{i}, m_{i}\}=m_{i}$, then
\begin{eqnarray*}
\begin{array}{rcl}
V_{p_{i}^{n_{i}}}^{*}V_{p_{i}^{n_{i}}} V_{p_{i}^{m_{i}}}^{*}V_{p_{i}^{m_{i}}}&=&
V_{p_{i}^{n_{i}}}^{*}V_{p_{i}^{n_{i}}} (V_{p_{i}^{(m_{i}-n_{i})}p_{i}^{n_{i}}})^{*}V_{p_{i}^{m_{i}}}\\
&=&V_{p_{i}^{n_{i}}}^{*}V_{p_{i}^{n_{i}}} (V_{p_{i}^{(m_{i}-n_{i})}}V_{p_{i}^{n_{i}}})^{*}V_{p_{i}^{m_{i}}}\\
&=&(V_{p_{i}^{n_{i}}}^{*}V_{p_{i}^{n_{i}}}V_{p_{i}^{n_{i}}}^{*}) V_{p_{i}^{(m_{i}-n_{i})}}^{*}V_{p_{i}^{m_{i}}}\\
&=&V_{p_{i}^{n_{i}}}^{*} V_{p_{i}^{(m_{i}-n_{i})}}^{*}V_{p_{i}^{m_{i}}}\\
&=&(V_{p_{i}^{(m_{i}-n_{i})}} V_{p_{i}^{n_{i}}})^{*} V_{p_{i}^{m_{i}}}\\
&=&(V_{p_{i}^{(m_{i}-n_{i})}p_{i}^{n_{i}}})^{*} V_{p_{i}^{m_{i}}}\\
&=&V_{p_{i}^{m_{i}}}^{*} V_{p_{i}^{m_{i}}}=V_{p_{i}^{t_{i}}}^{*} V_{p_{i}^{t_{i}}}.
\end{array}
\end{eqnarray*}
So, for every $1\leq i\leq k$,
$$V_{p_{i}^{n_{i}}}^{*}V_{p_{i}^{n_{i}}} V_{p_{i}^{m_{i}}}^{*}V_{p_{i}^{m_{i}}}=V_{p_{i}^{t_{i}}}^{*} V_{p_{i}^{t_{i}}},$$
from which, it follows that
\begin{eqnarray}
\label{eq36}
\begin{array}{rcl}
V_{a}^{*} V_{a} V_{b}^{*} V_{b}&=&(V_{p_{1}^{t_{1}}}^{*} V_{p_{1}^{t_{1}}})(V_{p_{2}^{t_{2}}}^{*} V_{p_{2}^{t_{2}}})
(V_{p_{3}^{t_{3}}}^{*} V_{p_{3}^{t_{3}}})\cdot\cdot\cdot(V_{p_{k}^{t_{k}}}^{*} V_{p_{k}^{t_{k}}})\\
&=&(V_{p_{1}^{t_{1}}}^{*} V_{p_{2}^{t_{2}}}^{*}) (V_{p_{1}^{t_{1}}} V_{p_{2}^{t_{2}}})
(V_{p_{3}^{t_{3}}}^{*} V_{p_{3}^{t_{3}}})\cdot\cdot\cdot(V_{p_{k}^{t_{k}}}^{*} V_{p_{k}^{t_{k}}})\\
&=&(V_{p_{1}^{t_{1}}}^{*} V_{p_{2}^{t_{2}}}^{*} V_{p_{3}^{t_{3}}}^{*})
(V_{p_{1}^{t_{1}}} V_{p_{2}^{t_{2}}} V_{p_{3}^{t_{3}}}) (V_{p_{4}^{t_{4}}}^{*} V_{p_{4}^{t_{4}}})
\cdot\cdot\cdot(V_{p_{k}^{t_{k}}}^{*} V_{p_{k}^{t_{k}}})\\
&\cdot&\\
&\cdot&\\
&\cdot&\\
&=&(V_{p_{1}^{t_{1}}}^{*} V_{p_{2}^{t_{2}}}^{*} V_{p_{3}^{t_{3}}}^{*}\cdot\cdot\cdot V_{p_{k}^{t_{k}}}^{*})
(V_{p_{1}^{t_{1}}} V_{p_{2}^{t_{2}}} V_{p_{3}^{t_{3}}}\cdot\cdot\cdot V_{p_{k}^{t_{k}}})\\
&=&(V_{p_{k}^{t_{k}}\cdot\cdot\cdot p_{3}^{t_{3}} p_{2}^{t_{2}} p_{1}^{t_{1}}})^{*}
(V_{p_{1}^{t_{1}} p_{2}^{t_{2}} p_{3}^{t_{3}} \cdot\cdot\cdot p_{k}^{t_{k}}})\\
&=&(V_{p_{1}^{t_{1}} p_{2}^{t_{2}} p_{3}^{t_{3}} \cdot\cdot\cdot p_{k}^{t_{k}}})^{*}
(V_{p_{1}^{t_{1}} p_{2}^{t_{2}} p_{3}^{t_{3}} \cdot\cdot\cdot p_{k}^{t_{k}}})\\
&=&(V_{\prod_{i=1}^{k} p_{i}^{t_{i}}})^{*}
(V_{\prod_{i=1}^{k} p_{i}^{t_{i}}})=V_{a\vee b}^{*} V_{a\vee b}.
\end{array}
\end{eqnarray}
Finally, by applying (\ref{eq36}) to (\ref{eq37}), since $x\vee y=(\prod_{i=1}^{k} p_{i}^{t_{i}})rs=(a\vee b)rs$, we get
\begin{eqnarray*}
\begin{array}{rcl}
V_{x}^{*}V_{x}V_{y}^{*}V_{y}&=&V_{rs}^{*} (V_{a\vee b}^{*} V_{a\vee b}) V_{rs}\\
&=&(V_{a\vee b} V_{rs})^{*} (V_{a\vee b} V_{rs})\\
&=&(V_{(a\vee b)rs})^{*} (V_{(a\vee b)rs})=V_{x\vee y}^{*} V_{x\vee y}.
\end{array}
\end{eqnarray*}
So, we are done.
\end{proof}

\begin{definition}
\label{L-Nica-act}
Let $(A,P,\alpha)$ be a dynamical system, in which $P$ and $P^{\textrm{o}}$ are both left LCM semigroups. The action
$\alpha$ is called \emph{left-Nica covariant} if it satisfies
\begin{align}
\label{LNA}
\overline{\alpha}_{x}(1)\overline{\alpha}_{y}(1)=
   \begin{cases}
      \overline{\alpha}_{z}(1) &\textrm{if}\empty\ \text{$xP\cap yP=zP$,}\\
      0 &\textrm{if}\empty\ \text{$xP\cap yP=\emptyset$.}\\
   \end{cases}
\end{align}
\end{definition}
We should mention that the above definition is well-defined. This is due to the fact that if $tP=xP\cap yP=zP$, then there is an
invertible element $u$ of $P$ such that $t=zu$. Since $u$ is invertible, $\alpha_{u}$ becomes an automorphism of $A$, and hence
$\overline{\alpha}_{u}(1)=1$. So, it follows that
$$\overline{\alpha}_{t}(1)=\overline{\alpha}_{zu}(1)=\overline{\alpha}_{z}(\overline{\alpha}_{u}(1))=
\overline{\alpha}_{z}(1).$$

\begin{remark}
\label{LNA-rmk}
Let $(A,P,\alpha)$ be a dynamical system, in which $P$ and $P^{\textrm{o}}$ are both left LCM semigroups, and the action
$\alpha$ is left Nica-covariant. If $(\pi,V)$ is covariant partial-isometric of the system, then the representation $V$
satisfies the left Nica-covariance condition (\ref{Nica-2}). One can easily see this by applying the equation
$V_{x}V_{x}^{*}=\overline{\pi}(\overline{\alpha}_{x}(1))$ (see Lemma \ref{cov2-lemma}). Thus, the representation
$V$ is actually bicovariant.
\end{remark}

Let us also recall that for $C^*$-algebras $A$ and $B$, there are nondegenerate homomorphisms
$$k_{A}:A\rightarrow \M(A\otimes_{\textrm{max}} B)\ \ \textrm{and}\ \ k_{B}:B\rightarrow \M(A\otimes_{\textrm{max}} B)$$
such that
$$k_{A}(a)k_{B}(b)=k_{B}(b)k_{A}(a)=a\otimes b$$
for all $a\in A$ and $b\in B$ (see \cite[Theorem B. 27]{RW}). Moreover, One can see that
the extensions $\overline{k_{A}}$ and $\overline{k_{B}}$ of the nondegenerate homomorphisms $k_{A}$ and $k_{B}$,
respectively, have also commuting ranges. Therefore, there is a homomorphism
$$\overline{k_{A}}\otimes_{\max}\overline{k_{B}}:\M(A)\otimes_{\max} \M(B)\rightarrow \M(A\otimes_{\textrm{max}} B),$$
which is the identity map on $A\otimes_{\textrm{max}} B$ (see \cite[Remark 2.2]{Larsen}).

\begin{theorem}
\label{P*S-tensor}
Suppose that the unital semigroups $P$, $P^{\textrm{o}}$, $S$, and $S^{\textrm{o}}$ are all left LCM.
Let $(A,P,\alpha)$ and $(B,S,\beta)$ be dynamical systems in which the actions $\alpha$ and $\beta$ are both left Nica-covariant.
Then, we have the following isomorphism:
\begin{align}
\label{P*S-CP}
(A\otimes_{\textrm{max}} B)\times_{\alpha\otimes \beta}^{\piso}(P\times S)
\simeq (A\times_{\alpha}^{\piso} P) \otimes_{\textrm{max}} (B\times_{\beta}^{\piso} S).
\end{align}
\end{theorem}

\begin{proof}
Let the triples $(A\times_{\alpha}^{\piso} P,i_{A},i_{P})$ and $(B\times_{\alpha}^{\piso} S,i_{B},i_{S})$ be
the partial-isometric crossed products of the dynamical systems $(A,P,\alpha)$ and $(B,S,\beta)$, respectively.
Suppose that $(k_{A\times_{\alpha} P},k_{B\times_{\alpha} S})$ is the canonical pair of the algebras $A\times_{\alpha}^{\piso} P$ and
$B\times_{\alpha}^{\piso} S$ into the multiplier algebra
$\M((A\times_{\alpha}^{\piso} P) \otimes_{\textrm{max}} (B\times_{\beta}^{\piso} S))$. Define the map
$$j_{A\otimes_{\textrm{max}} B}:A\otimes_{\textrm{max}} B\rightarrow
(A\times_{\alpha}^{\piso} P) \otimes_{\textrm{max}} (B\times_{\beta}^{\piso} S)$$
by $j_{A\otimes_{\textrm{max}} B}:=i_{A}\otimes_{\textrm{max}} i_{B}$ (see \cite[Lemma B. 31]{RW}), and therefore, we have
$$j_{A\otimes_{\textrm{max}} B}(a\otimes b)=(i_{A}\otimes_{\textrm{max}} i_{B})(a\otimes b)=i_{A}(a)\otimes i_{B}(b)
=k_{A\times_{\alpha} P}(i_{A}(a))k_{B\times_{\alpha} S}(i_{B}(b))$$ for all $a,b\in A$. Also, define a map
$$j_{P\times S}:P\times S\rightarrow \M((A\times_{\alpha}^{\piso} P) \otimes_{\textrm{max}} (B\times_{\beta}^{\piso} S))$$
by
$$j_{P\times S}(x,t)
=\overline{k_{A\times_{\alpha} P}}\otimes_{\max}\overline{k_{B\times_{\alpha} S}}(i_{P}(x)\otimes i_{S}(t))
=\overline{k_{A\times_{\alpha} P}}(i_{P}(x))\overline{k_{B\times_{\alpha} S}}(i_{S}(t))$$
for all $(x,t)\in P\times S$. We claim that the triple
$$\big((A\times_{\alpha}^{\piso} P) \otimes_{\textrm{max}} (B\times_{\beta}^{\piso} S),j_{A\otimes_{\textrm{max}} B},j_{P\times S}\big)$$
is a partial-isometric crossed product of the system $(A\otimes_{\max} B,P\times S,\alpha\otimes \beta)$. To prove our claim,
first note that, since the homomorphisms $i_{A}$ and $i_{B}$ are nondegenerate, so is the homomorphism $j_{A\otimes_{\textrm{max}} B}$.
Next, we show that the map $j_{P\times S}$ is a bicovariant partial-isometric representation of $P\times S$.
To do so, first note that, since the actions $\alpha$ and $\beta$ are left Nica-covariant,
the representations $i_{P}$ and $i_{S}$ are bicovariant (see Remark \ref{LNA-rmk}). It follows that the maps
$$j_{P}:P\rightarrow \M((A\times_{\alpha}^{\piso} P) \otimes_{\textrm{max}} (B\times_{\beta}^{\piso} S))$$
and
$$j_{S}:S\rightarrow \M((A\times_{\alpha}^{\piso} P) \otimes_{\textrm{max}} (B\times_{\beta}^{\piso} S))$$
given by compositions
$$P\stackrel{i_{P}}{\longrightarrow} \M(A\times_{\alpha}^{\piso} P) \stackrel{\overline{k_{A\times_{\alpha} P}}}
{\longrightarrow} \M((A\times_{\alpha}^{\piso} P) \otimes_{\textrm{max}} (B\times_{\beta}^{\piso} S))$$
and
$$S\stackrel{i_{S}}{\longrightarrow} \M(B\times_{\beta}^{\piso} S) \stackrel{\overline{k_{B\times_{\beta} S}}}
{\longrightarrow} \M((A\times_{\alpha}^{\piso} P) \otimes_{\textrm{max}} (B\times_{\beta}^{\piso} S)),$$
respectively, are bicovariant partial-isometric representations of $P$ and $S$ in the multiplier algebra
$\M((A\times_{\alpha}^{\piso} P) \otimes_{\textrm{max}} (B\times_{\beta}^{\piso} S))$.
Moreover, since the homomorphisms $\overline{k_{A\times_{\alpha} P}}$ and $\overline{k_{B\times_{\alpha} S}}$ have
commuting ranges, each $j_{P}(x)$ $*$-commutes with each $j_{S}(t)$. Therefore, as
$$j_{P\times S}(x,t)=\overline{k_{A\times_{\alpha} P}}(i_{P}(x))\overline{k_{B\times_{\alpha} S}}(i_{S}(t))
=j_{P}(x)j_{S}(t),$$
it follows by Lemma (\ref{bicov-lem}) that the map $j_{P\times S}$ must be a bicovariant partial-isometric
representation of $P\times S$. Moreover, by using the covariance equations of the pairs $(i_{A},i_{P})$ and
$(i_{B},i_{S})$, and the commutativity of the ranges of the homomorphisms
$\overline{k_{A\times_{\alpha} P}}$ and $\overline{k_{B\times_{\alpha} S}}$, one can see that the pair
$(j_{A\otimes_{\textrm{max}} B},j_{P\times S})$ satisfies the covariance equations
$$j_{A\otimes_{\textrm{max}} B}((\alpha\otimes \beta)_{(x,t)}(a\otimes b))=
j_{P\times S}(x,t)j_{A\otimes_{\textrm{max}} B}(a\otimes b)j_{P\times S}(x,t)^{*}$$ and
$$j_{P\times S}(x,t)^{*}j_{P\times S}(x,t)j_{A\otimes_{\textrm{max}} B}(a\otimes b)=
j_{A\otimes_{\textrm{max}} B}(a\otimes b)j_{P\times S}(x,t)^{*}j_{P\times S}(x,t).$$

Next, suppose that the pair $(\pi,U)$ is covariant partial-isometric representation of
$(A\otimes_{\max} B,P\times S,\alpha\otimes \beta)$ on a Hilbert space $H$. We want to get a nondegenerate representation
$\pi\times U$ of $(A\times_{\alpha}^{\piso} P) \otimes_{\max} (B\times_{\beta}^{\piso} S)$ such that
$$(\pi\times U)\circ j_{A\otimes_{\textrm{max}} B}=\pi\ \ \ \textrm{and}\ \ \
(\overline{\pi\times U})\circ j_{P\times S}=U.$$
Let $(k_{A},k_{B})$ be the canonical pair of the $C^{*}$-algebras $A$ and $B$ into the multiplier algebra $\M(A \otimes_{\textrm{max}} B)$.
The compositions
$$A\stackrel{k_{A}}{\longrightarrow} \M(A \otimes_{\textrm{max}} B) \stackrel{\overline{\pi}}{\longrightarrow} B(H)$$
and
$$B\stackrel{k_{B}}{\longrightarrow} \M(A \otimes_{\textrm{max}} B) \stackrel{\overline{\pi}}{\longrightarrow} B(H)$$
give us the nondegenerate representations $\pi_{A}$ and $\pi_{B}$ of $A$ and $B$ on $H$ with commuting ranges, respectively.
This is due to the fact that the ranges of $i_{A}$ and $i_{B}$ commute (see also \cite[Corollary B. 22]{RW}).
Also, define the maps
$$V:P\rightarrow B(H)\ \ \textrm{and}\ \ W:S\rightarrow B(H)$$
by
$$V_{x}:=U_{(x,e_{S})}\ \ \textrm{and}\ \ W_t:=U_{(e_{P},t)}$$
for all $x\in P$ and $t\in S$, respectively. Since the representation $U$ already satisfies the right Nica covariance condition
(\ref{Nica-1}), if we show that it satisfies the left Nica covariance condition (\ref{Nica-2}), too, then it is bicovariant.
Therefore, it follows by Lemma \ref{bicov-lem} that the maps $V$ and $W$ are bicovariant partial-isometric representations
such that each $V_{x}$ $*$-commutes with each $W_{t}$ for all $x\in P$ and $t\in S$. By Remark \ref{LNA-rmk}, we only need
to verify that the action $\alpha\otimes \beta$ in the system $(A\otimes_{\max} B,P\times S,\alpha\otimes \beta)$ is
left Nica-covariant. Firstly, we have
\begin{eqnarray}
\label{eq24}
\begin{array}{rcl}
\overline{(\alpha\otimes \beta)}_{(x,t)}\circ(\overline{k_{A}}\otimes_{\max}\overline{k_{B}})
&=&\overline{\alpha_{x}\otimes \beta_{t}}\circ(\overline{k_{A}}\otimes_{\max}\overline{k_{B}})\\
&=&(\overline{k_{A}}\circ\overline{\alpha}_{x})\otimes_{\max}(\overline{k_{B}}\circ\overline{\beta}_{t})
\end{array}
\end{eqnarray}
for all $(x,t)\in P\times S$ (see \cite[Lemma 2.3]{Larsen}). It follows that
\begin{eqnarray*}
\begin{array}{l}
\overline{(\alpha\otimes \beta)}_{(x,r)}(1_{\M(A \otimes_{\textrm{max}} B)})
\overline{(\alpha\otimes \beta)}_{(y,s)}(1_{\M(A \otimes_{\textrm{max}} B)})\\
=\big[\overline{\alpha_{x}\otimes \beta_{r}}\big(\overline{k_{A}}(1_{\M(A)})\overline{k_{B}}(1_{\M(B)})\big)\big]
\big[\overline{\alpha_{y}\otimes \beta_{s}}\big(\overline{k_{A}}(1_{\M(A)})\overline{k_{B}}(1_{\M(B)})\big)\big]\\
=\overline{k_{A}}\big(\overline{\alpha}_{x}(1_{\M(A)})\big) \overline{k_{B}}\big(\overline{\beta}_{r}(1_{\M(B)})\big)
\overline{k_{A}}\big(\overline{\alpha}_{y}(1_{\M(A)})\big) \overline{k_{B}}\big(\overline{\beta}_{s}(1_{\M(B)})\big)\\
=\overline{k_{A}}\big(\overline{\alpha}_{x}(1_{\M(A)})\big) \overline{k_{A}}\big(\overline{\alpha}_{y}(1_{\M(A)})\big)
\overline{k_{B}}\big(\overline{\beta}_{r}(1_{\M(B)})\big) \overline{k_{B}}\big(\overline{\beta}_{s}(1_{\M(B)})\big)\\
=\overline{k_{A}}\big(\overline{\alpha}_{x}(1_{\M(A)})\overline{\alpha}_{y}(1_{\M(A)})\big)
\overline{k_{B}}\big(\overline{\beta}_{r}(1_{\M(B)})\overline{\beta}_{s}(1_{\M(B)})\big),
\end{array}
\end{eqnarray*}
and hence,
\begin{eqnarray}
\label{eq25}
\begin{array}{l}
\overline{(\alpha\otimes \beta)}_{(x,r)}(1_{\M(A \otimes_{\textrm{max}} B)})
\overline{(\alpha\otimes \beta)}_{(y,s)}(1_{\M(A \otimes_{\textrm{max}} B)})\\
=\overline{k_{A}}\big(\overline{\alpha}_{x}(1_{\M(A)})\overline{\alpha}_{y}(1_{\M(A)})\big)
\overline{k_{B}}\big(\overline{\beta}_{r}(1_{\M(B)})\overline{\beta}_{s}(1_{\M(B)})\big)
\end{array}
\end{eqnarray}
for all $(x,r),(y,s)\in P\times S$. Now, if
$$(x,r)(P\times S)\cap (y,s)(P\times S)=(z,t)(P\times S)$$ for some $(z,t)\in P\times S$, then, since
$$xP\cap yP=zP\ \ \ \textrm{and}\ \ \ rS\cap sS=tS,$$ and the actions $\alpha$ and $\beta$ are left Nica-covariant,
for (\ref{eq25}), we get
\begin{eqnarray*}
\begin{array}{l}
\overline{(\alpha\otimes \beta)}_{(x,r)}(1_{\M(A \otimes_{\textrm{max}} B)})
\overline{(\alpha\otimes \beta)}_{(y,s)}(1_{\M(A \otimes_{\textrm{max}} B)})\\
=\overline{k_{A}}\big(\overline{\alpha}_{z}(1_{\M(A)})\big) \overline{k_{B}}\big(\overline{\beta}_{t}(1_{\M(B)})\big)\\
=(\overline{k_{A}}\circ\overline{\alpha}_{z})\otimes_{\max} (\overline{k_{B}}\circ\overline{\beta}_{t})(1_{\M(A)}\otimes 1_{\M(B)})\\
=\overline{\alpha_{z}\otimes \beta_{t}}\big(\overline{k_{A}}\otimes_{\max}\overline{k_{B}}(1_{\M(A)}\otimes 1_{\M(B)})\big)
\ \ \ [\textrm{by (\ref{eq24})}]\\
=\overline{(\alpha\otimes \beta)}_{(z,t)}(1_{\M(A \otimes_{\textrm{max}} B)}).
\end{array}
\end{eqnarray*}
If $(x,r)(P\times S)\cap (y,s)(P\times S)=\emptyset$, then
$$xP\cap yP=\emptyset\ \vee\ rS\cap sS=\emptyset,$$ which implies that
$$\overline{\alpha}_{x}(1_{\M(A)})\overline{\alpha}_{y}(1_{\M(A)})=0\ \vee\
\overline{\beta}_{r}(1_{\M(B)})\overline{\beta}_{s}(1_{\M(B)})=0$$
as $\alpha$ and $\beta$ are left Nica-covariant. Thus, for (\ref{eq25}), we get
$$\overline{(\alpha\otimes \beta)}_{(x,r)}(1_{\M(A \otimes_{\textrm{max}} B)})
\overline{(\alpha\otimes \beta)}_{(y,s)}(1_{\M(A \otimes_{\textrm{max}} B)})=0.$$
So, the action $\alpha\otimes \beta$ is left Nica-covariant.

Now, consider the pairs $(\pi_{A},V)$ and $(\pi_{B},W)$. They are indeed the covariant partial-isometric
representations of the systems $(A,P,\alpha)$ and $(B,S,\beta)$ on $H$, respectively. We only show this for $(\pi_{A},V)$ as the proof
for $(\pi_{B},W)$ follows similarly. We have to show that the pair $(\pi_{A},V)$ satisfies the covariance equations (\ref{cov1}).
We have
\begin{eqnarray*}
\begin{array}{rcl}
V_{x}\pi_{A}(a)V_{x}^{*}&=&U_{(x,e_{S})}\overline{\pi}(k_{A}(a))U_{(x,e_{S})}^{*}\\
&=&U_{(x,e_{S})}\overline{\pi}(k_{A}(a))\overline{\pi}(1_{\M(A \otimes_{\textrm{max}} B)})U_{(x,e_{S})}^{*}\\
&=&U_{(x,e_{S})}\overline{\pi}(\overline{k_{A}}(a))\overline{\pi}(\overline{k_{B}}(1_{\M(B)}))U_{(x,e_{S})}^{*}\\
&=&U_{(x,e_{S})}\overline{\pi}(\overline{k_{A}}(a)\overline{k_{B}}(1_{\M(B)}))U_{(x,e_{S})}^{*}\\
&=&\overline{\pi}\big(\overline{(\alpha\otimes \beta)}_{(x,e_{S})}(\overline{k_{A}}(a)\overline{k_{B}}(1_{\M(B)}))\big)
\ \ [\textrm{by the covariance of}\ (\overline{\pi},U)]\\
&=&\overline{\pi}\big(\overline{\alpha_{x}\otimes \beta_{e_{S}}}(\overline{k_{A}}(a)\overline{k_{B}}(1_{\M(B)}))\big)\\
&=&\overline{\pi}\big(\overline{k_{A}}(\overline{\alpha}_{x}(a))\overline{k_{B}}(\overline{\beta}_{e_{S}}(1_{\M(B)}))\big)
\ \ [\textrm{by (\ref{eq24})}]\\
&=&\overline{\pi}\big(k_{A}(\alpha_{x}(a))\overline{k_{B}}(\id_{\M(B)}(1_{\M(B)}))\big)\\
&=&\overline{\pi}\big(k_{A}(\alpha_{x}(a))\overline{k_{B}}(1_{\M(B)})\big)\\
&=&\overline{\pi}\big(k_{A}(\alpha_{x}(a))1_{\M(A \otimes_{\textrm{max}} B)}\big)\\
&=&\overline{\pi}\big(k_{A}(\alpha_{x}(a))\big)=\pi_{A}(\alpha_{x}(a))\\
\end{array}
\end{eqnarray*}
for all $a\in A$ and $x\in P$. Also, by a similar calculation using the covariance of the pair $(\overline{\pi},U)$, it follows that
$$\pi_{A}(a)V_{x}^{*}V_{x}=V_{x}^{*}V_{x}\pi_{A}(a).$$
Consequently, there are nondegenerate representations $\pi_{A}\times V$ and $\pi_{B}\times W$ of the algebras
$A\times_{\alpha}^{\piso} P$ and $B\times_{\beta}^{\piso} S$ on $H$, respectively, such that
$$(\pi_{A}\times V)\circ i_{A}=\pi_{A},\ \overline{\pi_{A}\times V}\circ i_{P}=V\ \ \textrm{and}
\ \ (\pi_{B}\times W)\circ i_{B}=\pi_{B},\ \overline{\pi_{B}\times W}\circ i_{S}=W.$$

Next, we aim to show that the representations $\pi_{A}\times V$ and $\pi_{B}\times W$ have commuting ranges, from which, it
follows that there is a representation $(\pi_{A}\times V)\otimes_{\max}(\pi_{B}\times W)$ of
$(A\times_{\alpha}^{\piso} P)\otimes_{\max} (B\times_{\beta}^{\piso} S)$, which is the desired (nondegenerate) representation
$\pi\times U$. So, it suffices to see that the pairs $(\pi_{A},\pi_{B})$, $(V,W)$, $(V^{*},W)$, $(\pi_{A},W)$, and $(\pi_{B},V)$ all
have commuting ranges. We already saw that this is indeed true for the first three pairs. So, we compute to show that this is also
true for the pair $(\pi_{A},W)$ and skip the similar computation for the pair $(\pi_{B},V)$. We have
\begin{eqnarray*}
\begin{array}{rcl}
W_{t}\pi_{A}(a)&=&U_{(e_{P},t)}\overline{\pi}(k_{A}(a))\\
&=&U_{(e_{P},t)}\overline{\pi}(k_{A}(a)1_{\M(A \otimes_{\textrm{max}} B)})\\
&=&U_{(e_{P},t)}\overline{\pi}(\overline{k_{A}}(a)\overline{k_{B}}(1_{\M(B)}))\\
&=&\overline{\pi}\big(\overline{(\alpha\otimes \beta)}_{(e_{P},t)}(\overline{k_{A}}(a)\overline{k_{B}}(1_{\M(B)}))\big)U_{(e_{P},t)}
\ \ [\textrm{by the covariance of}\ (\overline{\pi},U)]\\
&=&\overline{\pi}\big(\overline{\alpha_{e_{P}}\otimes \beta_{t}}(\overline{k_{A}}(a)\overline{k_{B}}(1_{\M(B)}))\big)U_{(e_{P},t)}\\
&=&\overline{\pi}\big(\overline{k_{A}}(\overline{\alpha}_{e_{P}}(a))\overline{k_{B}}(\overline{\beta}_{t}(1_{\M(B)}))\big)U_{(e_{P},t)}
\ \ [\textrm{by (\ref{eq24})}]\\
&=&\overline{\pi}\big(k_{A}(\id_{A}(a))\overline{k_{B}}(\overline{\beta}_{t}(1_{\M(B)}))\big)U_{(e_{P},t)}\\
&=&(\overline{\pi}\circ k_{A})(a)(\overline{\pi}\circ\overline{k_{B}})\big(\overline{\beta}_{t}(1_{\M(B)})\big)U_{(e_{P},t)}\\
&=&\pi_{A}(a)\overline{\pi_{B}}\big(\overline{\beta}_{t}(1_{\M(B)})\big)W_{t}\\
&=&\pi_{A}(a)W_{t}W_{t}^{*}W_{t}\ \ [\textrm{by the covariance of}\ (\overline{\pi_{B}},W)]\\
&=&\pi_{A}(a)W_{t}.
\end{array}
\end{eqnarray*}
Thus, there is a representation $(\pi_{A}\times V)\otimes_{\max}(\pi_{B}\times W)$ of
$(A\times_{\alpha}^{\piso} P)\otimes_{\max} (B\times_{\beta}^{\piso} S)$ on $H$ such that
$$(\pi_{A}\times V)\otimes_{\max}(\pi_{B}\times W)(\xi\otimes \eta)=(\pi_{A}\times V)(\xi)(\pi_{B}\times W)(\eta)$$
for all $\xi\in (A\times_{\alpha}^{\piso} P)$ and $\eta\in (B\times_{\beta}^{\piso} S)$. Let
$$\pi\times U=(\pi_{A}\times V)\otimes_{\max}(\pi_{B}\times W),$$ which is nondegenerate as both representations
$\pi_{A}\times V$ and $\pi_{B}\times W$ are. Then, we have
\begin{eqnarray*}
\begin{array}{rcl}
\pi\times U(j_{A\otimes_{\textrm{max}} B}(a\otimes b))&=&\pi\times U(i_{A}(a)\otimes i_{B}(b))\\
&=&(\pi_{A}\times V)(i_{A}(a))(\pi_{B}\times W)(i_{B}(b))\\
&=&\pi_{A}(a)\pi_{B}(b)\\
&=&\overline{\pi}(k_{A}(a))\overline{\pi}(k_{B}(b))\\
&=&\overline{\pi}(k_{A}(a)k_{B}(b))=\pi(a\otimes b).
\end{array}
\end{eqnarray*}
To see $(\overline{\pi\times U})\circ j_{P\times S}=U$, we apply the equation
\begin{eqnarray*}
\begin{array}{rcl}
\overline{\pi\times U}\circ
(\overline{k_{A\times_{\alpha} P}}\otimes_{\max}\overline{k_{B\times_{\alpha} S}})&=&
\overline{(\pi_{A}\times V)\otimes_{\max}(\pi_{B}\times W)}\circ
(\overline{k_{A\times_{\alpha} P}}\otimes_{\max}\overline{k_{B\times_{\alpha} S}})\\
&=&\overline{(\pi_{A}\times V)}\otimes_{\max}\overline{(\pi_{B}\times W)},
\end{array}
\end{eqnarray*}
which is valid by \cite[Lemma 2.4]{Larsen}. Therefore, we have
\begin{eqnarray*}
\begin{array}{rcl}
\overline{\pi\times U}(j_{P\times S}(x,t))&=&
\overline{\pi\times U}\big(\overline{k_{A\times_{\alpha} P}}(i_{P}(x))\overline{k_{B\times_{\alpha} S}}(i_{S}(t))\big)\\
&=&\overline{\pi\times U}
\big(\overline{k_{A\times_{\alpha} P}}\otimes_{\max}\overline{k_{B\times_{\alpha} S}}(i_{P}(x)\otimes i_{S}(t))\big)\\
&=&\overline{\pi\times U}\circ(\overline{k_{A\times_{\alpha} P}}\otimes_{\max}\overline{k_{B\times_{\alpha} S}})
(i_{P}(x)\otimes i_{S}(t))\\
&=&\overline{(\pi_{A}\times V)}\otimes_{\max}\overline{(\pi_{B}\times W)}(i_{P}(x)\otimes i_{S}(t))\\
&=&\overline{(\pi_{A}\times V)}(i_{P}(x))\overline{(\pi_{B}\times W)}(i_{S}(t))\\
&=&V_{x}W_{t}=U_{(x,e_{S})}U_{(e_{P},t)}=U_{(x,t)}.
\end{array}
\end{eqnarray*}
Finally, as the algebras $A\times_{\alpha}^{\piso} P$ and $B\times_{\alpha}^{\piso} S$ are spanned by the elements
$i_{P}(x)^{*}i_{A}(a)i_{P}(y)$ and $i_{S}(r)^{*}i_{B}(b)i_{S}(t)$, respectively, the algebra
$(A\times_{\alpha}^{\piso} P)\otimes_{\max} (B\times_{\alpha}^{\piso} S)$ is spanned by the elements
\begin{eqnarray*}
\begin{array}{l}
[i_{P}(x)^{*}i_{A}(a)i_{P}(y)]\otimes [i_{S}(r)^{*}i_{B}(b)i_{S}(t)]\\
=k_{A\times_{\alpha} P}\big(i_{P}(x)^{*}i_{A}(a)i_{P}(y)\big)k_{B\times_{\alpha} S}\big(i_{S}(r)^{*}i_{B}(b)i_{S}(t)\big)\\
=\overline{k_{A\times_{\alpha} P}}(i_{P}(x)^{*})k_{A\times_{\alpha} P}(i_{A}(a))\overline{k_{A\times_{\alpha} P}}(i_{P}(y))
\overline{k_{B\times_{\alpha} S}}(i_{S}(r)^{*})k_{B\times_{\alpha} S}(i_{B}(b))\overline{k_{B\times_{\alpha} S}}(i_{S}(t))\\
=\overline{k_{A\times_{\alpha} P}}(i_{P}(x)^{*})k_{A\times_{\alpha} P}(i_{A}(a))
\overline{k_{B\times_{\alpha} S}}(i_{S}(r)^{*})\overline{k_{A\times_{\alpha} P}}(i_{P}(y))
k_{B\times_{\alpha} S}(i_{B}(b))\overline{k_{B\times_{\alpha} S}}(i_{S}(t))\\
=\overline{k_{A\times_{\alpha} P}}(i_{P}(x)^{*})\overline{k_{B\times_{\alpha} S}}(i_{S}(r)^{*})
k_{A\times_{\alpha} P}(i_{A}(a))k_{B\times_{\alpha} S}(i_{B}(b))
\overline{k_{A\times_{\alpha} P}}(i_{P}(y))\overline{k_{B\times_{\alpha} S}}(i_{S}(t))\\
=\overline{k_{A\times_{\alpha} P}}(i_{P}(x))^{*}\overline{k_{B\times_{\alpha} S}}(i_{S}(r))^{*}
j_{A\otimes_{\textrm{max}} B}(a\otimes b)j_{P\times S}(y,t)\\
=\big[\overline{k_{B\times_{\alpha} S}}(i_{S}(r))\overline{k_{A\times_{\alpha} P}}(i_{P}(x))\big]^{*}
j_{A\otimes_{\textrm{max}} B}(a\otimes b)j_{P\times S}(y,t)\\
=\big[\overline{k_{A\times_{\alpha} P}}(i_{P}(x))\overline{k_{B\times_{\alpha} S}}(i_{S}(r))\big]^{*}
j_{A\otimes_{\textrm{max}} B}(a\otimes b)j_{P\times S}(y,t)\\
=j_{P\times S}(x,r)^{*}j_{A\otimes_{\textrm{max}} B}(a\otimes b)j_{P\times S}(y,t).
\end{array}
\end{eqnarray*}
So, the triple
$$\big((A\times_{\alpha}^{\piso} P) \otimes_{\textrm{max}} (B\times_{\beta}^{\piso} S),j_{A\otimes_{\textrm{max}} B},j_{P\times S}\big)$$
is a partial-isometric crossed product of the system $(A\otimes_{\max} B,P\times S,\alpha\otimes \beta)$. It thus follows that
there is an isomorphism
$$\Gamma:\big((A\otimes_{\textrm{max}} B)\times_{\alpha\otimes \beta}^{\piso}(P\times S),i_{A\otimes_{\textrm{max}} B},i_{P\times S}\big)
\rightarrow (A\times_{\alpha}^{\piso} P) \otimes_{\textrm{max}} (B\times_{\beta}^{\piso} S)$$
such that
\begin{eqnarray*}
\begin{array}{rcl}
\Gamma\big(i_{P\times S}(x,r)^{*}i_{A\otimes_{\textrm{max}} B}(a\otimes b)i_{P\times S}(y,t)\big)
&=&j_{P\times S}(x,r)^{*}j_{A\otimes_{\textrm{max}} B}(a\otimes b)j_{P\times S}(y,t)\\
&=&[i_{P}(x)^{*}i_{A}(a)i_{P}(y)]\otimes [i_{S}(r)^{*}i_{B}(b)i_{S}(t)].
\end{array}
\end{eqnarray*}
This completes the proof.
\end{proof}

Let $P$ be a unital semigroup such that itself and the opposite semigroup $P^{\textrm{o}}$ are both left LCM. For every
$y\in P$, define a map $1_{y}:P\rightarrow \C$ by
\[
1_{y}(x)=
   \begin{cases}
      1 &\textrm{if}\empty\ \text{$x\in yP$,}\\
      0 &\textrm{otherwise,}\\
   \end{cases}
\]
which is the characteristic function of $yP$. Each $1_{y}$ is obviously a function in $\ell^{\infty}(P)$. Then, since
$P$ is right LCM, one can see that we have
\[
1_{x}1_{y}=
   \begin{cases}
      1_{z} &\textrm{if}\empty\ \text{$xP\cap yP=zP$,}\\
      0 &\textrm{$xP\cap yP=\emptyset$.}\\
   \end{cases}
\]
Note that, if $\tilde{z}P=xP\cap yP=zP$, then there is an invertible element $u$ of $P$ such that $\tilde{z}=zu$. It therefore follows
that $s\in zP$ if and only if $s\in \tilde{z}P$ for all $s\in P$, which implies that we must have $1_{z}=1_{\tilde{z}}$. So, the above
equation is well-defined.
Also, we clearly have $1_{y}^{*}=1_{y}$ for all $y\in P$. Therefore, if $B_{P}$ is the $C^{*}$-subalgebra of $\ell^{\infty}(P)$
generated by the characteristic functions $\{1_{y}:y\in P\}$, then we have
$$B_{P}=\clsp\{1_{y}:y\in P\}.$$
Note that the algebra $B_{P}$ is abelian and unital, whose unit element is $1_{e}$ which is a constant function on $P$ with
the constant value $1$. One can see that, in fact, $1_{u}=1_{e}$ for every $u\in P^{*}$.
In addition, the shift on $\ell^{\infty}(P)$ induces an action on $B_{P}$ by injective
endomorphisms. More precisely, for every $x\in P$, the map $\alpha_{x}:\ell^{\infty}(P)\rightarrow \ell^{\infty}(P)$ defined by
\[
\alpha_{x}(f)(t)=
   \begin{cases}
      f(r) &\textrm{if}\empty\ \text{$t=xr$ for some $r\in P$ ($\equiv t\in xP$),}\\
      0 &\textrm{otherwise,}\\
   \end{cases}
\]
for every $f\in \ell^{\infty}(P)$ is an injective endomorphism of $\ell^{\infty}(P)$. Also, the map
$$\alpha:P\rightarrow \End(\ell^{\infty}(P));\ \ x\mapsto \alpha_{x}$$
is a semigroup homomorphism such that $\alpha_{e}=\id$, which gives us an action of $P$ on $\ell^{\infty}(P)$ by injective endomorphisms.
Since $\alpha_{x}(1_{y})=1_{xy}$ for all $x,y\in P$, $\alpha_{x}(B_{P})\subset B_{P}$, and therefore the restriction of the
action $\alpha$ to $B_{P}$ gives an action
$$\tau:P\rightarrow \End(B_{P})$$
by injective endomorphisms such that $\tau_{x}(1_{y})=1_{xy}$ for all $x,y\in P$. Note that $\tau_{x}(1_{e})=1_{x}\neq 1_{e}$ for
all $x\in P\backslash P^{*}$. Consequently, we obtain a dynamical system $(B_{P},P,\tau)$, for which, we want to describe the
corresponding partial-isometric crossed product $(B_{P}\times_{\tau}^{\piso} P,i_{B_{P}},i_{P})$. More precisely, we want to show that
the algebra $B_{P}\times_{\tau}^{\piso} P$ is universal for bicovariant partial-isometric representations of $P$. Once, we have
done this, it would be proper to denote $B_{P}\times_{\tau}^{\piso} P$ by $C^{*}_{\bicov}(P)$. So, this actually
generalizes \cite[Proposition 9.6]{Fowler} from the positive cones of quasi lattice-ordered groups (in the sense of Nica \cite{Nica})
to LCM semigroups.

To start, for our purpose, we borrow some notations from quasi lattice-ordered groups. For every $x,y\in P$, if
$xP\cap yP=zP$ for some $z\in P$, which means that $z$ is a right least common multiple of $x$ and $y$, then we denote such an
element $z$ by $x\vee_{\lt} y$, which may not be unique. If $xP\cap yP=\emptyset$, then we denote $x\vee_{\lt} y=\infty$.
Note that we are using the notation $\vee_{\lt}$ to indicate that we are treating $P$ as a right LCM semigroup. But if we are
treating $P$ as a left LCM semigroup, then we use the notation $\vee_{\rt}$ to distinguish it from $\vee_{\lt}$. Moreover, if
$F=\{x_{1}, x_{2},..., x_{n}\}$ is any finite subset of $P$, then $\sigma F$ is written
for $x_{1}\vee_{\lt} x_{2}\vee_{\lt}...\vee_{\lt} x_{n}$. Therefore, if
$\bigcap_{x\in F} xP=\bigcap_{i=1}^{n} x_{i}P\neq\emptyset$, $\sigma F$ denotes an
element in $$\{y:\bigcap_{i=1}^{n} x_{i}P=yP\},$$
and if $\bigcap_{i=1}^{n} x_{i}P=\emptyset$, then $\sigma F=\infty$.

\begin{lemma}
\label{rep-Bp}
Let $P$ be a unital semigroup such that itself and the opposite semigroup $P^{\textrm{o}}$ are both left LCM. Let
$V$ be any bicovariant partial-isometric representation of $P$ on a Hilbert space $H$. Then:

\begin{itemize}
\item[(i)] there is a (unital) representation $\pi_{V}$ of $B_{P}$ on $H$ such that $\pi_{V}(1_{x})=V_{x}V_{x}^{*}$ for all $x\in P$;
\item[(ii)] the pair $(\pi_{V},V)$ is a covariant partial-isometric representation of $(B_{P},P,\tau)$ on $H$.
\end{itemize}

\end{lemma}

\begin{proof}
We prove (i) by extending \cite[Proposition 1.3 (2)]{LacaR} to LCM semigroups for the particular family
$$\{L_{x}:=V_{x}V_{x}^{*}: x\in P\}$$
of projections, which satifies
$$L_{e}=1\ \ \ \textrm{and}\ \ \ L_{x}L_{y}=L_{x\vee_{\lt} y},$$
where $L_{\infty}=0$. To do so, we make some adjustment to the proof of \cite[Proposition 1.3 (2)]{LacaR}. Define a map
$$\pi:\lsp\{1_{x}:x\in P\}\rightarrow B(H)$$
by
$$\pi\bigg(\sum_{i=1}^{n}\lambda_{x_{i}}1_{x_{i}}\bigg)=\sum_{i=1}^{n}\lambda_{x_{i}}L_{x_{i}}
=\sum_{i=1}^{n}\lambda_{x_{i}}V_{x_{i}}V_{x_{i}}^{*},$$
where $\lambda_{x_{i}}\in \C$ for each $i$. It is obvious that $\pi$ is linear. Next, we show that
\begin{align}
\label{eq27}
\bigg\| \sum_{x\in F}\lambda_{x}L_{x} \bigg\|\leq \bigg\| \sum_{x\in F}\lambda_{x}1_{x} \bigg\|
\end{align}
for any finite subset $F$ of $P$. So, it follows that the map $\pi$ is a well-defined bounded linear map, and therefore, it
extends to a bounded linear map of $B_{p}$ in $B(H)$ such that $1_{x}\mapsto V_{x}V_{x}^{*}$ for all $x\in P$. To see
(\ref{eq27}), we exactly follow \cite[Lemma 1.4]{LacaR} to obtain an expression for the norm of the forms $\sum_{x\in F}\lambda_{x}L_{x}$
by using an appropriate set of mutually orthogonal projections. So, if $F$ is any finite subset of $P$, then for every
nonempty proper subset $A$ of $F$, take $Q_{A}^{L}=\Pi_{x\in F\backslash A}(L_{\sigma A}-L_{\sigma A\vee_{\lt} x})$. Moreover, let
$Q_{\emptyset}^{L}=\Pi_{x\in F}(1-L_{x})$ and $Q_{F}^{L}=\Pi_{x\in F} L_{x}=L_{\sigma F}$. Then, exactly by
following the proof of \cite[Lemma 1.4]{LacaR}, we can show that $\{Q_{A}^{L}: A\subset F\}$ is a decomposition of the identity into
mutually orthogonal projections, such that
\begin{align}
\label{eq28}
\sum_{x\in F}\lambda_{x}L_{x}=\sum_{A\subset F} \bigg(\sum_{x\in A}\lambda_{x}\bigg)Q_{A}^{L}
\end{align}
and
\begin{align}
\label{eq29}
\bigg\| \sum_{x\in F}\lambda_{x}L_{x} \bigg\|=\max\bigg\{\bigg| \sum_{x\in A}\lambda_{x} \bigg|: A\subset F\ \textrm{and}\ Q_{A}^{L}\neq 0\bigg\}.
\end{align}
Also, we have a fact similar to \cite[Remark 1.5]{LacaR}. Suppose that, similarly, $\{Q_{A}: A\subset F\}$ is the decomposition of the
identity corresponding to the family of projections $\{1_{x}: x\in F\}$. Consider
$$Q_{A}=\Pi_{x\in F\backslash A}(1_{\sigma A}-1_{\sigma A\vee_{\lt} x})$$ for any nonempty proper subset $A\subset F$.
If $\sigma A\in x_{0}P$ for some $x_{0}\in F\backslash A$, then
$(\sigma A) P=\bigcap_{y\in A} yP\subset x_{0}P$ which implies that $(\sigma A) P\cap x_{0}P=(\sigma A) P$. So, we have
$\sigma A\vee_{\lt} x_{0}=\sigma A$, and therefore,
$$1_{\sigma A}-1_{\sigma A\vee_{\lt} x_{0}}=1_{\sigma A}-1_{\sigma A}=0.$$
Thus, we get $Q_{A}=0$. Note that when we say $\sigma A\in x_{0}P$ (for some $x_{0}\in F\backslash A$), it means that at least one
element in
\begin{align}
\label{eq30}
\{z: \bigcap_{y\in A} yP=zP\}
\end{align}
belongs to $x_{0}P$, from which, it follows that all elements in (\ref{eq30}) must belong to $x_{0}P$. This is due to the fact that
if $z,\tilde{z}$ are in (\ref{eq30}), then $\tilde{z}=zu$ for some invertible element $u$ of $P$.
Now, conversely, suppose that $$0=Q_{A}=\Pi_{x\in F\backslash A}(1_{\sigma A}-1_{\sigma A\vee_{\lt} x}).$$
This implies that we must have $Q_{A}(r)=0$ for all $r\in P$, in particular, when $r=\sigma A$, and hence
$$0=Q_{A}(r)=\Pi_{x\in F\backslash A}\big(1_{r}(r)-1_{r\vee_{\lt} x}(r)\big)
=\Pi_{x\in F\backslash A}\big(1-1_{r\vee_{\lt} x}(r)\big).$$
Therefore, there is at least one element $x_{0}\in F\backslash A$ such that $1_{r\vee_{\lt} x_{0}}(r)=1$, which implies that
we must have $r\in (r\vee_{\lt} x_{0})P=rP \cap x_{0}P$. It follows that
$\sigma A=r\in x_{0}P$ and therefore, $r\vee_{\lt} x_{0}=\sigma A\vee_{\lt} x_{0}=\sigma A$.
Consequently, we have $Q_{A}\neq 0$ if and only if
$$A=\{x\in F: \sigma A\in xP\}.$$
Eventually, we conclude that if $Q_{A}^{L}\neq 0$, then $Q_{A}\neq 0$. This is due to the fact that, if $Q_{A}=0$, then there is
$x_{0}\in F\backslash A$ such that $\sigma A\in x_{0}P$. Therefore, we get $Q_{A}^{L}=0$ as the factor
$L_{\sigma A}-L_{\sigma A\vee_{\lt} x_{0}}$ in $Q_{A}^{L}$ becomes zero. Thus, it follows that
$$\bigg\{\bigg| \sum_{x\in A}\lambda_{x} \bigg|: A\subset F\ \textrm{and}\ Q_{A}^{L}\neq 0\bigg\}
\subset \bigg\{\bigg| \sum_{x\in A}\lambda_{x} \bigg|: A\subset F\ \textrm{and}\ Q_{A}\neq 0\bigg\},$$
which implies that the inequality (\ref{eq27}) is valid for any finite subset $F$ of $P$. So, we have a bounded linear map
$\pi_{V}:B_{p}\rightarrow B(H)$ (the extension of $\pi$) such that $\pi_{V}(1_{x})=V_{x}V_{x}^{*}$ for all $x\in P$.
Furthermore, since
$$\pi_{V}(1_{x})\pi_{V}(1_{y})=V_{x}V_{x}^{*}V_{y}V_{y}^{*}=V_{x\vee_{\lt} y}V_{x\vee_{\lt} y}^{*}=\pi_{V}(1_{x\vee_{\lt} y})
=\pi_{V}(1_{x}1_{y}),$$
and obviously, $\pi_{V}(1_{x})^{*}=\pi_{V}(1_{x})=\pi_{V}(1_{x}^{*})$, it follows that the map $\pi_{V}$ is actually a $*$-homomorphism,
which is clearly unital. This completes the proof of (i).

To see (ii), it is enough to show that the pair $(\pi_{V},V)$ satisfies the covariance equations (\ref{cov1}) on the spanning
elements of $B_{P}$. For all $x,y\in P$, we have
$$\pi_{V}(\tau_{x}(1_{y}))=1_{xy}=V_{xy}V_{xy}^{*}=V_{x}V_{y}[V_{x}V_{y}]^{*}=V_{x} V_{y}V_{y}^{*} V_{x}^{*}
=V_{x} \pi_{V}(1_{y}) V_{x}^{*}.$$

Also, since the product of partial isometries $V_{x}$ and $V_{y}$ is a partial isometry, namely, $V_{x}V_{y}=V_{xy}$, by
\cite[Lemma 2]{Halmos}, each $V_{x}^{*}V_{x}$ commutes with each $V_{y}V_{y}^{*}$. Hence, we have
$$V_{x}^{*}V_{x}\pi_{V}(1_{y})=V_{x}^{*}V_{x}V_{y}V_{y}^{*}=V_{y}V_{y}^{*}V_{x}^{*}V_{x}=\pi_{V}(1_{y})V_{x}^{*}V_{x}.$$
So, we are done with (ii), too.

\end{proof}

\begin{prop}
\label{C-st-bicov}
Suppose that $P$ is a unital semigroup such that itself and the opposite semigroup $P^{\textrm{o}}$ are both left LCM.
Then, the map
$$i_{P}: P\rightarrow B_{P}\times_{\tau}^{\piso} P$$
is a bicovariant partial-isometric representation of $P$ whose range generates the $C^{*}$-algebra $B_{P}\times_{\tau}^{\piso} P$.
Moreover, for every bicovariant partial-isometric representation $V$ of $P$, there is a (unital) representation $V_{*}$
of $B_{P}\times_{\tau}^{\piso} P$ such that $V_{*}\circ i_{P}=V$.
\end{prop}

\begin{proof}
To see that $i_{P}$ is a bicovariant partial-isometric representation of $P$, we only need to show that it satisfies the left Nica
covariance condition (\ref{Nica-2}). Since
$$i_{B_{P}}(1_{y})=i_{B_{P}}(\tau_{y}(1_{e}))=i_{P}(y)i_{B_{P}}(1_{e})i_{P}(y)^{*}=i_{P}(y)i_{P}(y)^{*}$$
for all $y\in P$, it follows that $i_{P}$ indeed satisfies (\ref{Nica-2}). Then, as the elements $\{1_{y}:y\in P\}$ generate
the algebra $B_{P}$,  the $C^{*}$-algebra
$B_{P}\times_{\tau}^{\piso} P$ is generated by the elements
$$i_{B_{P}}(1_{y}) i_{P}(x)=i_{P}(y)i_{P}(y)^{*}i_{P}(x),$$
which implies that $i_{P}(P)$ generates $B_{P}\times_{\tau}^{\piso} P$.

Suppose that now $V$ is a bicovariant partial-isometric representation of $P$ on a Hilbert space $H$. Then, by Lemma
\ref{rep-Bp}, there is a covariant partial-isometric representation $(\pi_{V},V)$ of $(B_{P},P,\tau)$ on $H$, such that
$\pi_{V}(1_{x})=V_{x}V_{x}^{*}$ for all $x\in P$. The corresponding (unital) representation
$\pi_{V}\times V$ of $(B_{P}\times_{\tau}^{\piso} P,i_{B_{P}},i_{P})$ on $H$ is the
desired representation $V_{*}$ which satisfies $V_{*}\circ i_{P}=V$.

\end{proof}
So, as we mentioned earlier, we denote the algebra $B_{P}\times_{\tau}^{\piso} P$ by $C^{*}_{\bicov}(P)$, which is universal for
bicovariant partial-isometric representations of $P$.

The following remark contains some point which will be applied in the next corollary and also in the next section.
\begin{remark}
\label{rmk-2}
Suppose that $P$ is a left LCM semigroup. Let $(A,P,\alpha)$ and $(B,P,\beta)$ be dynamical systems, and
$\psi:A\rightarrow B$ a nondegenerate homomorphism such that
$\psi\circ\alpha_{x}=\beta_{x}\circ\psi$ for all $x\in P$. Suppose that $(A\times_{\alpha}^{\piso} P,i)$ and
$(B\times_{\beta}^{\piso} P,j)$ are the partial-isometric crossed products of the systems $(A,P,\alpha)$ and $(B,P,\beta)$,
respectively. Now, one can see that the pair $(j_{B}\circ\psi,j_{P})$ is covariant partial-isometric representation of
$(A,P,\alpha)$ in the algebra $B\times_{\beta}^{\piso} P$. Hence, there is a nondegenerate homomorphism
$$\psi\times P:=[(j_{B}\circ\psi)\times j_{P}]:A\times_{\alpha}^{\piso} P\rightarrow B\times_{\beta}^{\piso} P$$
such that
$$(\psi\times P)\circ i_{A}=j_{B}\circ\psi\ \ \ \textrm{and}\ \ \ \overline{\psi\times P}\circ i_{P}=j_{P}.$$

One can see that if $\psi$ is an isomorphism, so is $\psi\times P$.
\end{remark}

\begin{cor}
\label{bicov-P*S}
Suppose that the unital semigroups $P$, $P^{\textrm{o}}$, $S$, and $S^{\textrm{o}}$ are all left LCM. Then,
\begin{align}
\label{bicov-isom}
C^{*}_{\bicov}(P\times S)\simeq C^{*}_{\bicov}(P)\otimes_{\max} C^{*}_{\bicov}(S).
\end{align}

\end{cor}

\begin{proof}
Corresponding to the pairs $(P,P^{\textrm{o}})$ and $(S,S^{\textrm{o}})$ we have the dynamical systems $(B_{P},P,\tau)$ and
$(B_{S},S,\beta)$ along with their associated $C^{*}$-algebras
$$\big(C^{*}_{\bicov}(P)=B_{P}\times_{\tau}^{\piso} P,i_{B_{P}},V\big)$$
and
$$\big(C^{*}_{\bicov}(S)=B_{S}\times_{\beta}^{\piso} S,i_{B_{S}},W\big),$$
respectively. By Theorem \ref{P*S-tensor}, there is an isomorphism
$$\Gamma:\big((B_{P}\otimes B_{S})\times_{\tau\otimes \beta}^{\piso}(P\times S),
i_{(B_{P}\otimes B_{S})},T\big)
\rightarrow (B_{P}\times_{\tau}^{\piso} P) \otimes_{\textrm{max}} (B_{S}\times_{\beta}^{\piso} S)$$
such that
$$\Gamma\big(i_{(B_{P}\otimes B_{S})}(1_{x}\otimes 1_{t})T_{(p,s)}\big)
=[i_{B_{P}}(1_{x})V_{p}]\otimes [i_{B_{S}}(1_{t})W_{s}]$$
for all $x,p\in P$ and $t,s\in S$. Note that
$B_{P}\otimes B_{S}=B_{P}\otimes_{\textrm{max}} B_{S}=B_{P}\otimes_{\textrm{min}} B_{S}$ as the algebras $B_{P}$ and $B_{S}$ are
abelian. Now, since the unital semigroups $P\times S$ and $(P\times S)^{\textrm{o}}$
are both left LCM, we have a dynamical system $(B_{(P\times S)},P\times S,\alpha)$ along with its associated $C^{*}$-algebra
$$\big(C^{*}_{\bicov}(P\times S)=B_{(P\times S)}\times_{\alpha}^{\piso} (P\times S),i_{B_{(P\times S)}},U\big),$$
where the action $\alpha:P\times S\rightarrow \End (B_{(P\times S)})$ is induced by the shift on $\ell^{\infty}(P\times S)$
such that $\alpha_{(p,s)}(1_{(x,t)})=1_{(p,s)(x,t)}=1_{(px,st)}$. Moreover, since there is an isomorphism
$$\psi:(1_{x}\otimes 1_{t})\in (B_{P}\otimes B_{S})\mapsto 1_{(x,t)}\in B_{(P\times S)}$$
which satisfies $\psi\circ (\tau\otimes \beta)_{(p,s)}=\alpha_{(p,s)}\circ \psi$ for all $(p,s)\in P\times S$, we have an isomorphism
$$\Lambda:(B_{P}\otimes B_{S})\times_{\tau\otimes \beta}^{\piso}(P\times S)
\rightarrow B_{(P\times S)}\times_{\alpha}^{\piso} (P\times S)$$
such that
$$\Lambda\circ i_{(B_{P}\otimes B_{S})}=i_{B_{(P\times S)}}\circ\psi\ \ \ \textrm{and}
\ \ \ \Lambda\circ T=U\ \ (\textrm{see Remark \ref{rmk-2}}).$$
Eventually, the composition
$$C^{*}_{\bicov}(P\times S)\stackrel{\Lambda^{-1}}
{\longrightarrow} (B_{P}\otimes B_{S})\times_{\tau\otimes \beta}^{\piso}(P\times S) \stackrel{\Gamma}
{\longrightarrow} C^{*}_{\bicov}(P)\otimes_{\max} C^{*}_{\bicov}(S)$$ of isomorphisms gives the desired isomorphism (\ref{bicov-isom}),
such that
$$U_{(p,s)}\mapsto V_{p}\otimes W_{s}$$
for all $(p,s)\in P\times S$.

\end{proof}

\section{Ideals in tensor products}
\label{sec:ideals-tensor}
Suppose that $P$ is a left LCM semigroup. Let $\alpha$ and $\beta$ be the actions of $P$ on $C^{*}$-algebras $A$ and $B$ by extendible
endomorphisms, respectively. Then, there is an action
$$\alpha\otimes \beta:P\rightarrow \End(A\otimes_{\max} B)$$
of $P$ on the maximal tensor product $A\otimes_{\max} B$ by extendible endomorphisms such that
$$(\alpha\otimes \beta)_{x}=\alpha_{x}\otimes\beta_{x}\ \ \textrm{for all}\ \ x\in P.$$
Note that the extendibility of $\alpha\otimes \beta$ follows by the extendibility of the actions $\alpha$ and $\beta$ (see \cite[Lemma 2.3]{Larsen}).
Therefore, we obtain a dynamical system $(A\otimes_{\max} B,P,\alpha\otimes \beta)$. Let
$(A\otimes_{\max} B)\times_{\alpha\otimes \beta}^{\piso} P$
be the partial-isometric crossed product of $(A\otimes_{\max} B,P,\alpha\otimes \beta)$. Our main goal in this section is to obtain a composition
series
$$0\leq \I_{1}\leq \I_{2} \leq (A\otimes_{\max} B)\times_{\alpha\otimes \beta}^{\piso} P$$
of ideals, and then identify the subquotients
$$\I_{1},\ \ \ \I_{2}/\I_{1},\ \ \textrm{and}\ \ ((A\otimes_{\max} B)\times_{\alpha\otimes \beta}^{\piso} P)/\I_{2}$$
with familiar terms. To do so, we first need to recall the following lemma from \cite{Larsen}:

\begin{lemma}\cite[Lemma 3.2]{Larsen}
\label{Lar-lem}
Suppose that $\alpha$ and $\beta$ are extendible endomorphisms of $C^{*}$-algebras $A$ and $B$, respectively. If $I$ is an extendible $\alpha$-invariant ideal of $A$ and $J$ is an extendible $\beta$-invariant ideal of $B$, then the ideal $I\otimes_{\max} J$
of $A\otimes_{\max} B$ is extendible $\alpha\otimes\beta$-invariant.
\end{lemma}

\begin{remark}
\label{rmk-3}
It follows by Lemma \ref{Lar-lem} that if $(A,P,\alpha)$ and $(B,P,\beta)$ are dynamical systems, and $I$ is an extendible $\alpha_{x}$-invariant
ideal of $A$ and $J$ is an extendible $\beta_{x}$-invariant ideal of $B$ for every $x\in P$, then $I\otimes_{\max} J$ is an extendible
$(\alpha\otimes\beta)_{x}$-invariant ideal of $A\otimes_{\max} B$ for all $x\in P$. Therefore, by Theorem \ref{piso-ext-seq},
the crossed product $(I\otimes_{\max} J)\times_{\alpha\otimes \beta}^{\piso} P$ sits in the algebra
$(A\otimes_{\max} B)\times_{\alpha\otimes \beta}^{\piso} P$ as an ideal (this will be the ideal $\I_{1}$ shortly later). As an
application of this fact, we observe that, by \cite[Proposition B. 30]{RW}, the short exact sequence
$$0 \longrightarrow J \stackrel{}{\longrightarrow} B \stackrel{q^{J}}{\longrightarrow} B/J \longrightarrow 0$$
gives rise to the short exact sequence
\begin{align}
\label{ext-tensor}
0 \longrightarrow A\otimes_{\max} J \stackrel{}{\longrightarrow} A\otimes_{\max} B
\stackrel{\id_{A}\otimes_{\max} q^{J}}{\longrightarrow} A\otimes_{\max} B/J \longrightarrow 0,
\end{align}
where $A\otimes_{\max} J$ is an extendible $(\alpha\otimes\beta)_{x}$-invariant ideal of $A\otimes_{\max} B$ for all $x\in P$.
Thus, (\ref{ext-tensor}) itself by Theorem \ref{piso-ext-seq} gives rise to the following short exact sequence
$$0 \longrightarrow (A\otimes_{\max} J)\times_{\alpha\otimes \beta}^{\piso} P
\stackrel{\mu}{\longrightarrow} \big((A\otimes_{\max} B)\times_{\alpha\otimes \beta}^{\piso} P, i\big)
\stackrel{\phi}{\longrightarrow} \big((A\otimes_{\max} B/J)\times_{\alpha\otimes \tilde{\beta}}^{\piso} P, j\big),$$
where $\tilde{\beta}:P\rightarrow \End (B/J)$ is the (extendible) action induced by $\beta$, and
the surjective homomorphism $\phi$ is indeed the homomorphism $(\id_{A}\otimes_{\max} q^{J})\times P$ (see Remark \ref{rmk-2})
such that
$$[(\id_{A}\otimes_{\max} q^{J})\times P]\circ i_{(A\otimes_{\max} B)}=j_{(A\otimes_{\max} B/J)}\circ(\id_{A}\otimes_{\max} q^{J})\ \ \
\textrm{and}\ \ \ \overline{[(\id_{A}\otimes_{\max} q^{J})\times P]}\circ i_{P}=j_{P}.$$
\end{remark}

In the following proposition and theorem, for the maximal tensor product between the $C^{*}$-algebras involved, we simply write $\otimes$
for convenience.

\begin{prop}
\label{large-diag.}
Let $(A,P,\alpha)$ and $(B,P,\beta)$ be dynamical systems, and $I$ an extendible $\alpha_{x}$-invariant ideal
of $A$ and $J$ an extendible $\beta_{x}$-invariant ideal of $B$ for every $x\in P$. Assume that
$\tilde{\alpha}:P\rightarrow \End (A/I)$ and $\tilde{\beta}:P\rightarrow \End (B/J)$ are the actions induced by $\alpha$ and
$\beta$, respectively. Then, the following diagram

\begin{center}
\begin{equation}\label{L-comm-diag}
\begin{tikzpicture}
  \matrix(m)[matrix of math nodes,row sep=2em,column sep=2em,minimum width=2em]
  {\empty & 0 & 0 & 0 & \empty\\
  0 & (I\otimes J)\times_{\alpha\otimes \beta}^{\piso} P & (I\otimes B)\times_{\alpha\otimes \beta}^{\piso} P
  & (I\otimes B/J)\times_{\alpha\otimes \tilde{\beta}}^{\piso} P & 0\\
  0 & (A\otimes J)\times_{\alpha\otimes \beta}^{\piso} P & (A\otimes B)\times_{\alpha\otimes \beta}^{\piso} P
  & (A\otimes B/J)\times_{\alpha\otimes \tilde{\beta}}^{\piso} P & 0\\
     0 & (A/I\otimes J)\times_{\tilde{\alpha}\otimes \beta}^{\piso} P & (A/I\otimes B)\times_{\tilde{\alpha}\otimes \beta}^{\piso} P
  & (A/I\otimes B/J)\times_{\tilde{\alpha}\otimes \tilde{\beta}}^{\piso} P & 0\\
   \empty & 0 & 0 & 0 & \empty\\};
  \path[-stealth]
    (m-1-2) edge node [above] {} (m-2-2)
    (m-1-3) edge node [above] {} (m-2-3)
    (m-1-4) edge node [above] {} (m-2-4)
    (m-2-1) edge node [above] {} (m-2-2)
    (m-2-2) edge node [above] {} (m-2-3) edge node [left] {} (m-3-2)
    (m-2-3) edge node [above] {$\phi_{1}$} (m-2-4) edge node [left] {} (m-3-3)
    (m-2-4) edge node [] {} (m-2-5) edge node [left] {} (m-3-4)
    (m-3-1) edge node [] {} (m-3-2)
    (m-3-2) edge node [above] {} (m-3-3)  edge node [right] {$\varphi_{1}$} (m-4-2)
    (m-3-3) edge node [above] {$q\times P$} (m-4-4) edge node [above] {$\phi_{2}$} (m-3-4) edge node [right] {$\varphi_{2}$} (m-4-3)
    (m-3-4) edge node [] {} (m-3-5) edge node [right] {$\varphi_{3}$} (m-4-4)
    (m-4-1) edge node [] {} (m-4-2)
    (m-4-2) edge node [above] {} (m-4-3) edge node [left] {} (m-5-2)
    (m-4-3) edge node [above] {$\phi_{3}$} (m-4-4) edge node [left] {} (m-5-3)
    (m-4-4) edge node [] {} (m-4-5) edge node [left] {} (m-5-4);
\end{tikzpicture}
\end{equation}
\end{center}
commutes, where
$$\phi_{1}:=(\id_{I}\otimes q^{J})\times P,\ \ \phi_{2}:=(\id_{A}\otimes q^{J})\times P,\ \ \phi_{3}:=(\id_{A/I}\otimes q^{J})\times P,$$

$$\varphi_{1}:=(q^{I}\otimes \id_{J})\times P,\ \ \varphi_{2}:=(q^{I}\otimes \id_{B})\times P,\ \textrm{and}\ \
\varphi_{3}:=(q^{I}\otimes \id_{B/J})\times P.$$
Also, there is a surjective homomorphism $q:A\otimes B\rightarrow (A/I)\otimes (B/J)$ which intertwines
the actions $\alpha\otimes \beta$ and $\tilde{\alpha}\otimes \tilde{\beta}$, and therefore, we have a homomorphism $q\times P$
of $(A\otimes B)\times_{\alpha\otimes \beta}^{\piso} P$ onto $(A/I\otimes B/J)\times_{\tilde{\alpha}\otimes \tilde{\beta}}^{\piso} P$
induced by $q$. Moreover, we have
\begin{align}
\label{ker-q.P}
\ker (q\times P)=(A\otimes J)\times_{\alpha\otimes \beta}^{\piso} P+(I\otimes B)\times_{\alpha\otimes \beta}^{\piso} P.
\end{align}

\end{prop}

\begin{proof}
First of all, in the diagram, each row as well as each column is obtained by a similar discussion to Remark \ref{rmk-3},
and hence, it is exact.

Next, for the quotient maps $q^{I}:A\rightarrow A/I$ and $q^{J}:B\rightarrow B/J$, by \cite[Lemma B. 31]{RW}, there is a
homomorphism $q^{I}\otimes q^{J}:A\otimes B\rightarrow (A/I)\otimes (B/J)$, which we denote it by $q$, such that
$$q(a\otimes b)=(q^{I}\otimes q^{J})(a\otimes b)=q^{I}(a)\otimes q^{J}(b)=(a+I)\otimes (b+J)$$
for all $a\in A$ and $b\in B$. It is obviously surjective. Moreover,
\begin{eqnarray}
\label{eq26}
\begin{array}{rcl}
q((\alpha\otimes \beta)_{x}(a\otimes b))&=&q((\alpha_{x}\otimes \beta_{x})(a\otimes b))\\
&=&q(\alpha_{x}(a)\otimes \beta_{x}(b))\\
&=&(\alpha_{x}(a)+I)\otimes (\beta_{x}(b)+J)\\
&=&\tilde{\alpha}_{x}(a+I)\otimes \tilde{\beta}_{x}(b+J)\\
&=&(\tilde{\alpha}_{x}\otimes \tilde{\beta}_{x})((a+I)\otimes (b+J))\\
&=&(\tilde{\alpha}\otimes \tilde{\beta})_{x}(q(a\otimes b))
\end{array}
\end{eqnarray}
for all $x\in P$. Therefore, by Remark \ref{rmk-2}, there is a (nondegenerate) homomorphism
$$q\times P:\big((A\otimes B)\times_{\alpha\otimes \beta}^{\piso} P, i\big)\rightarrow
\big((A/I\otimes B/J)\times_{\tilde{\alpha}\otimes \tilde{\beta}}^{\piso} P, k\big)$$
such that
$$(q\times P)\circ i_{(A\otimes B)}=k_{(A/I\otimes B/J)}\circ q\ \ \
\textrm{and}\ \ \ \overline{q\times P}\circ i_{P}=k_{P}.$$
One can easily see that as $q$ is surjective, so is $q\times P$.

Now, an inspection on spanning elements shows that the diagram commutes.

Finally, to see (\ref{ker-q.P}), we only show that
$$\ker (q\times P)\subset(A\otimes J)\times_{\alpha\otimes \beta}^{\piso} P+(I\otimes B)\times_{\alpha\otimes \beta}^{\piso} P$$
as the other inclusion can be verified easily. To do so, take a nondegenerate representation
$$\pi:(A\otimes B)\times_{\alpha\otimes \beta}^{\piso} P\rightarrow B(H)$$
with
$$\ker \pi=(A\otimes J)\times_{\alpha\otimes \beta}^{\piso} P+(I\otimes B)\times_{\alpha\otimes \beta}^{\piso} P.$$
Then, define a map $\rho:(A/I\otimes B/J)\rightarrow B(H)$ by
$$\rho(q(\xi))=\pi(i_{(A\otimes B)}(\xi))$$ for all $\xi\in (A\otimes B)$. Since
$$(A\otimes J)+(I\otimes B)=\ker q\subset \ker (\pi\circ i_{(A\otimes B)}),$$
it follows that the map $\rho$ is well-defined, which is actually a nondegenerate representation.
Also, the composition
$$P \stackrel{i_{P}}{\longrightarrow} \M\big((A\otimes B)\times_{\alpha\otimes \beta}^{\piso} P\big)
\stackrel{\overline{\pi}}{\longrightarrow} B(H)$$
gives a (right) Nica partial-isometric representation $W:P\rightarrow B(H)$. Now, by applying the covariance equations
of the pair $(i_{(A\otimes B)},i_{P})$ and (\ref{eq26}), one can see that the pair $(\rho,W)$ is a covariant
partial-isometric representation of $(A/I\otimes B/J,P,\tilde{\alpha}\otimes \tilde{\beta})$ on $H$. The corresponding
representation $\rho\times W$ lifts to $\pi$, which means that
$$(\rho\times W)\circ (q\times P)=\pi,$$
and therefore, we have
$$\ker (q\times P)\subset
\ker \pi=(A\otimes J)\times_{\alpha\otimes \beta}^{\piso} P+(I\otimes B)\times_{\alpha\otimes \beta}^{\piso} P.$$
Thus, the equation (\ref{ker-q.P}) holds.

\end{proof}

\begin{theorem}
\label{posit-series}
Let $(A,P,\alpha)$ and $(B,P,\beta)$ be dynamical systems, and $I$ an extendible $\alpha_{x}$-invariant ideal
of $A$ and $J$ an extendible $\beta_{x}$-invariant ideal of $B$ for every $x\in P$. Assume that
$\tilde{\alpha}:P\rightarrow \End (A/I)$ and $\tilde{\beta}:P\rightarrow \End (B/J)$ are the actions induced by $\alpha$ and
$\beta$, respectively. Then, there is a composition series
$$0\leq \I_{1}\leq \I_{2} \leq (A\otimes B)\times_{\alpha\otimes \beta}^{\piso} P$$
of ideals, such that:
\begin{itemize}
\item[(i)] the ideal $\I_{1}$ is (isomorphic to) $(I\otimes J)\times_{\alpha\otimes \beta}^{\piso} P$;
\item[(ii)] $\I_{2}/\I_{1}\simeq (A/I\otimes J)\times_{\tilde{\alpha}\otimes \beta}^{\piso} P\oplus
(I\otimes B/J)\times_{\alpha\otimes \tilde{\beta}}^{\piso} P$ ;
\item[(iii)] the surjection $q\times P$ induces an isomorphism of
$\big((A\otimes B)\times_{\alpha\otimes \beta}^{\piso} P\big)/\I_{2}$ onto
$(A/I\otimes B/J)\times_{\tilde{\alpha}\otimes \tilde{\beta}}^{\piso} P$.
\end{itemize}

\end{theorem}

\begin{proof}
For (i), as we mentioned in Remark \ref{rmk-3}, $I\otimes J$ is an extendible
$(\alpha\otimes\beta)_{x}$-invariant ideal of $A\otimes B$ for all $x\in P$. Therefore, by Theorem \ref{piso-ext-seq},
the crossed product $(I\otimes J)\times_{\alpha\otimes \beta}^{\piso} P$ sits in the algebra
$(A\otimes B)\times_{\alpha\otimes \beta}^{\piso} P$ as an ideal, which we denote it by $\I_{1}$.

To get (ii), we first define
$$\I_{2}:=(A\otimes J)\times_{\alpha\otimes \beta}^{\piso} P+ (I\otimes B)\times_{\alpha\otimes \beta}^{\piso} P,$$
which is an ideal of $(A\otimes B)\times_{\alpha\otimes \beta}^{\piso} P$ as each summand is. Note that we have
$$[(A\otimes J)\times_{\alpha\otimes \beta}^{\piso} P] \cap [(I\otimes B)\times_{\alpha\otimes \beta}^{\piso} P]
=(I\otimes J)\times_{\alpha\otimes \beta}^{\piso} P.$$
So, it follows that (see diagram (\ref{L-comm-diag}))
\begin{eqnarray*}
\begin{array}{rcl}
\I_{2}/\I_{1}&=&\big[(A\otimes J)\times_{\alpha\otimes \beta}^{\piso} P+ (I\otimes B)\times_{\alpha\otimes \beta}^{\piso} P\big]
/(I\otimes J)\times_{\alpha\otimes \beta}^{\piso} P\\
&=&\big[(A\otimes J)\times_{\alpha\otimes \beta}^{\piso} P\big]/[(I\otimes J)\times_{\alpha\otimes \beta}^{\piso} P]
\oplus \big[(I\otimes B)\times_{\alpha\otimes \beta}^{\piso} P\big]/[(I\otimes J)\times_{\alpha\otimes \beta}^{\piso} P]\\
&\simeq&[(A\otimes J)/(I\otimes J)]\times_{\widetilde{\alpha\otimes \beta}}^{\piso} P \oplus
[(I\otimes B)/(I\otimes J)]\times_{\widetilde{\alpha\otimes \beta}}^{\piso} P\\
&\simeq& (A/I\otimes J)\times_{\tilde{\alpha}\otimes \beta}^{\piso} P \oplus
(I\otimes B/J)\times_{\alpha\otimes \tilde{\beta}}^{\piso} P.
\end{array}
\end{eqnarray*}
Finally, for (iii), we recall from Proposition (\ref{large-diag.}) that we have a surjective homomorphism
$$q\times P: (A\otimes B)\times_{\alpha\otimes \beta}^{\piso} P \rightarrow
(A/I\otimes B/J)\times_{\tilde{\alpha}\otimes \tilde{\beta}}^{\piso} P $$
with
$$\ker (q\times P)=(A\otimes J)\times_{\alpha\otimes \beta}^{\piso} P+(I\otimes B)\times_{\alpha\otimes \beta}^{\piso} P=\I_{2}.$$
Therefore, we have
\begin{eqnarray*}
\begin{array}{rcl}
\big((A\otimes B)\times_{\alpha\otimes \beta}^{\piso} P\big)/\I_{2}&=&\big((A\otimes B)\times_{\alpha\otimes \beta}^{\piso} P\big)/\ker (q\times P)\\
&\simeq& (A/I\otimes B/J)\times_{\tilde{\alpha}\otimes \tilde{\beta}}^{\piso} P.
\end{array}
\end{eqnarray*}

\end{proof}

\section{The application}
\label{sec:apply}
In this section, as an application, we consider the dynamical systems $(C^{*}(G_{p,q}),\N^{2},\beta)$ studied in \cite{LPR}, where $\N^{2}$
is the positive cone of the abelian lattice-ordered group $\Z^{2}$. Let $p$ and $q$ be distinct odd primes, and consider the subgroup
$$G_{p,q}:=\{n p^{-k} q^{-l}: n,k,l \in \Z\}/\Z$$
of $\Q/\Z$. There is an averaging type action $\beta$ of $\N^{2}$ on the group $C^*$-algebra $C^{*}(G_{p,q})$ by endomorphisms, such that
on the canonical generating unitaries $\{u_{r}: r\in G_{p,q}\}$ of $C^{*}(G_{p,q})$ we have
\begin{align}
\label{beta}
\beta_{(m,n)}(u_{r})=\frac{1}{p^{m} q^{n}} \sum_{\{s\in G_{p,q}: p^{m}q^{n}s=r\}} u_{s}
\end{align}
for all $(m,n)\in \N^{2}$. Let $\Z_{p}$ be the compact topological ring of $p$-adic integers (similarly for $\Z_{q}$).
See in \cite[Lemma 1.1]{LPR} that by the Fourier transform $C^{*}(G_{p,q}) \simeq C(\Z_{p}\times\Z_{q})$, the action $\beta$ corresponds
to the action $\alpha$ of $\N^{2}$ on $C(\Z_{p}\times\Z_{q})$ by endomorphisms, such that
\begin{align}\label{alpha}
\alpha_{(m,n)}(f)(x,y)=
   \begin{cases}
      f(p^{-m}q^{-n}x, p^{-m}q^{-n}y) &\textrm{if}\empty\ \text{$x\in p^{m}q^{n}\Z_{p}$ and $y\in p^{m}q^{n}\Z_{q}$,}\\
      0 &\textrm{otherwise}\\
   \end{cases}
\end{align}
for all $(m,n)\in \N^{2}$ and $f\in C(\Z_{p}\times\Z_{q})$. Therefore, to study the partial-isometric crossed product
$C^{*}(G_{p,q})\times_{\beta}^{\piso} \N^{2}$ of the system $(C^{*}(G_{p,q}),\N^{2},\beta)$, it is enough to study the crossed product
$C(\Z_{p}\times\Z_{q})\times_{\alpha}^{\piso} \N^{2}$ of the corresponding system $(C(\Z_{p}\times\Z_{q}),\N^{2},\alpha)$.
Firstly, we have $C(\Z_{p}\times\Z_{q})\simeq C(\Z_{p})\otimes C(\Z_{q})$, and recall that the
action $\alpha$ decomposes as the tensor product
$\gamma\otimes \delta:\N^{2}\rightarrow \End\big(C(\Z_{p})\otimes C(\Z_{q})\big)$ of two actions of $\N^{2}$, such that
\begin{align}\label{gamma}
\gamma_{(m,n)}(g)(x)=
   \begin{cases}
      g(p^{-m}q^{-n}x) &\textrm{if}\empty\ \text{$x\in p^{m}q^{n}\Z_{p}$,}\\
      0 &\textrm{otherwise}\\
   \end{cases}
\end{align}
for all $(m,n)\in \N^{2}$ and $g\in C(\Z_{p})$ (similarly for $\delta:\N^{2}\rightarrow \End\big(C(\Z_{q})\big)$). It thus follows that
$C(\Z_{p}\times\Z_{q})\times_{\alpha}^{\piso} \N^{2}\simeq \big(C(\Z_{p})\otimes C(\Z_{q})\big)\times_{\gamma\otimes \delta}^{\piso} \N^{2}$,
for which we want to apply Theorem \ref{posit-series}. To do so, consider the extendible $\gamma$-invariant ideal
$I:=C_{0}(\Z_{p}\backslash\{0\})$ of $C(\Z_{p})$ and the extendible $\delta$-invariant ideal $J:=C_{0}(\Z_{q}\backslash\{0\})$ of
$C(\Z_{q})$ as in \cite{LPR}. Now, by applying Theorem \ref{posit-series} to the systems $(C(\Z_{p}),\N^{2},\gamma)$ and
$(C(\Z_{q}),\N^{2},\delta)$ along with the ideals $I$ and $J$, we get the following theorem which is the partial-isometric version of
\cite[Theorem 2.2]{LPR}:
\begin{theorem}
\label{piso-LPR}
There are ideals $\I_{1}$ and $\I_{2}$ in
$$C(\Z_{p}\times\Z_{q})\times_{\alpha}^{\piso} \N^{2}\simeq \big(C(\Z_{p})\otimes C(\Z_{q})\big)\times_{\gamma\otimes \delta}^{\piso} \N^{2}$$
which form the composition series
\begin{align}
\label{compose-2}
0\leq \I_{1}\leq \I_{2} \leq C(\Z_{p}\times\Z_{q})\times_{\alpha}^{\piso} \N^{2}
\end{align}
of ideals, such that:
\begin{itemize}
\item[(a)] $\I_{1}\simeq \A\otimes_{\max} \A \otimes C(\U(\Zp)\times \U(\Zq))$,
\item[(b)] $\I_{2}/\I_{1}\simeq \big(\A\otimes_{\max} \big[C(\U(\Zp))\times_{\sigma^{p,q}}^{\piso} \N\big]\big) \oplus
\big(\A\otimes_{\max} \big[C(\U(\Zq))\times_{\sigma^{q,p}}^{\piso} \N\big]\big)$, and
\item[(c)] $\big(C(\Z_{p}\times\Z_{q})\times_{\alpha}^{\piso} \N^{2}\big)/\I_{2}\simeq \T(\Z^{2})\simeq \T(\Z)\otimes \T(\Z)$,
\end{itemize}
where $\U(\Zp)$ is the group of the multiplicatively invertible elements in $\Zp$ (similarly for $\U(\Zq)$), $\sigma^{p,q}$
is the action of $\N$ on the algebra $C(\U(\Zp))$ by automorphisms such that
$\sigma^{p,q}_{n}(f)(x)=f(q^{-n}x)$ (similarly for $\sigma^{q,p}$), and the algebra $\A$ is a full corner in
the algebra $\K(\ell^{2}(\N)\otimes \c)$ of compact operators, in which $\c=B_{\N}=\clsp\{1_{n}: n\in \N\}$.
\end{theorem}

\begin{proof}
We apply Theorem \ref{posit-series} to the systems $(C(\Z_{p}),\N^{2},\gamma)$ and
$(C(\Z_{q}),\N^{2},\delta)$ together with the ideals $I=C_{0}(\Z_{p}\backslash\{0\})$ and $J=C_{0}(\Z_{q}\backslash\{0\})$. Therefore, we have
$$\I_{1}:=\big[C_{0}(\Z_{p}\backslash\{0\})\otimes C_{0}(\Z_{q}\backslash\{0\})\big]\times_{\gamma\otimes \delta}^{\piso} \N^{2},$$ and
$$\I_{2}:=\big[C_{0}(\Z_{p}\backslash\{0\})\otimes C(\Z_{q})\big]\times_{\gamma\otimes \delta}^{\piso} \N^{2}+
\big[C(\Z_{p})\otimes C_{0}(\Z_{q}\backslash\{0\})\big]\times_{\gamma\otimes \delta}^{\piso} \N^{2},$$ from which we obtain the
composition series (\ref{compose-2}) of ideals.

Next, to identify the subquotients with familiar terms, we start with (c). First, by
(iii) in Theorem \ref{posit-series}, we have
$$\big(C(\Z_{p}\times\Z_{q})\times_{\alpha}^{\piso} \N^{2}\big)/\I_{2}\simeq
\big(\big[C(\Z_{p})/C_{0}(\Z_{p}\backslash\{0\})\big]\otimes \big[C(\Z_{q})/C_{0}(\Z_{q}\backslash\{0\})\big]\big)
\times_{\tilde{\gamma}\otimes\tilde{\delta}}^{\piso} \N^{2}.$$
Then, as
$$C(\Z_{p})/C_{0}(\Z_{p}\backslash\{0\})\simeq \C \simeq C(\Z_{q})/C_{0}(\Z_{q}\backslash\{0\})\ \ (\textrm{see \cite{LPR}}),$$
we get
\begin{eqnarray*}
\begin{array}{rcl}
\big(C(\Z_{p}\times\Z_{q})\times_{\alpha}^{\piso} \N^{2}\big)/\I_{2}&\simeq&
(\C \otimes \C)\times_{\id\otimes \id}^{\piso} \N^{2}\\
&\simeq& \C \times_{\id}^{\piso} \N^{2}\simeq \T(\Z^{2}),
\end{array}
\end{eqnarray*}
where the bottom lines follows from Example \ref{C*(P)} as $(\Z^{2},\N^{2})$ is abelian (see also \cite[Remark 5.4]{SZ2}).
Also, by applying Theorem \ref{P*S-tensor},
\begin{eqnarray*}
\begin{array}{rcl}
\T(\Z^{2})\simeq \C \times_{\id}^{\piso} \N^{2}&\simeq& (\C \otimes \C)\times_{\id\otimes \id}^{\piso} \N^{2}\\
&\simeq& (\C \times_{\id}^{\piso} \N)\otimes (\C \times_{\id}^{\piso} \N)\simeq \T(\Z)\otimes \T(\Z),
\end{array}
\end{eqnarray*}
where $\C \times_{\id}^{\piso} \N\simeq \T(\Z)$ is known by \cite[Example 4.3]{AZ} (see also Example \ref{C*(P)} or \cite[Remark 5.4]{SZ2}).
So, we are done with (c).

To get (a), first, by \cite[Corollary 2.4]{LPR}, there is an isomorphism
$$C_{0}(\Z_{p}\backslash\{0\})\simeq \c_{0}\otimes C(\U(\Zp)),$$
where $\c_{0}=\clsp\{1_{n}-1_{m}:n,m\in \N\ \textrm{with}\ n<m\}=C_{0}(\N)$. By this isomorphism, the action
$\gamma$ corresponds to the tensor product action $\tau\otimes \sigma^{p,q}$, where $\tau$ is the action
of $\N$ on $\c_{0}$ by forward shifts (similarly for $C_{0}(\Z_{q}\backslash\{0\})$ and the action $\delta$).
Therefore, we have an isomorphism (see the ideal $\I_{1}$)
\begin{eqnarray*}
\begin{array}{rcl}
C_{0}(\Z_{p}\backslash\{0\})\otimes C_{0}(\Z_{q}\backslash\{0\})&\simeq& \c_{0}\otimes\c_{0}\otimes C(\U(\Zp))\otimes C(\U(\Zq))\\
&\simeq& C_{0}(\N\times \N) \otimes C(\U(\Zp)\times \U(\Zq))\\
&\simeq& C_{0}(\N\times \N \times \U(\Zp)\times \U(\Zq))
\end{array}
\end{eqnarray*}
which takes each endomorphism $(\gamma\otimes \delta)_{(m,n)}$ to
$\tau_{m}\otimes\tau_{n}\otimes \sigma^{p,q}_{n}\otimes \sigma^{q,p}_{m}$. So, it follows that
$$\I_{1}\simeq\big[C_{0}(\N\times \N) \otimes C(\U(\Zp)\times \U(\Zq))\big]
\times_{(\tau \otimes \tau) \otimes (\sigma^{p,q}\otimes \sigma^{q,p})}^{\piso} \N^{2}.$$
Moreover, there is an automorphism $\Phi$ of $C_{0}(\N\times \N) \otimes C(\U(\Zp)\times \U(\Zq))$ such that we have
$$\Phi\circ (\tau_{m}\otimes\tau_{n}\otimes \sigma^{p,q}_{n}\otimes \sigma^{q,p}_{m})
=\tau_{m}\otimes\tau_{n}\otimes \id\otimes \id\ \ (\textrm{see again \cite{LPR}}).$$
The automorphism $\Phi$ then induces the isomorphism
$$\I_{1}\simeq\big[C_{0}(\N\times \N) \otimes C(\U(\Zp)\times \U(\Zq))\big]
\times_{(\tau \otimes \tau) \otimes \id}^{\piso} \N^{2}.$$
Next, we need to show that
$$\big[C_{0}(\N\times \N) \otimes C(\U(\Zp)\times \U(\Zq))\big]\times_{(\tau \otimes \tau) \otimes \id}^{\piso} \N^{2}
\simeq \big[C_{0}(\N\times \N)\times_{\tau \otimes \tau}^{\piso} \N^{2}\big] \otimes C(\U(\Zp)\times \U(\Zq)).$$
We skip the proof as it is routine and refer readers to Remark \ref{rmk-4} for an indication on the proof.
Also, by applying Theorem \ref{P*S-tensor},
$$C_{0}(\N\times \N)\times_{\tau \otimes \tau}^{\piso} \N^{2}
\simeq (\c_{0}\otimes \c_{0})\times_{\tau \otimes \tau}^{\piso} \N^{2}
\simeq (\c_{0}\times_{\tau}^{\piso} \N) \otimes_{\max} (\c_{0}\times_{\tau}^{\piso} \N),$$
and hence, we get
$$\I_{1}\simeq (\c_{0}\times_{\tau}^{\piso} \N) \otimes_{\max} (\c_{0}\times_{\tau}^{\piso} \N) \otimes C(\U(\Zp)\times \U(\Zq)).$$
Finally, see in \cite[Example 4.3]{AZ} that the algebra $\c_{0}\times_{\tau}^{\piso} \N$ is a full corner
in the algebra $\K(\ell^{2}(\N)\otimes \c)$ of compact operators. More precisely, let $P$ be the projection
in $\M(\K(\ell^{2}(\N)\otimes \c))\simeq \L(\ell^{2}(\N)\otimes \c)$ defined by
$$P(\xi)(n)=\tau_{n}(1)\xi(n)=1_{n}\xi(n)\ \ \textrm{for all}\ \xi\in \ell^{2}(\N)\otimes \c,$$
where $\tau$ is the action of $\N$ on the algebra $\c$ by forward shifts. Then, $\c_{0}\times_{\tau}^{\piso} \N$ is
isomorphic to the full corner $P\K(\ell^{2}(\N)\otimes \c)P$, which we denote it by $\A$. Thus, we have
$$\I_{1}\simeq \A\otimes_{\max} \A \otimes C(\U(\Zp)\times \U(\Zq)).$$

At last, to see (b), we first apply (ii) in Theorem \ref{posit-series} to get
$$\I_{2}/\I_{1}\simeq \big[C_{0}(\Z_{p}\backslash\{0\})\otimes C(\Z_{q})/J\big]\times_{\gamma\otimes \tilde{\delta}}^{\piso} \N^{2}\oplus
\big[C(\Z_{p})/I\otimes C_{0}(\Z_{q}\backslash\{0\})\big]\times_{\tilde{\gamma}\otimes \delta}^{\piso} \N^{2},$$
and since $C(\Z_{q})/J\simeq \C \simeq C(\Z_{p})/I$, we have
\begin{eqnarray*}
\begin{array}{rcl}
\I_{2}/\I_{1}&\simeq& \big[C_{0}(\Z_{p}\backslash\{0\})\otimes \C\big]\times_{\gamma\otimes \id}^{\piso} \N^{2}\oplus
\big[\C\otimes C_{0}(\Z_{q}\backslash\{0\})\big]\times_{\id\otimes \delta}^{\piso} \N^{2}\\
&\simeq& \big(C_{0}(\Z_{p}\backslash\{0\})\times_{\gamma}^{\piso} \N^{2}\big) \oplus
\big(C_{0}(\Z_{q}\backslash\{0\})\times_{\delta}^{\piso} \N^{2}\big).\\
\end{array}
\end{eqnarray*}
Then, by applying \cite[Corollary 2.4]{LPR} and Theorem \ref{P*S-tensor} (see also the proof of (a)), for the crossed product
$\big(C_{0}(\Z_{p}\backslash\{0\})\times_{\gamma}^{\piso} \N^{2}\big)$, we have
\begin{eqnarray*}
\begin{array}{rcl}
C_{0}(\Z_{p}\backslash\{0\})\times_{\gamma}^{\piso} \N^{2}&\simeq&
\big(\c_{0}\otimes C(\U(\Z_{p}))\big)\times_{\tau\otimes \sigma^{p,q}}^{\piso} (\N\times\N)\\
&\simeq& \big[\c_{0}\times_{\tau}^{\piso} \N\big] \otimes_{\max} \big[C(\U(\Zp))\times_{\sigma^{p,q}}^{\piso} \N\big]\\
&\simeq& \A \otimes_{\max} \big[C(\U(\Zp))\times_{\sigma^{p,q}}^{\piso} \N\big].
\end{array}
\end{eqnarray*}
Similarly,
$$C_{0}(\Z_{q}\backslash\{0\})\times_{\delta}^{\piso} \N^{2}\simeq \A \otimes_{\max} \big[C(\U(\Zq))\times_{\sigma^{q,p}}^{\piso} \N\big].$$
Therefore, it follows that
$$\I_{2}/\I_{1}\simeq \big(\A\otimes_{\max} \big[C(\U(\Zp))\times_{\sigma^{p,q}}^{\piso} \N\big]\big) \oplus
\big(\A\otimes_{\max} \big[C(\U(\Zq))\times_{\sigma^{q,p}}^{\piso} \N\big]\big).$$
This completes the proof.
\end{proof}

\begin{remark}
\label{BD-algebra}
Recall that, if $m$ and $n$ are relatively prime integers, then $m\in \U(\Z/n\Z)$ (similarly $n\in \U(\Z/m\Z)$).
Let $\order_{n}(m)$ denote the order of $m$ in $\U(\Z/n\Z)$. Since $p$ and $q$ in Theorem \ref{piso-LPR} are distinct odd primes, it
follows by \cite[Theorem 3.1]{LPR} that there is a positive integer $L=L_{p}(q)$ such that
\begin{align}
\label{order}
\order_{p^{\ell}}(q)=
   \begin{cases}
      \order_{p}(q) &\textrm{if}\empty\ \text{$1\leq \ell\leq L$,}\\
      p^{\ell-L} \order_{p}(q) &\textrm{if}\empty\ \text{$\ell>L$.}\\
   \end{cases}
\end{align}
Now, since the action $\sigma^{p,q}$ of $\N$ on the algebra $C(\U(\Zp))$ is given by automorphisms, it follows by \cite[Theorem 4.1]{AZ} that
the algebra of compact operators $\K(\ell^{2}(\N))\otimes C(\U(\Zp))$ sits in $C(\U(\Zp))\times_{\sigma^{p,q}}^{\piso} \N$ as an ideal,
such that
$$\big[C(\U(\Zp))\times_{\sigma^{p,q}}^{\piso} \N\big]\big/\big[\K(\ell^{2}(\N))\otimes C(\U(\Zp))\big]
\simeq C(\U(\Zp))\times_{\sigma^{p,q}}^{\iso} \N \simeq C(\U(\Zp))\times_{\sigma^{p,q}} \Z.$$
But, again by \cite[Theorem 3.1]{LPR}, the classical crossed product $C(\U(\Zp))\times_{\sigma^{p,q}} \Z$ is the direct sum of
$p^{L_{p}(q)-1}(p-1)/\order_{p}(q)$ Bunce-Deddens algebras with supernatural number $\order_{p}(q)p^{\infty}$. Similarly,
$$\big[C(\U(\Zq))\times_{\sigma^{q,p}}^{\piso} \N\big]\big/\big[\K(\ell^{2}(\N))\otimes C(\U(\Zq))\big]
\simeq C(\U(\Zq))\times_{\sigma^{q,p}}^{\iso} \N \simeq C(\U(\Zq))\times_{\sigma^{q,p}} \Z,$$
where $C(\U(\Zq))\times_{\sigma^{q,p}} \Z$ is the direct sum of
$q^{L_{q}(p)-1}(q-1)/\order_{q}(p)$ Bunce-Deddens algebras with supernatural number $\order_{q}(p)q^{\infty}$. Readers are
referred to \cite{B-D, David, Exel} on Bunce-Deddens algebras. However, we will provide a quick recall on these algebras shortly later.
\end{remark}

Next, we would like to analyze the crossed product $C(\U(\Zp))\times_{\sigma^{p,q}}^{\piso} \N$ in Theorem \ref{piso-LPR} more.
Recall from \cite{LPR} that if $\Gamma$ is the closed subgroup of $\U(\Zp)$ generated by $q$, then it is invariant under multiplication
by powers of $q$. It thus follows that the ideal $C(\Gamma)$ of $C(\U(\Zp))$ is $\sigma^{p,q}$-invariant. Note that the
same facts hold for each closed subset $x\Gamma$ of $\U(\Zp)$ and each ideal $C(x\Gamma)$ of $C(\U(\Zp))$, where $x\in \U(\Zp)$.
Moreover, since $\U(\Zp)$ is the disjoint union of $N:=p^{L_{p}(q)-1}(p-1)/\order_{p}(q)$ closed invariant subsets of the form
$x\Gamma$, $C(\U(\Zp))$ is the direct sum of $N$ $\sigma^{p,q}$-invariant ideals of the form $C(x\Gamma)$. Now, since
each ideal $C(x\Gamma)$, as well as the algebra $C(\U(\Zp))$, is unital and the action $\sigma^{p,q}$ of $\N$ is given
by automorphisms, it is not difficult to see that each algebra $C(x\Gamma)$ is actually an extendible $\sigma^{p,q}$-invariant
ideal of $C(\U(\Zp))$. Therefore, each crossed product $C(x\Gamma)\times_{\sigma^{p,q}}^{\piso} \N$ sits in
$C(\U(\Zp))\times_{\sigma^{p,q}}^{\piso} \N$ as an ideal by Theorem \ref{piso-ext-seq} or \cite[Theorem 3.1]{AZ2}. Also,
calculation shows that
\begin{align}
\label{eq32}
C(\U(\Zp))\times_{\sigma^{p,q}}^{\piso} \N
\simeq\big(C(x_{1}\Gamma)\times_{\sigma^{p,q}}^{\piso} \N\big)\oplus\big(C(x_{2}\Gamma)\times_{\sigma^{p,q}}^{\piso} \N\big)\oplus
\cdot\cdot\cdot\oplus\big(C(x_{N}\Gamma)\times_{\sigma^{p,q}}^{\piso} \N\big),
\end{align}
where one may take $x_{1}$ to be $1$, the unit element of the group $\U(\Zp)$. Now, since for every $x\in\U(\Zp)$, the closed subsets $\Gamma$ and
$x\Gamma$ of $\U(\Zp)$ are homeomorphic, there is an isomorphism $\psi:C(\Gamma)\rightarrow C(x\Gamma)$ of $C^{*}$-algebras such that
$\sigma^{p,q}_{n}\circ\psi=\psi\circ\sigma^{p,q}_{n}$ for all $n\in\N$. The isomorphism $\psi$ then induces an isomorphism between
crossed products $C(\Gamma)\times_{\sigma^{p,q}}^{\piso} \N$ and $C(x\Gamma)\times_{\sigma^{p,q}}^{\piso} \N$ (see Remark \ref{rmk-2}).
So, this fact and (\ref{eq32}) imply that the crossed product $C(\U(\Zp))\times_{\sigma^{p,q}}^{\piso} \N$ is actually isomorphic to
the direct sum of $N$ ideals $C(\Gamma)\times_{\sigma^{p,q}}^{\piso} \N$. At last, we want to have a familiar description for
the crossed product $C(\Gamma)\times_{\sigma^{p,q}}^{\piso} \N$. To do so, we need to recall on Bunce-Deddens algebras quickly.
These algebras were first defined in \cite{B-D} as the $C^{*}$-algebras related to certain weighted shift operators on the Hilbert space
$\ell^{2}(\N)$. They were then identified as the classical crossed products by (automorphic) actions induced by
the odometer map (see \cite{Exel, LPR}). Suppose that $\{n_{i}\}_{i=0}^{\infty}$ is a strictly increasing sequence of positive integers,
such that $n_{i}$ divides $n_{i+1}$ for all $i$. Note that one can assume that $n_{0}=1$ without loss of generality. For every $i\geq 0$,
let $m_{i}=n_{i+1}/n_{i}$, and then consider the Cantor set $\mathbf{K}$ given by the model
$$\mathbf{K}=\prod_{i=0}^{\infty} \{0,1,..., m_{i}-1\}.$$
The odometer map on $\mathbf{K}$, namely $\mathcal{O}:\mathbf{K}\rightarrow \mathbf{K}$, is given by addition of $(1,0,0,...)$ with carry over
to the right. For example, $\mathcal{O}(m_{0}-1,m_{1}-1,0,0,0,...)=(0,0,1,0,0,...)$. So, it induces an action
$\tau$ of $\Z$ on the algebra $C(\mathbf{K})$ by automorphisms, such that the classical crossed product
$C(\mathbf{K})\times_{\tau} \Z$ is a Bunce-Deddens algebra with supernatural number $\Pi_{i\geq 0}\ m_{i}$. Now, in particular,
for the sequence
$$\{n_{i}\}_{i=0}^{\infty}=\{1, d, dp, dp^{2}, dp^{3},...\},$$
the Cantor set $\mathbf{K}$ is
$$\mathbf{K}=\{0,1,...,d-1\}\times \prod_{i=1}^{\infty} \{0,1,..., p-1\},$$
where $d=\order_{p}(q)$. Then, there is a homeomorphism of $\mathbf{K}$ onto the closed subgroup $\Gamma$, which induces an
isomorphism $\varphi$ of $C(\mathbf{K})$ onto $C(\Gamma)$ such that $\varphi\circ\tau_{n}=\sigma^{p,q}_{n}\circ \varphi$
for all $n\in \N$ (see \cite{LPR}).
Therefore, each ideal $C(\Gamma)\times_{\sigma^{p,q}}^{\piso} \N$ is actually isomorphic to the partial-isometric crossed product
$C(\mathbf{K})\times_{\tau}^{\piso} \N$ (see Remark \ref{rmk-2}). Consequently, the algebra $C(\U(\Zp))\times_{\sigma^{p,q}}^{\piso} \N$
is in fact isomorphic to the direct sum of $N$ crossed products $C(\mathbf{K})\times_{\tau}^{\piso} \N$. Note that, since the action
$\tau$ is automorphic, by \cite[Theorem 4.1]{AZ}, $C(\mathbf{K})\times_{\tau}^{\piso} \N$ contains the algebra of compact operators
$\K(\ell^{2}(\N))\otimes C(\mathbf{K})$ as an ideal, such that the quotient algebra
$$\big[C(\mathbf{K})\times_{\tau}^{\piso} \N\big]\big/\big[\K(\ell^{2}(\N))\otimes C(\mathbf{K})\big]
\simeq C(\mathbf{K})\times_{\tau}^{\iso} \N \simeq C(\mathbf{K})\times_{\tau} \Z$$
is a Bunce-Deddens algebra with supernatural number $\order_{p}(q)p^{\infty}$. By swapping the roles of $p$ and $q$, similarly,
the algebra $C(\U(\Zq))\times_{\sigma^{q,p}}^{\piso} \N$ is indeed isomorphic to the direct sum of $q^{L_{q}(p)-1}(q-1)/\order_{q}(p)$
crossed products $C(\widetilde{\mathbf{K}})\times_{\tilde{\tau}}^{\piso} \N$, where
$$\widetilde{\mathbf{K}}=\{0,1,...,\order_{q}(p)-1\}\times \prod_{i=1}^{\infty} \{0,1,..., q-1\},$$
and $\tilde{\tau}$ is the action of $\Z$ on $C(\widetilde{\mathbf{K}})$ by automorphisms induced by the odometer map on
$\widetilde{\mathbf{K}}$. Also, the quotient algebra (again by \cite[Theorem 4.1]{AZ})
$$\big[C(\widetilde{\mathbf{K}})\times_{\tilde{\tau}}^{\piso} \N\big]\big/\big[\K(\ell^{2}(\N))\otimes C(\widetilde{\mathbf{K}})\big]
\simeq C(\widetilde{\mathbf{K}})\times_{\tilde{\tau}}^{\iso} \N \simeq C(\widetilde{\mathbf{K}})\times_{\tilde{\tau}} \Z$$
is a Bunce-Deddens algebra with supernatural number $\order_{q}(p)q^{\infty}$. Therefore, we have actually proved the following
corollary as a refinement of Theorem \ref{piso-LPR}:
\begin{cor}
\label{piso-LPR-2}
There are ideals $\I_{1}$ and $\I_{2}$ in
$$C(\Z_{p}\times\Z_{q})\times_{\alpha}^{\piso} \N^{2}\simeq \big(C(\Z_{p})\otimes C(\Z_{q})\big)\times_{\gamma\otimes \delta}^{\piso} \N^{2}$$
which form the composition series
\begin{align}
\label{compose-3}
0\leq \I_{1}\leq \I_{2} \leq C(\Z_{p}\times\Z_{q})\times_{\alpha}^{\piso} \N^{2}
\end{align}
of ideals, such that:
\begin{itemize}
\item[(a)] $\I_{1}\simeq \A\otimes_{\max} \A \otimes C(\U(\Zp)\times \U(\Zq))$,
\item[(b)] $\I_{2}/\I_{1}\simeq (\A\otimes_{\max} \mathcal{C}) \oplus (\A\otimes_{\max} \mathcal{D})$, and
\item[(c)] $\big(C(\Z_{p}\times\Z_{q})\times_{\alpha}^{\piso} \N^{2}\big)/\I_{2}\simeq \T(\Z^{2})\simeq \T(\Z)\otimes \T(\Z)$,
\end{itemize}
where $\U(\Zp)$ is the group of the multiplicatively invertible elements in $\Zp$ (similarly for $\U(\Zq)$), the algebra
$\A$ is a full corner in the algebra $\K(\ell^{2}(\N)\otimes \c)$ of compact operators, the algebra $\mathcal{C}$ is the
direct sum of $p^{L_{p}(q)-1}(p-1)/\order_{p}(q)$ crossed products $C(\mathbf{K})\times_{\tau}^{\piso} \N$, and the algebra
$\mathcal{D}$ is the direct sum of $q^{L_{q}(p)-1}(q-1)/\order_{q}(p)$ crossed products
$C(\widetilde{\mathbf{K}})\times_{\tilde{\tau}}^{\piso} \N$.
\end{cor}

\begin{remark}
\label{rmk-4}
To see that for any $C^{*}$-algebra $A$,
$$(C_{0}(\N\times \N) \otimes A) \times_{(\tau \otimes \tau) \otimes \id}^{\piso} \N^{2}
\simeq \big(C_{0}(\N\times \N)\times_{\tau \otimes \tau}^{\piso} \N^{2}\big) \otimes_{\max} A,$$
let $(j_{C_{0}(\N\times \N)}, j_{\N^{2}})$ be the canonical covariant partial-isometric pair of the system
$(C_{0}(\N\times \N),\N^{2}, \tau \otimes \tau)$ in the algebra $B:=C_{0}(\N\times \N)\times_{\tau \otimes \tau}^{\piso} \N^{2}$.
Suppose that $i_{B}$ and $i_{A}$ are the canonical nondegenerate homomorphisms of the algebras $B$ and $A$ in the multiplier algebra
$\M(B\otimes_{\max} A)$, respectively (see \cite[Theorem B.27]{RW}). Consider the homomorphism given by the composition
$$C_{0}(\N\times \N)\stackrel{j_{C_{0}(\N\times \N)}}{\longrightarrow} B \stackrel{i_{B}}{\longrightarrow} \M(B\otimes_{\max} A).$$
Then, the ranges of $(i_{B}\circ j_{C_{0}(\N\times \N)})$ and $i_{A}$ commute, and therefore, there is a homomorphism
$k_{C_{0}(\N\times \N)\otimes A}:=(i_{B}\circ j_{C_{0}(\N\times \N)})\otimes_{\max} i_{A}$ of $C_{0}(\N\times \N)\otimes A$
in $\M(B\otimes_{\max} A)$ such that
$$k_{C_{0}(\N\times \N)\otimes A}(f\otimes a)=i_{B}(j_{C_{0}(\N\times \N)}(f)) i_{A}(a)=j_{C_{0}(\N\times \N)}(f)\otimes a$$
for all $f\in C_{0}(\N\times \N)$ and $a\in A$. One can see that $k_{C_{0}(\N\times \N)\otimes A}$ is indeed nondegenerate.
Next, let $k_{\N^{2}}$ be the map defined by the composition
$$\N^{2}\stackrel{j_{\N^{2}}}{\longrightarrow} \M(B) \stackrel{\overline{i_{B}}}{\longrightarrow} \M(B\otimes_{\max} A).$$
It is not difficult to see that $k_{\N^{2}}$ is a (right) Nica partial-isometric representation. Now, it is routine for one to
check that the triple $(B\otimes_{\max} A, k_{C_{0}(\N\times \N)\otimes A}, k_{\N^{2}})$ is a partial-isometric crossed product for
the system $\big(C_{0}(\N\times \N)\otimes A, \N^{2}, (\tau \otimes \tau) \otimes \id\big)$. Therefore, there is an isomorphism
$$\Upsilon:
\big((C_{0}(\N\times \N) \otimes A)\times_{(\tau \otimes \tau) \otimes \id}^{\piso} \N^{2},i_{C_{0}(\N\times \N)\otimes A}, i_{\N^{2}}\big)
\rightarrow \big(C_{0}(\N\times \N)\times_{\tau \otimes \tau}^{\piso} \N^{2}\big) \otimes_{\max} A$$
such that
\begin{eqnarray*}
\begin{array}{rcl}
\Upsilon\big(i_{C_{0}(\N\times \N)\otimes A}(f\otimes a) i_{\N^{2}}(m,n)\big)&=&k_{C_{0}(\N\times \N)\otimes A}(f\otimes a) k_{\N^{2}}(m,n)\\
&=&i_{B}(j_{C_{0}(\N\times \N)}(f)) i_{A}(a) \overline{i_{B}}(j_{\N^{2}}(m,n))\\
&=&i_{B}(j_{C_{0}(\N\times \N)}(f)) \overline{i_{B}}(j_{\N^{2}}(m,n)) i_{A}(a)\\
&=&i_{B}(j_{C_{0}(\N\times \N)}(f) j_{\N^{2}}(m,n)) i_{A}(a)\\
&=&[j_{C_{0}(\N\times \N)}(f) j_{\N^{2}}(m,n)] \otimes a.
\end{array}
\end{eqnarray*}

\end{remark}

\end{document}